\documentclass[english,11pt]{article}
\usepackage[a4paper, total={7.1in, 8.1in}]{geometry}
\usepackage{amsmath,hyperref}
\usepackage{amssymb}
\usepackage{amsfonts}
\usepackage{fancyhdr}
\newtheorem{theorem}{Theorem}

\newtheorem{assumption}[theorem]{Assumption}

\newtheorem{corollary}[theorem]{Corollary}

\newtheorem{definition}[theorem]{Definition}
\newtheorem{example}{Example}

\newtheorem{lemma}[theorem]{Lemma}

\newtheorem{remark}[theorem]{Remark}

\numberwithin{equation}{section}
\numberwithin{example}{section}
\numberwithin{theorem}{section}

\newenvironment{proof}[1][Proof]{\noindent\textbf{#1.} }{\ \rule{0.5em}{0.5em}}
\makeatother
\begin{document}
	
	\title{Exit Time Risk-Sensitive Control for Systems of Cooperative Agents}
	\author{Paul Dupuis\thanks{%
			Research supported in part by AFOSR FA9550-12-1-0399}, Vaios Laschos\thanks{%
			Research supported in part by AFOSR FA9550-12-1-0399}, and Kavita Ramanan%
		\thanks{%
			Research supported in part by AFOSR FA9550-12-1-0399 and PI NSF DMS-171303}}
	\date{\today}
	\maketitle
	
	\begin{abstract}
		We study a sequence of many-agent exit time
		stochastic control problems, parameterized by the number of agents,
		with risk-sensitive cost structure. We identify
		a fully characterizing assumption, under which  each such  control problem
		corresponds to a risk-neutral stochastic control problem with additive cost,
		and sequentially
		to a risk-neutral stochastic control problem on the simplex  that retains only the
		distribution of states of agents, while 
		discarding further specific information about the state of each agent.   
		Under some additional assumptions, we also prove that the 
		sequence of value functions of these stochastic control problems converges to the value function of a
		deterministic control problem, which can be used for the design of nearly
		optimal controls for the original problem, when the number of agents is
		sufficiently large.
	\end{abstract}
	
	\section{Introduction}
	
	\subsection{Motivation and Background}
	\label{subs-mot}
	
	In this paper, we study many-agent exit time stochastic control problems
	with risk-sensitive cost. Each agent occupies a
	state that takes values in a finite set $\mathcal{X}$, and by controlling
	the {transition } rates{\ between states for }each agent, we try to keep the
	system away from a \textquotedblleft ruin\textquotedblright\ set $\mathcal{K}
	$, for as long as possible and with the least cost. We prove, under suitable
	assumptions, that for every finite number $n$ of agents the control problem
	is equivalent to one with an additive cost structure. Moreover, when $%
	\mathcal{K}\subset \mathcal{X}^{n}$ can be identified with a subset of the
	simplex of probability measures $\mathcal{P}(\mathcal{X}),$ in the sense
	that for every permutation $\sigma$ of $\{1, 2, \ldots, n\}$ we have $\sigma  \mathcal{K}=\mathcal{K}$, 
	then we can replace the original problem by one
	on $\mathcal{P}^{n}(\mathcal{X})=\mathcal{P}(\mathcal{X})\cap \frac{1}{n}%
	\mathbb{Z}^{d},$ getting in this way a control problem whose state is the
	empirical measure on the states of the individual{\ agents}. We also study
	the behavior as $n\rightarrow \infty $ of the sequence of suitable
	renormalized value functions, and prove uniform convergence to the value
	function of a deterministic control problem.
	
	We first describe the model without control, which we call the
	\textquotedblleft base\textquotedblright\ or \textquotedblleft
	nominal\textquotedblright\ model. Let $\mathcal{X}=\{e_{1},\dots ,e_{d}\},$
	where $e_{i}$ is the $i$th unit vector in $\mathbb{R}^{d}$. Let also $%
	\boldsymbol{\gamma }=\{\gamma _{xy}\}_{(x,y)\in \mathcal{X}\times \mathcal{X}%
	}$ denote the rates of an ergodic Markov jump process on $\mathcal{X}.$ This
	process has the generator%
	\begin{equation}
	\mathcal{L}_{\boldsymbol{\boldsymbol{\gamma }}}[f](x)=\sum_{y\in \mathcal{X}%
	}\gamma _{xy}\left[ f(y)-f(x)\right],  \label{one agent process}
	\end{equation}%
	for functions $f: \mathcal{X} \mapsto \mathbb{R}$. 
	For $n\in \mathbb{N},$ consider $n$ agents that move independently and stochastically, 
	with each taking values in  ${\mathcal X}=\{e_{1},\dots ,e_{d}\}$. 
	Then the dynamics of these agents can be represented by a stochastic process 
	taking values in $\mathcal{X}^n$.    Let $\boldsymbol{x}^{n}=(x_{1}^{n},\ldots
	,x_{n}^{n})$ denote a generic element of $\mathcal{X}^n$. Also, for 
	$x, y \in {\mathcal X}$, $i \in \{1, \ldots, n\}$, let 
	$\boldsymbol{v}_{xy}:= y-x$
	and let $\boldsymbol{v}%
	_{i,xy}^{n}=(0,\dots ,0,\boldsymbol{v}_{xy},0,\dots ,0)$ be a $d\times n$
	matrix with all columns equal to zero apart from the $i$th column, which is equal to
	the vector $\boldsymbol{v}_{xy}$.
	Also, let  $\mathcal{Z}:= \{(x,y)\in \mathcal{X}\times \mathcal{X}:\gamma_{xy}>0\}$, 
	and define $\mathcal{Z}_{x}:= \{y\in \mathcal{X}:(x,y)\in \mathcal{Z}\}$ 
	to be  the set of allowed transitions from $x$. 
	Then the generator of the state process of the base model takes the form 
	\begin{equation}
	\mathcal{L}_{\boldsymbol{\gamma }}^{n}[f](\boldsymbol{x}^{n})=\sum_{i=1}^{n}%
	\sum_{y\in \mathcal{Z}_{x_{i}^{n}}}\gamma _{x_{i}^{n}y}\left[ f(\boldsymbol{x%
	}^{n}+\boldsymbol{v}_{i,x_{i}^{n}y}^{n})-f(\boldsymbol{x}^{n})\right], 
	\label{many agent process}
	\end{equation}%
	for $f: {\mathcal X}^n \mapsto \mathbb{R}$. Note that  the  span of $\mathcal{Z}$ is  the hyperplane 
	\begin{equation}
	\mathcal{H}:= \left\{ \sum_{(x,y)\in \mathcal{Z}}a_{xy}\boldsymbol{v}%
	_{xy}:a_{xy}>0,\,(x,y)\in \mathcal{Z}\right\} ,  \label{Hdef}
	\end{equation}%
	which, since $\gamma $ is ergodic, coincides with the hyperplane through the
	origin that is parallel to $\mathcal{P}(\mathcal{X})$.
	We claim that the set $\mathcal{H}$ does not change if the $a_{xy}$ are
	allowed to be arbitrary real numbers.  To see why this is true, note that by ergodicity,
	for any two states $(x,y)\in \mathcal{Z}$ there is a sequence of states $x=x_{1},...,x_{j}=x$
	that satisfies  $y=x_{2}$ and the property that $(x_{i},x_{i+1})\in \mathcal{Z}$
	for $i=1,\ldots ,j-1$, and hence, $\sum_{i=1}^{j-1}\boldsymbol{v}%
	_{x_{i}x_{i+1}}=0$. Repeating this for every possible $(x,y)\in \mathcal{Z}$%
	, there are strictly positive integers $b_{xy}$ such that $\sum_{(x,y)\in 
		\mathcal{Z}}b_{xy}\boldsymbol{v}_{xy}=0$, which implies the claim.
	
	Next we introduce the empirical measure process. \ This process is obtained
	by projecting from $\mathcal{X}^{n}$ onto $\mathcal{P}^{n}(\mathcal{X})=%
	\mathcal{P}(\mathcal{X})\cap \frac{1}{n}\mathbb{Z}^{d}\subset \mathcal{P}(%
	\mathcal{X})$, and has the generator 
	\begin{equation}
	\mathcal{M}_{\boldsymbol{\gamma }}^{n}[f](\boldsymbol{m})=n\sum_{(x,y)\in 
		\mathcal{Z}}\gamma _{xy}m_{x}\left[ f\left( \boldsymbol{m}+\frac{1}{n}%
	\boldsymbol{v}_{xy}\right) -f(\boldsymbol{m})\right] .
	\label{mean field process}
	\end{equation}
	
	One can interpret the  base model introduced above as a collection of independent
	agents with each evolving according to the transition rate $\boldsymbol{%
		\gamma }$. This is the \textquotedblleft preferred\textquotedblright\ or
	\textquotedblleft nominal\textquotedblright\ dynamics, and is what would
	occur if no \textquotedblleft outside influence\textquotedblright\ or other
	form of control acts on the agents. \ If a controller should wish to change
	this behavior, then it must pay a cost to do so. We would like to model the
	situation in which limited information about the system state, and in particular
	information relating only to the empirical measure of the states of all
	agents, is used to produce a desired behavior of the group of agents, which
	again will be characterized in terms of their empirical measure. 
	
	To precisely formulate the control problem,  we consider a continuous \textquotedblleft reward\textquotedblright\
	function $R:\mathcal{P}(\mathcal{X})\rightarrow \lbrack 0,\infty )$, where
	we recall 
	\begin{equation*}
	\mathcal{P}(\mathcal{X}):= \left\{ \boldsymbol{m}\in \mathbb{R}^{%
		\mathcal{X}}:m_{x}\geq 0\text{ for all }x\in \mathcal{X}\text{ and}%
	\sum_{x\in \mathcal{X}}m_{x}=1\right\}  \label{probability space}
	\end{equation*}
	is the simplex of probability measures on $\mathcal{X}$. We also have a cost
	function $\boldsymbol{C}=\{C_{xy}:[0,\infty )\rightarrow \lbrack 0,\infty
	]\}_{(x,y)\in \mathcal{Z}}.$ In the controlled setting, the jump rates of
	each agent can be perturbed from $\boldsymbol{\gamma }$ to$\boldsymbol{\ u}$%
	, and we let $\boldsymbol{\chi }^{n}$ denote the corresponding controlled state occupied by the
	collection of agents.   
	If the problem is of interest over the interval $[0,T],$ where $T$ can be a random variable,
	and the initial state is
	$\boldsymbol{x}^{n}=\{x_{i}^{n}\}_{i\leq n}\in 
	\mathcal{X}^{n}$,    then there is a
	collective risk-sensitive cost (paid by the coordinating controller) equal to
	\begin{equation}
	\mathbb{E}_{\boldsymbol{x}^{n}}\left[ \exp \left( \int_{0}^{T}\left(
	\sum_{i=1}^{n}\sum_{y\in \mathcal{Z}_{\chi _{i}^{n}(t)}}\gamma _{\chi
		_{i}^{n}(t)y}C_{\chi _{i}^{n}(t)y}\left( \frac{u_{\chi _{i}^{n}(t)y}(t,i)}{%
		\gamma _{\chi _{i}^{n}(t)y}}\right) -nR(L(\boldsymbol{\chi }^{n}(t)))\right)
	dt\right) \right],   \label{eqdef:agent cost}
	\end{equation}%
	where  for any $\boldsymbol{x}^{n}=\{x_{i}^{n}\}_{i\leq n}\in 
	\mathcal{X}^{n}$, define 
	\begin{equation}
	\label{def-L}
	L(\boldsymbol{x}^{n}):= \sum_{i=1}^{n}\delta _{x_{i}^{n}}.
	\end{equation}%
	Here, the control process $\boldsymbol{u}$ takes values in a space that will
	be defined later, and for a collection of $n|\mathcal{Z}|$ independent
	Poisson random measures (PRM) $\{N_{i,xy}^{1}\}_{1\leq i\leq n,(x,y)\in 
		\mathcal{Z}}$ with intensity measure equal to Lebesgue measure, the
	controlled dynamics are given by 
	\begin{equation}
	\chi _{i}^{n}(t)=x_{i}^{n}+\sum_{(x,y)\in \mathcal{Z}}\boldsymbol{v}%
	_{xy}\int_{(0,t]}\int_{[0,\infty )}1_{\left[ 0,1_{x}(\chi
		_{i}^{n}(s-))u_{xy}(s,i)\right] }(r)N_{i,xy}^{1}(dsdr).
	\label{eqdef:individual_dynamics}
	\end{equation}%
	Thus $\chi _{i}^{n}$ changes from state $x$ to $y$ with rate $u_{xy}$. The
	formulation of the dynamics in terms of a stochastic differential equation
	will be convenient in the analysis to follow.
	
	In this paper we present two results. The first is that, under additional
	assumptions on the cost $\boldsymbol{C}$, for each $n$, the risk-sensitive
	control problem is equivalent to an ordinary control problem  with the cost
	function $\boldsymbol{F}=\{F_{xy}\}_{(x,y)\in \mathcal{Z}}$, where $F_{xy}$
	is defined by
	
	\begin{equation}
	F_{xy}(q):= \sup_{u\in (0,\infty )}G_{xy}(u,q)\hspace{8pt}\text{and}%
	\hspace{8pt}G_{xy}(u,q):= \left[ u\ell \left( \frac{q}{u}\right) -\gamma
	_{xy}C_{xy}\left( \frac{u}{\gamma _{xy}}\right) \right] ,
	\label{eqdef:Fxy,Gxy}
	\end{equation}%
	with  
	\begin{equation}
	\ell (q):= q\log q-q+1, \qquad \text{ for }q\geq 0.  \label{ell}
	\end{equation}%
	Under the additional conditions we do not end up with a \textit{stochastic
		game}, as is typically the case for risk-sensitive control problems, but rather a control problem with additive cost.
	Control problems are often substantially simpler than games, and in
	particular are often more tractable from a computational perspective. The
	second contribution, again under additional assumptions on $\boldsymbol{C}$,
	is that the sequence of value functions, suitably renormalized, converges to
	the value function  \eqref{eqdef:V} of a deterministic control problem. This convergence result is also helpful
	in the construction of near-optimal controls  for a large  $n$-agent system. 
	
	\begin{example} 
		As an example consider the issue of modeling the users of a resource such as
		energy. Here the agents would be households or similar entities. The state
		of an agent indicates their use of the common resource, and this usage
		evolves in a Markovian fashion.  In exchange for a cost paid by the central
		controller to the individual agents, the agents agree to modify their
		behavior based on the current loading of the system. Thus an energy consumer
		would agree to give up control on if or when certain activities requiring energy consumption take place 
		thus altering the evolution of his own state, but will be compensated for
		doing so  by the central controller. The goal of the central controller, and the motivation for paying
		this cost, is to manage the group behavior so as to keep the system, as
		characterized by the empirical measure, in a desired operating region for as
		long as possible and with minimal cost. In this context, the use of risk
		sensitive cost is motivated in part by the resulting properties of
		robustness with respect to model error.
	\end{example}
	
	\begin{remark}  \label{rem-rn}
		If one wishes, it is possible to work with sequences $\boldsymbol{C}%
		^{n},R^{n}$ of cost and reward functions, as long as some type of
		convergence is assumed for when $n$ goes to infinity. The reader that is interested in such a
		generalization can look at a previous version of our paper in \url{http://www.wias-berlin.de/preprint/2407/wias\_preprints\_2407\_20180212.pdf}
	\end{remark}
	
	\subsection{Related literature and remarks}
	
	For ordinary discrete-time and continuous-time stochastic control problems
	(also referred to  as Markov decision processes) 
	\cite{Arapostathis1993, Hernandez-Lerma1999,put,Guo2009,fleson},  one controls a
	random process to optimize an expected cost. The most common objective function that is optimized for continuous-time escape (or ruin) stochastic control problems
	are of the form 
	\begin{equation}
	J_{T}(x_{0},\pi )=\mathbb{E}_{x_{0},\pi }\left[
	\int_{0}^{T}C(X_{t},u_{t})dt+P(X_{T})\right] ,  \label{ctMdC}
	\end{equation}%
	where $C$ is some cost function that depends on the state $x\in X$ and the
	control/action $u\in U,$ and $\pi $ is a policy or strategy that influences the dynamics $\{X_t, t \geq 0\}$, and $P$ is a terminal cost that depends on the final state of the system.  For
	risk-sensitive stochastic control problems one deals with optimality
	criteria of the form 
	\begin{equation}
	J_{T}(x_{0},\pi )=g^{-1}\left( \mathbb{E}_{x_{0},\pi }\left[ g\left(
	\int_{0}^{T}C(X_{t},u_{t})dt+P(X_{T})\right) \right] \right) ,
	\label{rsctMdC}
	\end{equation}%
	where $g$ is a monotone convex/concave function, and $C$ and $P$ are as above.  One motivation behind the
	use of risk-sensitive cost structures is that, depending on the type of
	monotonicity, variation from the average is  more (risk-averting behavior)
		or less (risk-seeking behavior) penalized. One of the most studied cases
	is the entropic risk measure corresponding to $g_{\theta }(x)=e^{\theta x}, \theta\in\mathbb{R}$ 
	(see \cite{Avila-Godoy1998, Chung1987,
		DiMasi2007, Fleming1997, Hernandez-Hernandez1996, Howard1972,
		Jaskiewicz2007, Marcus1997} for discrete time and \cite%
	{Dupuisb,Dupuis1997b,Ghosh} for continuous time). The function $g_{\theta
	}(x)=e^{\theta x}$ is special because it satisfies the property 
	\begin{equation*}
	\frac{1}{\theta }\log \left( \mathbb{E}\left[ \exp \left( \theta X\right) %
	\right] \right) =\tilde{X}+\frac{1}{\theta }\log \left( \mathbb{E}\left[
	e^{\theta (X-\tilde{X})}\right] \right) ,
	\end{equation*}%
	where $X$ is a random variable and $\tilde{X}$ its expectation. This
	property implies that the weight that is given to deviations from the
	expectation depends only on the difference from the expectation and not the
	expectation itself. It can be proved that the exponential is the only
	function that satisfies such a property (see \cite{Whittle1996}).
	Furthermore, exponential integrals have a variational characterization
	involving entropy, which turns out to be convenient from the mathematical point
	of view, and also allows for an explicit analysis of the robust and
	model insensitivity properties of the resulting controls \cite{dupjampet,
		petjamdup}. In our problem $\theta $ is integrated into the choice of cost $C.$
	
	\subsection{Notation}
	
	We now introduce some common notation that will be used throughout the article. 
	For a locally compact Polish space $\mathcal{S}$, the space of positive
	Borel measures on $\mathcal{S}$ is denoted by $\mathcal{M}(\mathcal{S})$.
	We use  $\mathcal{M}_f(\mathcal{S})$ and $\mathcal{M}_c(\mathcal{S})$ to  denote
	the subspaces of $\mathcal{M}(\mathcal{S})$ consisting, respectively, of  finite
	measures, and of  measures that are finite on every compact subset.
	Letting $C_{c}(S)$ denote the space of continuous functions with compact
	support, we equip $\mathcal{M}_{c}(\mathcal{S})$ with the weakest topology
	such that for every $f\in C_{c}(S),$ the function $\nu\rightarrow\int_{S}fd%
	\nu,\, \nu\in M_{c}(S),$ is continuous. Let $\mathcal{B}(\mathcal{S})$ be the
	Borel $\sigma $-algebra on $\mathcal{S}$ and $\mathcal{P}(\mathcal{S})$ the
	set of probability measures on $(\mathcal{S},\mathcal{B}(\mathcal{S}))$.
	Finally, for a second Polish space $\mathcal{S}^{\prime },$ we let 
	\begin{equation}
	\mathcal{F}(\mathcal{S};\mathcal{S}^{\prime })=\{f:\mathcal{S}\rightarrow 
	\mathcal{S}^{\prime }:f\,\text{measurable}\}  \label{measurable}
	\end{equation}%
	denote the space of measurable functions from $\mathcal{S}$ to $\mathcal{S}%
	^{\prime }.$ For the finite set $\mathcal{X}$  and $a > 0$, let 
	\begin{equation}
	\mathcal{P}_{\ast }(\mathcal{X}){=}\left\{\boldsymbol{m}\in \mathcal{P}(%
	\mathcal{X}) : m_{x}>0\text{ for all }x\in \mathcal{X}\right\} \hspace{6pt}%
	\text{and}\hspace{6pt} \mathcal{P}_{a}(\mathcal{X}){=}\left\{\boldsymbol{m}\in 
	\mathcal{P}(\mathcal{X}) :m_{x}\geq a\text{ for all }x\in \mathcal{X}%
	\right\} .  \label{eqn:pa}
	\end{equation}
	
	For a set $K\subset \mathcal{P}(\mathcal{X})$, the closure $\bar{K},$ the
	complement $K^{c}\,$ and the interior $K^{\circ }, $ will be considered with
	respect to the restriction of the Euclidean topology on the set $\mathcal{P}(%
	\mathcal{X}).$ Let $\mathcal{D}([0,\infty );\mathcal{S})$ denote the space of
	c\`{a}dl\`{a}g functions on $\mathcal{S}$, equipped with the Skorohod topology (see 
	\cite[Section 16]{MR1324786}), i.e., the Skorohod space. This space is
	separable and complete \cite[Theorem 16.3]{MR1324786}, and a set is
	relatively compact in $\mathcal{D}([0,\infty );\mathcal{S}),$ if and only if
	for every $M<\infty $, its natural projection on $\mathcal{D}([0,M];\mathcal{%
		S}),$ is relatively compact \cite[Theorem 16.4]{MR1324786}.
	
	For $\bar{\mathcal{M}}=\mathcal{M}_{c}([0,\infty )^{2})$, let $\mathbb{P}$
	be the probability measure on $(\bar{\mathcal{M}},\mathcal{B}(\bar{\mathcal{M%
	}})),$ under which the canonical map $N(\omega )=\omega $ is a Poisson
	measure with intensity measure equal to Lebesgue measure on $[0,\infty )^{2}$%
	. Let $\mathcal{G}_{t}=\sigma \{N((0,s]\times A):0\leq s\leq t,A\in \mathcal{%
		B}([0,\infty ))\},$ and let $\mathcal{F}_{t}$ be the completion of $\mathcal{%
		G}_{t}$ under $\mathbb{P}.$ Let $\mathcal{P}$ be the corresponding
	predictable $\sigma $-field in $[0,\infty )\times \bar{\mathcal{M}}.$
	For natural numbers $k,k^{\prime }$, we similarly define a measure 
	$\mathbb{P}^{k,k^{\prime }}$ on $(\bar{\mathcal{M}}^{k^{\prime }},\mathcal{B}%
	(\bar{\mathcal{M}}^{k^{\prime }}))$ under which the maps $N^{k}_{i}(\boldsymbol{%
		\omega })=\omega _{i},\,1\leq i\leq k^{\prime },$ are independent Poisson
	measures with intensity measure equal to $k$ times the Lebesgue measure on $%
	[0,\infty )^{2}$. $\{\mathcal{G}_{t}^{k,k^{\prime }}\},\{\mathcal{F}%
	_{t}^{k,k^{\prime }}\},$ and $\mathcal{P}^{k,k^{\prime }}$ are defined
	analogously. Let $\mathcal{A}$ be the class of $\mathcal{P}\setminus 
	\mathcal{B}([0,\infty ))$ measurable maps $\phi :[0,\infty )\times \bar{\mathcal{M}}\rightarrow \lbrack 0,\infty )$, and $\mathcal{A}_{b}$ the
	subset of these maps that are uniformly bounded from below away from
	zero and above by a positive constant. Similarly we define $\mathcal{A}%
	^{k,k^{\prime }}$ to be the set of $\mathcal{P}^{k,k^{\prime }}\setminus 
	\mathcal{B}([0,\infty )^{k^{\prime }})$ measurable maps $\boldsymbol{\phi }%
	:[0,\infty )\times \bar{\mathcal{M}}^{k^{\prime }}\rightarrow \lbrack
	0,\infty )^{k^{\prime }}$, and $\mathcal{A}_{b}^{k,k^{\prime }}$ the subset
	of these  maps for which each component is uniformly bounded from below and
	above by strictly positive constants.

	\section{Model Description}
	
	Throughout this section, fix $n \in \mathbb{N}$, and let $C$ and $R$ be, respectively, the cost and reward functions introduced in Section \ref{subs-mot}.   
	
	\subsection{The many-agent\ control problem} 
	
	For a subset $\mathcal{K}$ of $\mathcal{X}^{n}$, we define a risk-sensitive
	cost $\mathcal{I}_{\mathcal{K}}^{n}:\mathcal{X}^{n}\times \mathcal{A}%
	_{b}^{1,n|\mathcal{Z}|}\rightarrow \lbrack 0,\infty ]$ that corresponds to the 
	cost/reward up to the first time of hitting  $\mathcal{K}$  as follows: 
	\begin{equation}
	\mathcal{I}_{\mathcal{K}}^{n}(\boldsymbol{x}^{n},\boldsymbol{u}):= 
	\mathbb{E}_{\boldsymbol{x}^{n}}\left[ \exp\left(\int_{0}^{T_{\mathcal{K}%
	}}\left( \sum_{i=1}^{n}\sum_{y\in \mathcal{Z}_{\chi _{i}^{n}(t)}}\gamma
	_{\chi _{i}^{n}(t)y}C_{\chi _{i}^{n}(t)y}\left( \frac{u_{\chi
			_{i}^{n}(t)y}(t,i)}{\gamma _{\chi _{i}^{n}(t)y}}\right) -nR(L(\boldsymbol{%
		\chi }^{n}(t)))\right) dt\right)\right] ,  \label{eqdef:JWscr}
	\end{equation}%
	where 
	$\mathbb{E}_{\boldsymbol{x}^{n}}$ denotes the expected value given $%
	\boldsymbol{\chi }^{n}(0)=\boldsymbol{x}^{n}$, 
	$\{ \boldsymbol{\chi}^{n} (t), t \geq 0\}$ follows   the dynamics  given in 
	(\ref{eqdef:individual_dynamics}), and $T_{\mathcal{K}}$ is the hitting time  
	\begin{equation}
	T_{\mathcal{K}}:= \inf \left\{ t\in \lbrack 0,\infty ]:\boldsymbol{\chi }%
	^{n}(t)\in \mathcal{K}\right\} .  \label{eqdef:stop_time}
	\end{equation}%
	We define the value function $\mathcal{W}_{\mathcal{K}}^{n}:\mathcal{X}%
	^{n}\rightarrow \lbrack 0,\infty ]$ by
	
	\begin{equation}
	\mathcal{W}_{\mathcal{K}}^{n}(\boldsymbol{x}^{n}):= \inf_{\boldsymbol{u}%
		\in \mathcal{A}_{b}^{1,n|\mathcal{Z}|}}\mathcal{I}_{\mathcal{K}}^{n}(%
	\boldsymbol{x}^{n},\boldsymbol{u}).  \label{eqdef:Wscr}
	\end{equation}
	
	Similarly, for a set $\mathcal{K}\subset \mathcal{X}^{n}$ we define the
	ordinary cost $\mathcal{J}_{\mathcal{K}}^{n}:\mathcal{X}^{n}\times \mathcal{A%
	}_{b}^{1,n|\mathcal{Z}|}\rightarrow \lbrack 0,\infty ]$ and corresponding
	value function $\mathcal{V}_{\mathcal{K}}^{n}:\mathcal{X}^{n}\rightarrow
	\lbrack 0,\infty ]$ by 
	\begin{equation}
	\mathcal{J}_{\mathcal{K}}^{n}(\boldsymbol{x}^{n},\boldsymbol{q}):= 
	\mathbb{E}_{\boldsymbol{x}^{n}}\left[ \int_{0}^{T_{\mathcal{K}}}\left( \frac{%
		1}{n}\sum_{i=1}^{n}\sum_{y\in \mathcal{Z}_{\chi _{i}^{n}(t)}}F_{\chi
		_{i}^{n}(t)y}(q_{\chi _{i}^{n}(t)y}(t,i))+R(L(\boldsymbol{\chi }%
	^{n}(t)))\right) dt\right] ,  \label{eqdef:JVncal}
	\end{equation}%
	where $\boldsymbol{F} = \{F_{xy}\}_{(x, y) \in {\mathcal Z}\}}$ is  defined in
	(\ref{eqdef:Fxy,Gxy}), and  
	\begin{equation}
	\mathcal{V}_{\mathcal{K}}^{n}(\boldsymbol{x}^{n}):= \inf_{\boldsymbol{q}%
		\in \mathcal{A}_{b}^{1,n|\mathcal{Z}|}}\mathcal{J}_{\mathcal{K}}^{n}(%
	\boldsymbol{x}^{n},\boldsymbol{q}),  \label{eqdef:Vncal}
	\end{equation}%
	where the dynamics of $ \{ \boldsymbol{\chi}^{n} (t), t \geq 0\}$  are now given by (\ref{eqdef:individual_dynamics}) with $%
	\boldsymbol{u}$ replaced by $\boldsymbol{q}$, and the stopping time $T_{\mathcal{K}}$ is, as earlier,  given by  (\ref{eqdef:stop_time}). We remark that the reason for two different notations
	for controls is to aid the reader, by associating one with the risk
	sensitive problem and one with the regular control problem. Moreover, there
	are occasions that both variables appear at the same time, as in the
	definition of $\boldsymbol{F}$ or that of the Hamiltonian. Specific
	conditions on the cost functions will be given in Section \ref{cost function}%
	, and properties of $\boldsymbol{F}$ will be proved in Lemma \ref%
	{lemma:fnproperties}.   Note that for the many agent systems there are $n|%
	\mathcal{Z}|$ PRMs, each with intensity $1$

	\subsection{The mean-field control problems}
	
	Suppose that we have some exchangeability in the sense that for every
	permutation $\sigma$  of $\{1, 2, \ldots, n\}$,  $\sigma \mathcal{K}=\mathcal{K}.$
	Then  $\mathcal{K}$ can be identified with the subset
	\[   K := \{ L(x^n): x^n \in {\mathcal K}\},  \] 
	of the simplex of
	probability measures $\mathcal{P}(\mathcal{X})$.  Here, $L$ is as defined in \eqref{def-L}.  Then we can replace a control problem on $\mathcal{X}^{n}$ by one on $\mathcal{P}(\mathcal{X}).$ In this case $\mathcal{W}_{\mathcal{K}}^{n}$ and $%
	\mathcal{V}_{\mathcal{K}}^{n}$ can be considered as functions on $\mathcal{P}%
	^{n}(\mathcal{X})$, in the sense that we can find $W_{K}^{n},V_{K}^{n}:%
	\mathcal{P}^{n}(\mathcal{X})\rightarrow \lbrack 0,\infty ],$ such that $%
	\mathcal{W}_{\mathcal{K}}^{n}(\boldsymbol{x}^{n})=W_{K}^{n}(L(\boldsymbol{x}%
	^{n}))$ and $\mathcal{V}_{\mathcal{K}}^{n}(\boldsymbol{x}^{n})=V_{K}^{n}(L(%
	\boldsymbol{x}^{n})),$ where $L$ is as defined in \eqref{def-L}.  To see this, pick a starting point $\boldsymbol{x}%
	^{n}\in \mathcal{X}^{n}$ and some permutation $\sigma $. Then for any
	admissible control $\boldsymbol{u},$ the total cost generated starting at $%
	\boldsymbol{x}^{n}$ is the same as that generated when starting from $\boldsymbol{x}_{\sigma
	}^{n}$ and picking $\boldsymbol{u}_{\sigma }$ as the control. Therefore, for
	every $\boldsymbol{x}^{n}\in \mathcal{X}^{n},\sigma \in \mathbb{S}_{n},$ we
	have $\mathcal{V}_{\mathcal{K}}(\boldsymbol{x}^{n})=\mathcal{V}_{\mathcal{K}%
	}(\boldsymbol{x}_{\sigma }^{n}).$
	
	Define $h^{n}:\mathcal{D}([0,\infty );\mathcal{P}^{n}(\mathcal{X}))\times 
	\mathcal{A}_{b}^{n,|\mathcal{Z}|}\times \mathcal{P}^{n}(\mathcal{X})\times 
	\bar{\mathcal{M}}^{n,|\mathcal{Z}|}\rightarrow \mathcal{D}([0,\infty );%
	\mathbb{R}^{d})$ by 
	\begin{equation*}
	h^{n}\left( \boldsymbol{\mu },\boldsymbol{u},\boldsymbol{m},\frac{1}{n}%
	\boldsymbol{N}^{n}\right) (t):= \boldsymbol{m}+\sum_{(x,y)\in \mathcal{Z}%
	}\boldsymbol{v}_{xy}\int_{(0,t]}\int_{[0,\infty )}1_{[0,\mu
		_{x}(-s)u_{xy}(s)]}(r)\frac{1}{n}N_{xy}^{n}(dsdr).
	\end{equation*}%
	Since $\boldsymbol{u}\in \mathcal{A}_{b}^{n,|\mathcal{Z}|}$ implies the
	rates $u_{xy}(s)$ are uniformly bounded, one can explicitly construct a
	unique $\mathcal{D}([0,\infty );\mathcal{P}^{n}(\mathcal{X}))$-valued
	process that satisfies 
	\begin{equation}
	\boldsymbol{\mu }=h^{n}\left( \boldsymbol{\mu },\boldsymbol{u},\boldsymbol{m}%
	,\frac{1}{n}\boldsymbol{N}^{n}\right) .  \label{eqdef:calTn}
	\end{equation}%
	\cite{Dupuis2016}. Here $\boldsymbol{\mu }$ is the controlled process, $%
	\boldsymbol{u}$ is the control, $\boldsymbol{m}$ is an initial condition,
	and $\boldsymbol{N}^{n}/n$ is scaled noise.
	
	Now with $T_{K}:= \inf \left\{ t\in \lbrack 0,\infty ]:\boldsymbol{\mu }%
	(t)\in K\right\} $, the functions $I_{K}^{n},J_{V}^{n}:\mathcal{P}^{n}(%
	\mathcal{X})\times \mathcal{A}_{b}^{n,|\mathcal{Z}|}\rightarrow \lbrack
	0,\infty ]$ and $W_{K}^{n},V_{K}^{n}:\mathcal{P}^{n}(\mathcal{X})\rightarrow
	\lbrack 0,\infty ]$ are given by
	
	\begin{equation}
	W_{K}^{n}(\boldsymbol{m}):= \inf_{\boldsymbol{u}\in \mathcal{A}_{b}^{n,|%
			\mathcal{Z}|}}I_{K}^{n}(\boldsymbol{m},\boldsymbol{u}),  \label{eqdef:Wn}
	\end{equation}
	where 
	\begin{equation}
	I_{K}^{n}(\boldsymbol{m},\boldsymbol{u}):= \mathbb{E}_{\boldsymbol{m}}%
	\left[ e^{n\int_{0}^{T_{K}}\left( \sum_{(x,y)\in \mathcal{Z}}\mu
		_{x}(t)\gamma _{xy}C_{xy}\left( \frac{u_{xy}(t)}{\gamma _{xy}}\right) -R(%
		\boldsymbol{\mu }(t))\right) dt}:\boldsymbol{\mu }=h^{n}\left( \boldsymbol{%
		\mu },\boldsymbol{u},\boldsymbol{m},\frac{1}{n}\boldsymbol{N}^{n}\right) %
	\right] ,  \label{eqdef:JWn}
	\end{equation}%
	and 
	\begin{equation}
	V_{K}^{n}(\boldsymbol{m}):= \inf_{\boldsymbol{q}\in \mathcal{A}_{b}^{n,|%
			\mathcal{Z}|}}J_{K}^{n}(\boldsymbol{m},\boldsymbol{q}),  \label{eqdef:Vn}
	\end{equation}
	where 
	\begin{equation}
	J_{K}^{n}(\boldsymbol{m},\boldsymbol{q}):= \mathbb{E}_{\boldsymbol{m}}%
	\left[ \int_{0}^{T_{K}}\left( \sum_{(x,y)\in \mathcal{Z}}\mu
	_{x}(t)F_{xy}(q_{xy}(t))+R(\boldsymbol{\mu }(t))\right) dt:\boldsymbol{\mu }%
	=h^{n}\left( \boldsymbol{\mu },\boldsymbol{q},\boldsymbol{m},\frac{1}{n}%
	\boldsymbol{N}^{n}\right) \right] .  \label{eqdef:JVn}
	\end{equation}%
	For these control problems, there are $|\mathcal{Z}|$ PRMs, each with
	intensity $n$.  In contrast, recall from the discussion prior to
	\eqref{eqdef:individual_dynamics}  that the $n$-agent system dynamics are driven by $n|\mathcal{Z}|$ PRMS, each with intensity $1$.

	\section{Equivalence of the control problems}
	
	In this section we prove that, after a natural renormalization, the value
	function $\mathcal{W}_{\mathcal{K}}^{n}$ defined in (\ref{eqdef:Wscr}) is
	linked to $\mathcal{V}_{\mathcal{K}}^{n}$ defined in (\ref{eqdef:Vncal})
	which, as noted before,  is the value function of an ordinary stochastic
	control problem with a new cost function. Specifically, we show that $-\log (%
	\mathcal{W}_{K}^{n})/n$ equals $\mathcal{V}_{K}^{n}$, and that the many
	agent and the mean field control problem are equivalent  when the exchangeability condition holds:  
	\begin{equation}
	-\frac{1}{n}\log (W_{K}^{n}(L(\boldsymbol{x}^{n})))=V_{K}^{n}(L(\boldsymbol{x%
	}^{n}))= \mathcal{V}_{\mathcal{K}}^{n}(\boldsymbol{x}^{n})=-\frac{1}{n}\log (%
	\mathcal{W}_{\mathcal{K}}^{n}(\boldsymbol{x}^{n})).  \label{bigequality}
	\end{equation}
	
	\subsection{The cost function}
	
	\label{cost function}
	
	One of the aims of this paper is to identify cost structures that make sense
	for the problem formulation and for which the risk-sensitive problem is
	equivalent to a control problem (rather than a game). The only place where
	restrictions are needed are in the cost $\boldsymbol{C}$ paid by the
	centralized controller to the agents for deviating from the nominal rates $%
	\boldsymbol{\gamma }$. To see what conditions will be needed, we first
	discuss briefly the strategy to be used for the proof of \eqref{bigequality}%
	. The proof will use a related Bellman equation. Let $H:\mathcal{P}(\mathcal{%
		X})\times \mathbb{R}^{|\mathcal{Z}|}\rightarrow \mathbb{R}$ be given by 
	\begin{equation}
	H(\boldsymbol{m},\boldsymbol{\xi }):= \inf_{\boldsymbol{q}\in \lbrack
		0,\infty )^{|\mathcal{Z}|}}\left\{ \sum_{(x,y)\in \mathcal{Z}}m_{x}\left(
	q_{xy}\xi _{xy}+F_{xy}(q_{xy})\right) \right\} ,  \label{eqn:defofHn}
	\end{equation}%
	where 
	\begin{equation}
	F_{xy}(q):= \sup_{u\in (0,\infty )}G_{xy}(u,q)\hspace{8pt}\text{ and}%
	\hspace{8pt}G_{xy}(u,q):= \left[ u\ell \left( \frac{q}{u}\right) -\gamma
	_{xy}C_{xy}\left( \frac{u}{\gamma _{xy}}\right) \right] .  \label{def:F}
	\end{equation}%
	Consider the equation  
	\begin{equation}
	H\left( \boldsymbol{m},\Delta ^{n}V(\boldsymbol{m})\right) +R(\boldsymbol{m}%
	)=0\hspace{4pt}\text{in}\hspace{4pt}\mathcal{P}^{n}(\mathcal{X})\setminus K,
	\label{prelimeq}
	\end{equation}%
	where  $\Delta ^{n}V(\boldsymbol{m})$
	denotes the $|\mathcal{Z}|$-dimensional vector $n\left( V(\boldsymbol{m}+%
	\frac{\boldsymbol{v}_{xy}}{n})-V(\boldsymbol{m})\right) $, and $\Delta
	_{xy}^{n}V(\boldsymbol{m})$ is the component $n\left( V(\boldsymbol{m}+\frac{%
		\boldsymbol{v}_{xy}}{n})-V(\boldsymbol{m})\right) _{xy}$,
	$(x,y)\in \mathcal{Z}$.   
	We will show that $V_{K}^{n}$ is the unique solution $V$ to
	\eqref{prelimeq} that satisfies the  boundary condition
	$V(\boldsymbol{m})=0$ for $\boldsymbol{m}\in K$.
	We will also prove that $W_{K}^{n}$ is the unique solution to the equation 
	\begin{equation}
	\sup_{\boldsymbol{\ u}\in (0,\infty )^{|\mathcal{Z}|}}\left\{ \sum_{(x,y)\in 
		\mathcal{Z}}m_{x}\left( u_{xy}\left( \frac{W(\boldsymbol{m})-W\left( 
		\boldsymbol{m}+\frac{\boldsymbol{v}_{xy}}{n}\right) }{W(\boldsymbol{m})}%
	\right) -\gamma _{xy}C_{xy}\left( \frac{u_{xy}}{\gamma _{xy}}\right) \right)
	\right\} =-R(\boldsymbol{m})  \label{eqdef:Weq}
	\end{equation}%
	for $\boldsymbol{m}\in \mathcal{P}^{n}(\mathcal{X})\setminus K$ with
	boundary condition $W(\boldsymbol{m})=1$ for $\boldsymbol{m}\in K$.
	
	In the proof of the relation  $-\frac{1}{n} \log (W_{K}^{n})=V_{K}^{n}$,
	we will use the following lemma,  which holds under suitable conditions on
	the cost functional specified in Assumption \ref{assumption} below. The proof  of the lemma is given in  Section \ref{subs-equiv} (right after Lemma \ref{lemma-n}).

	\begin{lemma}
		\label{iff}  Suppose Assumption \ref{assumption} below holds. 
		Then, if  $\tilde{V}:\mathcal{P}^{n}(\mathcal{X})\rightarrow \lbrack
		0,\infty )$ is a solution to (\ref{prelimeq}) and $\tilde{V}(\boldsymbol{m}%
		)=0$ for $\boldsymbol{m}\in K$, then $\tilde{W}=e^{-n\tilde{V}}:\mathcal{P}%
		^{n}(\mathcal{X})\rightarrow (0,\infty )$ is a solution of \eqref{eqdef:Weq}
		and $\tilde{W}(\boldsymbol{m})=1$ for $\boldsymbol{m}\in K.$
	\end{lemma}
	
	We now provide an  outline of the proof of Lemma \ref{iff} and also provide motivation for 
	the form of our main assumption, Assumption \ref{assumption} below, on the cost function. 
	First, note that by \eqref{eqn:defofHn}-\eqref{def:F}, we have
	\begin{equation}
	\begin{split}
	H(\mathbf{m},\boldsymbol{\xi })& := \inf_{\boldsymbol{q}\in \lbrack
		0,\infty )^{|\mathcal{Z}|}}\left\{ \sum_{(x,y)\in \mathcal{Z}}m_{x}\left(
	q_{xy}\xi _{xy}+F_{xy}(q_{xy})\right) \right\} \\
	& =\inf_{\boldsymbol{q}\in \lbrack 0,\infty )^{|\mathcal{Z}|}}\sup_{%
		\boldsymbol{\ u}\in (0,\infty )^{|\mathcal{Z}|}}\left\{ \sum_{(x,y)\in 
		\mathcal{Z}}m_{x}\left( q_{xy}\xi _{xy}+G_{xy}(u_{xy},q_{xy})\right) \right\}
	\\
	& =  \inf_{\boldsymbol{q}\in \lbrack 0,\infty )^{|\mathcal{Z}|}}\sup_{%
		\boldsymbol{\ u}\in (0,\infty )^{|\mathcal{Z}|}}\left\{ \sum_{(x,y)\in 
		\mathcal{Z}}m_{x}   L_{xy} (u_{xy}, q_{xy}) \right\}, 
	\end{split}
	\label{hamiltonian0}
	\end{equation}
	where, $L_{xy}$ is defined, in terms of $\xi_{xy}$,  $\gamma_{xy}$, $C_{xy}$, $L_{xy}$ and
	the function $\ell$ defined in \eqref{ell}, as 
	\begin{equation*}
	L_{xy}(u,q) :=q \xi_{xy} +u\ell \left( \frac{q}{u}\right) -\gamma_{xy}C_{xy}\left( 
	\frac{u}{\gamma_{xy}}\right). 
	\end{equation*}
	The proof of Lemma \ref{iff} will proceed by first showing that Isaac's condition holds, that is,  that the supremum and infimum in \eqref{hamiltonian0} can be exchanged: 
	\begin{equation}
	\begin{split}
	\inf_{\boldsymbol{q}\in \lbrack 0,\infty )^{|\mathcal{Z}|}}\sup_{%
		\boldsymbol{\ u}\in (0,\infty )^{|\mathcal{Z}|}}\left\{ \sum_{(x,y)\in 
		\mathcal{Z}}m_{x}  L_{xy} (u_{xy}, q_{xy}) \right\} &= 
	\sup_{\boldsymbol{\ u}\in (0,\infty )^{|\mathcal{Z}|}}
	\inf_{\boldsymbol{q}\in \lbrack 0,\infty )^{|\mathcal{Z}|}} 
	\left\{ \sum_{(x,y)\in 
		\mathcal{Z}}m_{x}  L_{xy} (u_{xy}, q_{xy}) \right\}. 
	\end{split}
	\label{hamiltonian}
	\end{equation}

	
	Equation \eqref{hamiltonian} is clearly equivalent to 
	\begin{equation}
	\label{toshow}
 \inf_{q_{xy} \in [0,\infty)}  \sup_{u_{xy} \in (0,\infty )} L_{xy} (u_{xy}, q_{xy})  =
	\sup_{u_{xy} \in (0,\infty )} \inf_{q_{xy} \in [0,\infty)}  L_{xy} (u_{xy}, q_{xy}),\hspace{16pt} \forall	(x,y) \in {\mathcal Z}.
	\end{equation}

	We show in Lemma \ref{lemma:isaac} below  that  \eqref{toshow} holds under  
	the following main assumption on the cost function.

	\begin{assumption}
		\label{assumption} $R:\mathcal{P}(\mathcal{X})\rightarrow \lbrack 0,\infty )$
		is a continuous function. Moreover, for every $(x,y)\in \mathcal{Z},$ $%
		C_{xy}:[0,\infty )\rightarrow \lbrack 0,\infty ]$ is a convex function that
		satisfies the following:
		
		\begin{enumerate}
			\item \label{enumerate:convex}$uC_{xy}^{\prime }\left( u\right) -u\hspace{4pt}$ is increasing on the maximal open interval where $C_{xy}$ is finite; 
			
			\item \label{enumerate:nochange} $C_{xy}(1)=0.$
		\end{enumerate}
	\end{assumption}
	
	The following result, which is proved in Appendix \ref{ap:a},
shows that  part \ref{enumerate:convex} of Assumption \ref{assumption} 
is close to being necessary for \eqref{hamiltonian} to hold.

\begin{theorem}
	\label{minmax} 
	If \eqref{hamiltonian} is satisfied and for each
	$(x,y) \in {\mathcal Z}$, $C_{xy}$ is twice differentiable on some non-empty 
	interval $(u_{1,xy}, u_{2,xy})$, then part \ref{enumerate:convex} of Assumption \ref{assumption}
	is satisfied on that interval. 
\end{theorem}

Part \ref{enumerate:nochange} of Assumption \ref{assumption}  is not necessary, but
it simplifies the analysis,  and  it is appropriate for the situation being
modeled to have zero cost when there is no change from the nominal rates.
The proof of Lemma \ref{lemma:isaac}, which relies on (a modification of) Sion's theorem
(Corollary 3.3 in \cite{Sion1958}), is also deferred to Appendix \ref{ap:a}. We proceed by providing a concrete example of a family of cost functions that satisfy Assumption \ref{assumption}.
	
	\begin{example}\label{example}
		The family of functions $C_{xy}(u)=\frac{1}{pu^{p}}+\frac{u^{q}}{q}-\frac{p+q}{pq},$
		where $p\geq 1$ and $q\geq 1,$ satisfy Assumption \ref{assumption}.
		Clearly,   $C_{xy} (1) = 0$. 
		The derivative of $C_{xy}$ is $-\frac{1}{u^{p+1}}+u^{q-1}$, and so
		$u C_{xy}^\prime (u) - u = -\frac{1}{u^{p}}+u^{q}-u,$ which is always finite. 
		Taking the derivative again gives $\frac{p}{u^{p+1}}+qu^{q-1}-1$, which is
		always bigger than zero, since $\frac{p}{u^{p+1}}\,$\ and $qu^{q-1}$ are
		everywhere positive and bigger than one on the intervals $[0,1]$ and $%
		[1,\infty ),$ respectively.
	\end{example}

	\begin{lemma}
		\label{lemma:isaac} Under Assumption \ref{assumption}, the relation \eqref{toshow} holds for each
		$(x,y) \in {\mathcal Z}$, and hence, the 
		Isaac's condition stated in 
		\eqref{hamiltonian},  is 
		satisfied.
	\end{lemma}

	As an immediate corollary of the lemma, we have the following result:
	
	\begin{corollary}
		\label{cor:isaac}
		Under Assumption \ref{assumption}, for each $\mathbf{m} \in {\mathcal P}({\mathcal X})$ and
		$\boldsymbol{\xi} \in   \mathbb{R}^{|{\mathcal Z}|}$,  
		\[  H(\mathbf{m},\boldsymbol{\xi }) 
		= \sum_{(x,y)\in \mathcal{Z}} m_{x}\gamma _{xy}(C_{xy})^{\ast }\left( 1-e^{-\xi _{xy}}\right).
		\]
		where $(C_{xy})^{\ast }:(-\infty ,1)\rightarrow \mathbb{R}$ is 
		given by
		\begin{equation}
		\label{def-cast}
		(C_{xy})^{\ast }(z) :=\ sup_{u>0}\left[ zu-C_{xy}(u)\right].
		\end{equation}
	\end{corollary} 
	\begin{proof}
		First,  note that for each $(x,y)\in \mathcal{Z}$, using the fact that
		$\partial_q L_{xy} (u,q) = \xi_{xy} + \log (q/u)$, and $\partial_{qq} L_{xy} (u,q) > 0$ for
		$q > 0$, we see that 
		\begin{equation}
		\begin{split}
		\inf_{q_{xy}\in [0,\infty)}  L_{xy} (u_{xy}, q_{xy})  &= u_{xy} ( 1 - e^{-\xi_{xy}}) - \gamma_{xy} C_{xy} \left( \frac{u_{xy}}{\gamma_{xy}} \right). 
		\end{split}
		\end{equation}
		Also note  that, by the definition of $(C_{xy})^{\ast}$,
		\begin{equation} 
		\begin{split}
		\sup_{u_{xy}\in (0,\infty)} \left[  u_{xy} ( 1 - e^{-\xi_{xy}})
		- \gamma_{xy} C_{xy} \left( \frac{u_{xy}}{\gamma_{xy}} \right) \right]
		= \gamma_{xy} (C_{xy})^{\ast} (1 - e^ -\xi_{xy}).  
		\end{split}
		\end{equation}
		The corollary is then a simple consequence of the above two observations,
		\eqref{hamiltonian0} and  Lemma \ref{lemma:isaac}.
	\end{proof}

	We now summarize some other properties of the cost function that will be useful
	in the sequel.

	\begin{lemma}
		\label{lemos} Under Assumption \ref{assumption}, the cost function $C_{xy}$
		satisfy the following on $(0,\infty)$:
		
		\begin{enumerate}
			\item for every $(x,y)\in \mathcal{Z}$ we have $(C_{xy})^{\prime }(u)\geq 1-%
			\frac{1}{u}$ for $u>1,$ and therefore $\liminf_{u\rightarrow \infty
			}(C_{xy})^{\prime }(u)\geq 1$,
			
			\item \label{enumerate:lemospointtwo} for every $(x,y)\in \mathcal{Z}$ and $%
			u\in (0,\infty )$ we have $C_{xy}(u)\geq -\log u+u-1.$
		\end{enumerate}
	\end{lemma}
	\begin{proof}
		It follows from the monotonicity that $uC_{xy}^{\prime }(u)-{u\geq -1}$ for $%
		u>1$, which gives the first statement. The second follows by comparing $%
		C_{xy}(u)$ with $\int_{1}^{u}\left[ 1-\frac{1}{s}\right] ds$ and using $%
		C_{xy}(1)=0.$
	\end{proof} \\

	We conclude with  a lemma
	that collects some
	properties of $F_{xy}$, and whose proof is provided in Appendix \ref{ap:b}.  
	
	\begin{lemma}
		\label{lemma:fnproperties} For every $(x,y)\in \mathcal{Z},$ let $F_{xy}$ be
		as in (\ref{eqdef:Fxy,Gxy}), where $\{C_{xy}\}$ satisfy Assumption \ref%
		{assumption}. Then the following properties hold: 
		\begin{equation*}
		1.\,F_{xy}(q)\geq \gamma _{xy}\ell \left( \frac{q}{\gamma _{xy}}\right) \geq
		0,\hspace{16pt}2.\,F_{xy}(\gamma _{xy})=0,\hspace{16pt}3.\,F_{xy}\text{ is
			convex on}\hspace{8pt}[0,\infty ).
		\end{equation*}
	\end{lemma}
	
	\subsection{Equivalence of the stochastic problems}
	\label{subs-equiv}
	
	\begin{theorem}
		\label{Main A} Let $n\in \mathbb{N},\,\mathcal{K}\subset \mathcal{X}^{n},$
		(resp. $K\subset \mathcal{P}^{n}(\mathcal{X})$), and $\boldsymbol{C},R$ be
		as in Assumption \ref{assumption}. Then 
		\begin{equation}
		V_{K}^{n}(\boldsymbol{m})=-\frac{1}{n}\log (W_{K}^{n}(\boldsymbol{m}))
		\end{equation}%
		and%
		\begin{equation}
		\mathcal{V}_{\mathcal{K}}^{n}(\boldsymbol{x}^{n})=-\frac{1}{n}\log (\mathcal{%
			W}_{\mathcal{K}}^{n}(\boldsymbol{x}^{n})).
		\end{equation}
		
		If, in addition, $\mathcal{K}\subset \mathcal{X}^{n}$ is invariant under
		permutations, and therefore can be identified with a subset of $\mathcal{P}%
		^{n}(\mathcal{X})$, then%
		\begin{equation}
		-\frac{1}{n}\log (W_{K}^{n}(L(\boldsymbol{x}^{n})))=V_{K}^{n}(L(\boldsymbol{x%
		}^{n}))= \mathcal{V}_{\mathcal{K}}^{n}(\boldsymbol{x}^{n})=-\frac{1}{n}\log (%
		\mathcal{W}_{\mathcal{K}}^{n}(\boldsymbol{x}^{n})).
		\end{equation}
	\end{theorem}
	
	The proof of this result appears later in this section.
	Also, we will only prove the first equality and note that the third follows in a similar manner. We begin with some preparatory lemmas.

	\begin{lemma}
		\label{lemma-n}
		Let $n\in \mathbb{N},\, \emptyset \neq K\subset \mathcal{P}^{n}(\mathcal{X})$%
		, and $\boldsymbol{C},R$ be as in Assumption \ref{assumption}. Then, the equation (\ref{prelimeq}) has at least one solution.
	\end{lemma}
	
	\begin{proof}
		For the proof we use the equivalent discrete time stochastic control
		problem. We consider the following set of controls 
		\begin{equation}
		\begin{split}
		A_{a}(\boldsymbol{m}):= \left\{\boldsymbol{q}\in [0,\infty )^{|\mathcal{Z}|} : 
		\frac{1}{a}\geq \sum_{(x,y)\in \mathcal{Z}}m_{x}q_{xy}(\boldsymbol{m})\geq
		a\right\}\hspace{8pt}\text{and}\hspace{8pt} A_{+}(\boldsymbol{m}%
		):=\cup_{a>0}A_{a}(\boldsymbol{m}).
		\end{split}%
		\end{equation}
		For such a control the probability of moving from state $\boldsymbol{m}$ to
		state $\boldsymbol{m}+\frac{1}{n}\boldsymbol{v}_{\tilde{x},\tilde{y}}$ will
		be given by 
		\begin{equation*}
		\frac{m_{\tilde{x}}q_{\tilde{x}\tilde{y}}(\boldsymbol{m})}{\sum_{(x,y)\in 
				\mathcal{Z}}m_{x}q_{xy}(\boldsymbol{m})},
		\end{equation*}%
		and the (conditional) expected cost till the time of transition is given by 
		\begin{equation*}
		\frac{\sum_{(x,y)\in \mathcal{Z}}m_{x}F_{xy}(q_{xy}(\boldsymbol{m}))+R(%
			\boldsymbol{m})}{n\sum_{(x,y)\in \mathcal{Z}}m_{x}q_{xy}(\boldsymbol{m})}.
		\end{equation*}%
		Also, with some abuse of notation, we define the set of feedback controls 
		\begin{equation}
		A_{a}=\{\boldsymbol{q}\in [0,\infty )^{|\mathcal{P}^{n}(\mathcal{X})\times 
			\mathcal{Z}|} : \boldsymbol{q}(\boldsymbol{m})\in A_{a}(\boldsymbol{m})) 
		\hspace{8pt}\text{and}\hspace{8pt} A_{+}=\cup_{a>0}A_{a}.
		\end{equation}
		Given controlled transition probabilities as above, let $\boldsymbol{\mu }%
		(i) $ be the corresponding controlled discrete time process. We define the
		value function $\bar{V}_{K}^{n}(\boldsymbol{m}):\mathcal{P}(\mathbb{R}%
		^{d})\rightarrow \lbrack 0,\infty )$ by 
		\begin{equation}  \label{no R}
		\bar{V}_{K}^{n}(\boldsymbol{m}):= \inf_{\boldsymbol{q}\in A_{+}}\mathbb{E%
		}_{\boldsymbol{m}}\left[ \sum_{i=1}^{T_{K}}\frac{\sum_{(x,y)\in \mathcal{Z}%
			}\mu _{x}(i)F_{xy}(q_{xy}(\boldsymbol{\mu }(i)))+R(\boldsymbol{\mu }(i))}{%
			n\sum_{(x,y)\in \mathcal{Z}}\mu _{x}(i)q_{xy}(\boldsymbol{\mu }(i))}\right] ,
		\end{equation}%
		where $\mathbb{E}_{\boldsymbol{m}}$ denotes expected value given $%
		\boldsymbol{\mu }(0)=\boldsymbol{m}$ and $T_{K}:= \inf \{i\in \mathbb{N}:%
		\boldsymbol{\mu }(i)\in K\}.$
		
		To see that $\bar{V}_{K}^{n}(\boldsymbol{m})$ is finite, we just have to use
		the original rates and note that the total cost is proportional to the
		expected exit time, which is finite by classical results on Markov chains.
		Since $F_{xy},R\geq 0,$ and $F_{xy}$ is convex with $\gamma_{xy}\ell\left(%
		\frac{\cdot}{\gamma_{xy}}\right)$ as a lower bound (see Lemma \ref%
		{lemma:fnproperties}), one can see that we can find a constant $a_{0}>0$
		such that only controls in $A_{a_{0}}$ (or any $a<a_{0}$) should be
		considered. More specifically to see that a term in the sum appearing on the
		RHS of \eqref{no R} gets large when $\sum_{(x,y)\in \mathcal{Z}}\mu
		_{x}(i)q_{xy}(\boldsymbol{\mu }(i))$ gets small we bound the denominator by $%
		|\mathcal{Z}|$ times the biggest term and the nominator by the same term and
		then we use the fact that $F_{xy}(0)\geq \gamma_{xy}.$ For the other bound
		we use the superlinearity of $F_{xy}.$ Now by \cite[Proposition 1.1 in
		Chapter 3]{Bertsekas2012}, we have that this value function satisfies
		
		\begin{equation*}
		\bar{V}_{K}^{n}(\boldsymbol{m})=\inf_{\boldsymbol{q}\in A_{a_{0}}(%
			\boldsymbol{m})}\left\{ \frac{\sum_{(x,y)\in \mathcal{Z}%
			}m_{x}F_{xy}(q_{xy})+R(\boldsymbol{m})}{n\sum_{(x,y)\in \mathcal{Z}%
			}m_{x}q_{xy}}+\sum_{(\tilde{x},\tilde{y})\in \mathcal{Z}}\frac{m_{\tilde{x}%
			}q_{\tilde{x}\tilde{y}}}{\sum_{(x,y)\in \mathcal{Z}}m_{x}q_{xy}}\bar{V}%
		_{K}^{n}\left( \boldsymbol{m}+\frac{1}{n}\boldsymbol{v}_{\tilde{x}\tilde{y}%
		}\right) \right\} .
		\end{equation*}
		
		It then follows that $\bar{V}_{K}^{n}(\boldsymbol{m})$ satisfies the last
		display if and only if [with $\Delta _{xy}^{n}\bar{V}_{K}^{n}(\boldsymbol{m}%
		):= n\left( \bar{V}_{K}^{n}(\boldsymbol{m}+\frac{\boldsymbol{v}_{xy}}{n}%
		)-\bar{V}_{K}^{n}(\boldsymbol{m})\right) $] 
		\begin{equation*}
		\inf_{\boldsymbol{q} \in A_{a_{0}}(\boldsymbol{m})}\left\{ \sum_{(x,y)\in 
			\mathcal{Z}}m_{x}\left( q_{xy}\Delta _{xy}^{n}\bar{V}_{K}^{n}(\boldsymbol{m}%
		)+F_{xy}(q_{xy})\right) \right\} +R(\boldsymbol{m})=0.
		\end{equation*}
		Since $a_{0}$ can be chosen arbitrary small and the left side on the
		previous display is continuous with respect to $\boldsymbol{q},$ we get 
		\begin{equation*}
		\inf_{\boldsymbol{q}\in [0,\infty)^{|\mathcal{Z}|}}\left\{ \sum_{(x,y)\in 
			\mathcal{Z}}m_{x}\left( q_{xy}\Delta _{xy}^{n}\bar{V}_{K}^{n}(\boldsymbol{m}%
		)+F_{xy}(q_{xy})\right) \right\} +R(\boldsymbol{m})=0.
		\end{equation*}
		Then using the definition (\ref{eqn:defofHn}) this is the same as 
		\begin{equation*}
		H^{n}\left( \boldsymbol{m},\Delta ^{n}\bar{V}_{K}^{n}(\boldsymbol{m})\right)
		+R(\boldsymbol{m})=0,
		\end{equation*}%
		and we also have the boundary condition $\bar{V}_{K}^{n}(\boldsymbol{m})=0$
		for all $\boldsymbol{m}\in K.$
	\end{proof}
	
	\medskip
	
	\begin{proof}[Proof of Lemma \protect\ref{iff}]
		Let $\tilde{V}$ be a solution to (\ref{prelimeq}). We then have $H^{n}(%
		\boldsymbol{m},\Delta ^{n}\tilde{V}(\boldsymbol{m}))+R(\boldsymbol{m})=0. $
		Using Corollary \ref{cor:isaac} and the definition \eqref{def-cast}
		of $C^{\ast}$, this implies  
		\begin{equation*}
		\sup_{\boldsymbol{\ u}\in (0,\infty )^{|\mathcal{Z}|}}\left\{ \sum_{(x,y)\in 
			\mathcal{Z}}m_{x}\left( u_{xy}\left( 1-e^{-n\left( \tilde{V}(\boldsymbol{m}+%
			\frac{\boldsymbol{v}_{xy}}{n})-\tilde{V}(\boldsymbol{m})\right) }\right)
		-\gamma _{xy}C_{xy}\left( \frac{u_{xy}}{\gamma _{xy}}\right) \right)
		\right\} +R(\boldsymbol{m})=0.
		\end{equation*}%
		By making the substitution   $\tilde{W}  =   e^{-n \tilde{V}}$, we have 
		\begin{equation*}
		\sup_{\boldsymbol{\ u}\in (0,\infty )^{|\mathcal{Z}|}}\left\{ \sum_{(x,y)\in 
			\mathcal{Z}}m_{x}\left( u_{xy}\left( 1-\frac{\tilde{W}(\boldsymbol{m}+\frac{%
				\boldsymbol{v}_{xy}}{n})}{\tilde{W}(\boldsymbol{m})}\right) -\gamma
		_{xy}C_{xy}\left( \frac{u_{xy}}{\gamma _{xy}}\right) \right) \right\} +R(%
		\boldsymbol{m})=0,
		\end{equation*}%
		which is the same as (\ref{eqdef:Weq}).
	\end{proof}
	
	\begin{lemma}
		\label{martingale} Let $f:\mathcal{P}^{n}(\mathcal{X})\rightarrow \mathbb{R}%
		, $ $\boldsymbol{m}\in \mathcal{P}^{n}(\mathcal{X}),$ and $\boldsymbol{q}\in 
		\mathcal{A}_{b}^{n,|\mathcal{Z}|}$ be given, and let $\boldsymbol{\mu }$
		solve (\ref{eqdef:calTn}). Then 
		\begin{equation*}
		f(\boldsymbol{\mu }(t\wedge T_{K}))-f(\boldsymbol{\mu }(t^{\prime }\wedge
		T_{K}))-\int_{t^{\prime }\wedge T_{K}}^{t\wedge T_{K}}\sum_{(x,y)\in 
			\mathcal{Z}}\mu _{x}(s)q_{xy}(s)\Delta _{xy}^{n}f(\boldsymbol{\mu }(s))ds,
		\end{equation*}%
		is a martingale with respect to the filtration $\{\mathcal{F}_{t}\}.$ 
	\end{lemma}
	
	This is a classical result, and the proof entails a suitable application of Ito's
	formula (see \cite[Chapter 2, Theorem 5.1]{MR1011252} for more details).
	
	\begin{lemma}
		\label%
		{expmartingale} Let $g:\mathcal{P}^{n}(\mathcal{X})\rightarrow (0,\infty ),$ 
		$\boldsymbol{m}\in \mathcal{P}^{n}(\mathcal{X}),$ and $\boldsymbol{u}\in 
		\mathcal{A}_{b}^{n,|\mathcal{Z}|}$ be given, and let $\boldsymbol{\mu }$
		solve (\ref{eqdef:calTn}). Then 
		\begin{equation}
		\frac{g(\boldsymbol{\mu }(t\wedge T_{K}))}{g(\boldsymbol{\mu }(t^{\prime
			}\wedge T_{K}))}\exp \left\{ -\int_{t^{\prime }\wedge T_{K}}^{t\wedge
			T_{K}}\sum_{(x,y)\in \mathcal{Z}}\mu _{x}(s)u_{xy}(s)\frac{\Delta _{xy}^{n}g(%
			\boldsymbol{\mu }(s))}{g(\boldsymbol{\mu }(s))}ds\right\}
		\end{equation}%
		is a martingale with respect to the filtration $\mathcal{F}_{t}.$
	\end{lemma}
	
	\begin{proof}
		The proof is a direct application of the corollary in \cite[Page 66]%
		{N.Ethier1986}.
	\end{proof}

\begin{lemma}
\label{finitetime} Let $\boldsymbol{m}\in \mathcal{P}^{n}(\mathcal{X})$ and $%
\boldsymbol{\ u}\in \mathcal{A}_{b}^{n,|\mathcal{Z}|}.$ There exists a
constant $c>0$, that depends only on the bounds on $\boldsymbol{u},$ the
dimension $d$, the constant $R_{\max }=\max \{R(\boldsymbol{m}):\boldsymbol{m%
}\in \mathcal{P}^{n}(\mathcal{X})\},$ and the number $n$ of agents, such
that for every $t\geq t^{\prime }\geq 0,$ 
\begin{equation*}
\mathbb{E}_{\boldsymbol{m}}\left[ e^{-nR_{\max }(t\wedge T_{K}-t^{\prime
}\wedge T_{K})}\Big|\mathcal{F}_{t^{\prime }}\right] \geq c.
\end{equation*}%
Furthermore it is true that

\begin{equation*}
T_{K}<\infty \text{ a.s.},\hspace{8pt}\text{and}\hspace{8pt}\mathbb{E}_{%
\boldsymbol{m}}\left[ e^{-nR_{\max }(T_{K}-t^{\prime }\wedge T_{K})}\Big|%
\mathcal{F}_{t^{\prime }}\right] \geq c.
\end{equation*}
\end{lemma}

\begin{proof}
We claim there exists $g$ such that for all $s$

\begin{equation}
\sum_{(x,y)\in \mathcal{Z}}\mu _{x}(s)u_{xy}(s)\frac{\Delta _{xy}^{n}g(%
\boldsymbol{\mu }(s))}{g(\boldsymbol{\mu }(s))}\geq nR_{\max }.
\label{wfinitetime}
\end{equation}%
To show the existence of such a $g$ we use the following procedure. Since
the one agent process with generator given in (\ref{one agent process}%
) is ergodic, we have that the process on $\mathcal{X}^{n},$ with generator
given in (\ref{many agent process}), as well as the one on $\mathcal{P}%
^{n}(\mathcal{X}),$ with generator given in (\ref{mean field process}), are
also ergodic. We split $\mathcal{P}^{n}(\mathcal{X})$ into sets $%
\{K_{i}\}_{0\leq i\leq i_{\max }},$ where $K_{0}=K,$ and $K_{i+1}$ is
generated inductively as the set of all points in $\mathcal{P}^{n}(\mathcal{X%
})$ that do not belong to $K_{i}$ but such that the process with generator (%
\ref{mean field process}) can reach $K_{i}$ in one jump. Since the original
process has $d$ states, it is easy to see that $i_{\max }\leq d^{n}.$ Since$%
\boldsymbol{\ u}\in \mathcal{A}_{b}^{n,|\mathcal{Z}|}$, there exist
constants $0<c_{1}\leq c_{2}<\infty $ such that $c_{1}\leq u_{xy}(t)\leq
c_{2}$ for all $t\geq 0$ a.s. Let $g$ be defined by 
\begin{equation*}
g(\boldsymbol{m})\doteq \left( \frac{nR_{\max }+nd^{2}c_{2}+c_{1}}{c_{1}}%
\right) ^{i_{\max }-i},\hspace{8pt}\text{for}\hspace{8pt}\boldsymbol{m}\in
K_{i}.
\end{equation*}%
Let $\boldsymbol{\mu }(\cdot )$ be the process with control $\boldsymbol{u}$%
. For $0\leq s\leq t$ suppose that $\boldsymbol{\mu }(s)\in K_{i}$ for some $%
i\geq 1.$ Then there exists at least one $(\tilde{x},\tilde{y})\in \mathcal{Z%
}$ such that $\boldsymbol{\mu }(s)+\frac{v_{\tilde{x}\tilde{y}}}{n}\in
K_{i-1}.$ Therefore 
\begin{equation*}
\begin{split}
& \sum_{(x,y)\in \mathcal{Z}}\mu _{x}(s)u_{xy}(s)\frac{\Delta _{xy}^{n}g(%
\boldsymbol{\mu }(s))}{g(\boldsymbol{\mu }(s))}=\mu _{\tilde{x}}(s)u_{\tilde{%
x}\tilde{y}}(s)\frac{\Delta _{\tilde{x}\tilde{y}}^{n}g(\boldsymbol{\mu }(s))%
}{g(\boldsymbol{\mu }(s))}+n\sum_{(x,y)\in \mathcal{Z},(x,y)\neq \left( 
\tilde{x},\tilde{y}\right) }\frac{g(\boldsymbol{\mu }(s)+\frac{\boldsymbol{v}%
_{xy}}{n})}{g(\boldsymbol{\mu }(s))}\mu _{x}(s)u_{xy}(s) \\
& -n\sum_{(x,y)\in \mathcal{Z},(x,y)\neq \left( \tilde{x},\tilde{y}\right) }%
\frac{g(\boldsymbol{\mu }(s))}{g(\boldsymbol{\mu }(s))}\mu
_{x}(s)u_{xy}(s)\geq \mu _{\tilde{x}}(s)u_{\tilde{x}\tilde{y}}(s)\frac{%
\Delta _{\tilde{x}\tilde{y}}^{n}g(\boldsymbol{\mu }(s))}{g(\boldsymbol{\mu }%
(s))}-n\sum_{(x,y)\in \mathcal{Z}}\mu _{x}(s)u_{xy}(s) \\
& \geq c_{1}\left( \frac{nR_{\max }+nd^{2}c_{2}+c_{1}}{c_{1}}-1\right)
-nc_{2}d^{2}\geq nR_{\max },
\end{split}
\end{equation*}
where in the next to last inequality we used the fact that $\mu _{\tilde{x}%
}(s)\geq \frac{1}{n}$ (because otherwise there is no agent at $\tilde{x}$
to move), and that $\Delta _{xy}^{n}V(\boldsymbol{m})=n\left( V(\boldsymbol{m%
}+\frac{\boldsymbol{v}_{xy}}{n})-V(\boldsymbol{m})\right) $.

Using Lemma \ref{expmartingale}, we have 

\begin{equation*}
\mathbb{E}_{\boldsymbol{m}}\left[ \frac{g(\boldsymbol{\mu }(t\wedge T_{K}))}{%
g(\boldsymbol{\mu }(t^{\prime }\wedge T_{K}))}\exp \left\{ -\int_{t^{\prime
}\wedge T_{K}}^{t\wedge T_{K}}\sum_{(x,y)\in \mathcal{Z}}\mu _{x}(s)u_{xy}(s)%
\frac{\Delta _{xy}^{n}g(\boldsymbol{\mu }(s))}{g(\boldsymbol{\mu }(s))}%
ds\right\} \Bigg|\mathcal{F}_{t^{\prime }}\right] =1,
\end{equation*}%
from which we get%
\begin{equation*}
\mathbb{E}_{\boldsymbol{m}}\left[ \exp \left\{ -\int_{t^{\prime }\wedge
T_{K}}^{t\wedge T_{K}}\sum_{(x,y)\in \mathcal{Z}}\mu _{x}(s)u_{xy}(s)\frac{%
\Delta _{xy}^{n}g(\boldsymbol{\mu }(s))}{g(\boldsymbol{\mu }(s))}ds\right\} %
\Bigg|\mathcal{F}_{t^{\prime }}\right] \geq c\doteq \frac{\min_{\mathcal{P}%
^{n}(\mathcal{X})}g}{\max_{\mathcal{P}^{n}(\mathcal{X})}g}.
\end{equation*}%
By applying equation (\ref{wfinitetime}) 
\begin{equation*}
\mathbb{E}_{\boldsymbol{m}}\left[ e^{-nR_{\max }(t\wedge T_{K}-t^{\prime
}\wedge T_{K})}\Big|\mathcal{F}_{t^{\prime }}\right] \geq c.
\end{equation*}

Now choose now $\tau >0$ such that $e^{-nR_{\max }\tau }\leq c/2$. We claim
that 
\begin{equation*}
\left( T_{K}\leq t^{\prime }+\tau \right) \Leftrightarrow \left( T_{K}\wedge
(t^{\prime }+2\tau )-t^{\prime }\wedge T_{K}\right) \leq \tau .
\end{equation*}%
Indeed if $t^{\prime }\geq T_{K},$ then both parts are trivially true. Let
assume that $t^{\prime }\leq T_{K},$ and $T_{K}\leq t^{\prime }+\tau .$ Then 
$T_{K}\wedge (t^{\prime }+2\tau )=T_{K},$ and $t^{\prime }\wedge
T_{K}=t^{\prime },$ and therefore $\left( T_{K}\wedge (t^{\prime }+2\tau
)-t^{\prime }\wedge T_{K}\right) =T_{K}-t^{\prime }\leq \tau .$ If on the
other hand $t^{\prime }\leq T_{K}$ and $\left( T_{K}\wedge (t^{\prime
}+2\tau )-t^{\prime }\wedge T_{K}\right) \leq \tau ,$ we get $\left(
T_{K}\wedge (t^{\prime }+2\tau )\right) \leq \tau +t^{\prime },$ which gives
that $T_{K}\leq (t^{\prime }+2\tau ),$ and therefore $T_{K}=\left(
T_{K}\wedge (t^{\prime }+2\tau )\right) \leq t^{\prime }+\tau $. Using the
claim just proved gives 
\begin{eqnarray*}
\mathbb{P}_{\boldsymbol{m}}(T_{K}\leq t^{\prime }+\tau |\mathcal{F}%
_{t^{\prime }}\!) \!=\!\mathbb{P}_{\boldsymbol{m}}\!(T_{K}\wedge (t^{\prime
}+2\tau )-t^{\prime }\wedge T_{K}\leq \tau |\mathcal{F}_{t^{\prime }}\!)
\!=\!\mathbb{P}_{\boldsymbol{m}}\!\!\left(\! e^{-nR_{\max }\left( T_{K}\wedge
(t^{\prime }+2\tau )-t^{\prime }\wedge T_{K}\right) }\!\geq\! e^{-nR_{\max}\tau
}|\mathcal{F}_{t^{\prime }}\!\right)\!\! .
\end{eqnarray*}%
Let $E_{1}\doteq \{e^{-nR_{\max }\left( T_{K}\wedge (t^{\prime }+2\tau
)-t^{\prime }\wedge T_{K}\right) }\geq e^{-nR_{\max }\tau }\}$ and $%
E_{2}\doteq E_{1}^{c}$. Then since $T_{K}\wedge (t^{\prime }+2\tau
)-t^{\prime }\wedge T_{K}\geq 0$ 
\begin{equation*}
\begin{split}
& \mathbb{E}_{\boldsymbol{m}}\left[ e^{-nR_{\max }\left( T_{K}\wedge
(t^{\prime }+2\tau )-t^{\prime }\wedge T_{K}\right) }\Big|\mathcal{F}%
_{t^{\prime }}\right] =\mathbb{E}_{\boldsymbol{m}}\left[ 1_{E_{1}}e^{-nR_{%
\max }^{n}\left( T_{K}\wedge (t^{\prime }+2\tau )-t^{\prime }\wedge
T_{K}\right) }\Big|\mathcal{F}_{t^{\prime }}\right] \\
& \quad \quad \quad +\mathbb{E}_{\boldsymbol{m}}\left[ 1_{E_{2}}e^{-nR_{\max
}^{n}\left( T_{K}\wedge (t^{\prime }+2\tau )-t^{\prime }\wedge T_{K}\right) }%
\Big|\mathcal{F}_{t^{\prime }}\right] \leq \mathbb{E}_{\boldsymbol{m}}\left[
1_{E_{1}}\Big|\mathcal{F}_{t^{\prime }}\right] +e^{-R_{\max }\tau }.
\end{split}%
\end{equation*}%
From this, the first part of the lemma and $e^{-nR_{\max }\tau }\leq c/2$,
we get 
\begin{equation*}
\begin{split}
 \mathbb{P}_{\boldsymbol{m}}\left( e^{-nR_{\max }\left( T_{K}\wedge
(t^{\prime }+2\tau )-t^{\prime }\wedge T_{K}\right) }\geq e^{-nR_{\max
}^{n}\tau }|\mathcal{F}_{t^{\prime }}\right)  \geq \mathbb{E}_{\boldsymbol{m}}\left[ e^{-nR_{\max }\left( T_{K}\wedge
(t^{\prime }+2\tau )-t^{\prime }\wedge T_{K}\right) }\Big|\mathcal{F}%
_{t^{\prime }}\right] -e^{-nR_{\max }\tau }\geq \frac{c}{2}.
\end{split}%
\end{equation*}
Now we have 
\vspace{-12pt}
\begin{equation*}
\begin{split}
\mathbb{P}_{\boldsymbol{m}}(T_{K}=\infty)&=\lim_{k\rightarrow\infty}\mathbb{P%
}_{\boldsymbol{m}}(T_{K}>k\tau)=\mathbb{P}_{\boldsymbol{m}%
}\left(T_{K}>0\right)\lim_{k\rightarrow\infty}\prod^{k}_{k^{\prime
}=0}\left(1-\mathbb{P}_{\boldsymbol{m}}\left(T_{K}\leq(k^{\prime }+1)\tau
\,|\, T_{K}>k^{\prime }\tau\right)\right) \\
&\leq \lim_{k\rightarrow\infty}\left(1-\frac{c}{2}\right)^k=0,
\end{split}%
\end{equation*}
where in the second inequality we iteratively used the formula for
conditional probability. The remaining inequality is just an application of
the monotone convergence theorem.
\end{proof}

\begin{lemma}
\label{eventuallydecreasingcontrols} Given $\boldsymbol{m}\in \mathcal{P}%
^{n}(\mathcal{X}),\epsilon >0$ and $\boldsymbol{u}\in \mathcal{A}_{b}^{n,|%
\mathcal{Z}|}$ with 
\begin{equation*}
\mathbb{E}_{\boldsymbol{m}}\left[ e^{n\int_{0}^{T_{K}}\left( \sum_{(x,y)\in 
\mathcal{Z}}\boldsymbol{\mu }_{x}(t)C_{xy}\left( \frac{u_{xy}(t)}{\gamma
_{xy}}\right) -R(\boldsymbol{\mu }(t))\right) dt}\right] <\infty ,
\end{equation*}%
there exists $\tilde{\boldsymbol{u}}\in \mathcal{A}_{b}^{n,|\mathcal{Z}|}$
and $\tau <\infty ,$ such that 
\begin{equation*}
\sum_{(x,y)\in \mathcal{Z}}\tilde{\mu}_{x}(t)\gamma _{xy}C_{xy}\left( \frac{%
\tilde{u}_{xy}(t)}{\gamma _{xy}}\right) -R(\tilde{\boldsymbol{\mu }}(t))\leq
0\hspace{8pt}\text{for every}\hspace{8pt} t>\tau,\hspace{16pt} \text{and} \hspace{16pt}
I_{K}^{n}(\boldsymbol{m},\tilde{\boldsymbol{u}})\leq I_{K}^{n}(\boldsymbol{m}%
,\boldsymbol{u})+\epsilon .
\end{equation*}
\end{lemma}

\begin{proof}
Let such $\boldsymbol{m}\in \mathcal{P}^{n}(\mathcal{X}),\epsilon >0,$ and $%
\boldsymbol{u}\in \mathcal{A}_{b}^{n,|\mathcal{Z}|}$ be given, and let $c>0$
from Lemma \ref{finitetime} be such that 
\begin{equation}
\mathbb{E}_{\boldsymbol{m}}\left[ e^{nR_{\max }(T_{K}-t^{\prime }\wedge
T_{K})}\Big|\mathcal{F}_{t^{\prime }}\right] \geq c  \label{eqn:TLB}
\end{equation}%
for $t^{\prime }\in \lbrack 0,\infty )$. Since by Lemma \ref{finitetime} $%
T_{K}$ is finite a.s., we can find $\tau <\infty $ such that 
\begin{equation*}
\mathbb{E}_{\boldsymbol{m}}\left[ I_{\{T_{K}\geq \tau
\}}e^{n\int_{0}^{T_{K}}\left( \sum_{(x,y)\in \mathcal{Z}}\mu _{x}(t)\gamma
_{xy}C_{xy}\left( \frac{u_{xy}(t)}{\gamma _{xy}}\right) -R(\boldsymbol{\mu }%
(t))\right) dt}\right] \leq \epsilon c.
\end{equation*}

Now set $\tilde{\boldsymbol{u}}(t)=\boldsymbol{u}(t)$ for $t\leq \tau ,$ and 
$\tilde{\boldsymbol{u}}(t)=\boldsymbol{\gamma }$ so that $C_{xy}\left( 
\tilde{u}_{xy}(t)/\gamma _{xy}\right) =0$ for $t\geq \tau $. Let $\tilde{%
\boldsymbol{\mu }}$ and $\tilde{T}_{K}$ be the corresponding controlled
process and stopping time. Then the first claim of the lemma follows. The
remaining claim follows from the following display, where the first
inequality uses again that $C_{xy}\left( 1\right) =0$, the following
equality uses that $(\tilde{\boldsymbol{u}},\tilde{\boldsymbol{\mu }},\tilde{%
T}_{K})$ had the same distribution as the original versions up till time $%
\tau $, and the second inequality uses (\ref{eqn:TLB}): 
\begin{align*}
I_{K}^{n}(\boldsymbol{m},\tilde{\boldsymbol{u}})& =\mathbb{E}_{\boldsymbol{m}%
}\left[ e^{n\int_{0}^{T_{K}}\left( \sum_{(x,y)\in \mathcal{Z}}\tilde{\mu}%
_{x}(t)\gamma _{xy}C_{xy}\left( \frac{\tilde{u}_{xy}(t)}{\gamma _{xy}}%
\right) -R(\tilde{\boldsymbol{\mu }}(t))\right) dt}\right] \\
& \leq \mathbb{E}_{\boldsymbol{m}}\left[ I_{\{T_{K}\leq \tau
\}}e^{n\int_{0}^{T_{K}}\left( \sum_{(x,y)\in \mathcal{Z}}\tilde{\mu}%
_{x}(t)\gamma _{xy}C_{xy}\left( \frac{\tilde{u}_{xy}(t)}{\gamma _{xy}}%
\right) -R(\tilde{\boldsymbol{\mu }}(t))\right) dt}\right] \\
& \quad +\mathbb{E}_{\boldsymbol{m}}\left[ I_{\{T_{K}\geq \tau
\}}e^{n\int_{0}^{T_{K}\wedge \tau }\left( \sum_{(x,y)\in \mathcal{Z}}\tilde{%
\mu}_{x}(t)\gamma _{xy}C_{xy}\left( \frac{\tilde{u}_{xy}(t)}{\gamma _{xy}}%
\right) -R(\tilde{\boldsymbol{\mu }}(t))\right) dt}\right] \\
& =\mathbb{E}_{\boldsymbol{m}}\left[ I_{\{T_{K}\leq \tau
\}}e^{n\int_{0}^{T_{K}}\left( \sum_{(x,y)\in \mathcal{Z}}\mu _{x}(t)\gamma
_{xy}C_{xy}\left( \frac{u_{xy}(t)}{\gamma _{xy}}\right) -R(\boldsymbol{\mu }%
(t))\right) dt}\right] \\
& \quad +\mathbb{E}_{\boldsymbol{m}}\Bigg[I_{\{T_{K}\geq \tau
\}}e^{n\int_{0}^{T_{K}\wedge \tau }\left( \sum_{(x,y)\in \mathcal{Z}}\mu
_{x}(t)\gamma _{xy}C_{xy}\left( \frac{u_{xy}(t)}{\gamma _{xy}}\right) -R(%
\boldsymbol{\mu }(t))\right) dt} \\
& \hspace{120pt}\times \frac{\mathbb{E}_{\boldsymbol{m}}\left[
e^{n\int_{T_{K}\wedge {\tau }}^{T_{K}}\left( \sum_{(x,y)\in \mathcal{Z}}{\mu 
}_{x}(t)\gamma _{xy}C_{xy}\left( \frac{u_{xy}(t)}{\gamma _{xy}}\right) -R({%
\boldsymbol{\mu }}(t))\right) dt}\Bigg|\mathcal{F}_{\tau }\right] }{\mathbb{E%
}_{\boldsymbol{m}}\left[ e^{n\int_{T_{K}\wedge {\tau }}^{T_{K}}\left(
\sum_{(x,y)\in \mathcal{Z}}{\mu }_{x}(t)\gamma _{xy}C_{xy}\left( \frac{%
u_{xy}(t)}{\gamma _{xy}}\right) -R({\boldsymbol{\mu }}(t))\right) dt}\Bigg|%
\mathcal{F}_{\tau }\right] }\Bigg] \\
& \leq \mathbb{E}_{\boldsymbol{m}}\left[ e^{n\int_{0}^{T_{K}}\left(
\sum_{(x,y)\in \mathcal{Z}}\mu _{x}(t)C_{xy}\left( \frac{u_{xy}(t)}{\gamma
_{xy}}\right) -R(\boldsymbol{\mu }(t))\right) dt}\right] \\
& \quad +\frac{1}{c}\mathbb{E}_{\boldsymbol{m}}\Bigg[I_{\{T_{K}\geq \tau
\}}e^{n\int_{0}^{T_{K}}\left( \sum_{(x,y)\in \mathcal{Z}}\mu _{x}(t)\gamma
_{xy}C_{xy}\left( \frac{u_{xy}(t)}{\gamma _{xy}}\right) -R(\boldsymbol{\mu }%
(t))\right) dt}\Bigg]  \leq I_{K}^{n}(\boldsymbol{m},\boldsymbol{u})+\epsilon .
\end{align*}
\end{proof}

\medskip

\begin{proof}[Proof of Theorem \protect\ref{Main A}]
We are first going to prove that $V_{K}^{n}$ is the unique solution to (\ref%
{prelimeq}). We will prove that, by showing that if $\tilde{V}$ is any solution to (\ref{prelimeq}), then it has to coincide with $V_{K}^{n}.$ Let $\tilde{V}$ be any solution to (\ref{prelimeq}), and let $%
\boldsymbol{m}\in \mathcal{P}(\mathcal{X}).$ Let also $\boldsymbol{q}\in 
\mathcal{A}_{b}^{n,|\mathcal{Z}|}$ be given and let $\boldsymbol{\mu }$
solve (\ref{eqdef:calTn}). By Lemma \ref{martingale}, 
\begin{equation*}
\tilde{V}(\boldsymbol{\mu }(t\wedge T_{K}))-\tilde{V}(\boldsymbol{m}%
)-\int_{0}^{t\wedge {T_{K}}}\sum_{(x,y)\in \mathcal{Z}}\mu
_{x}(s)q_{xy}(s)\Delta ^{n}\tilde{V}(\boldsymbol{\mu }(s))ds
\end{equation*}%
is a martingale. Taking expectation gives%
\begin{equation*}
\mathbb{E}_{\boldsymbol{m}}\left[ \tilde{V}(\boldsymbol{\mu }(t\wedge T_{K}))%
\right] -\mathbb{E}_{\boldsymbol{m}}\left[ \int_{0}^{t\wedge {T_{K}}%
}\sum_{(x,y)\in \mathcal{Z}}\mu _{x}(s)q_{xy}(s)\Delta ^{n}\tilde{V}(%
\boldsymbol{\mu }(s))ds\right] =\tilde{V}(\boldsymbol{m}),
\end{equation*}%
and since $\tilde{V}$ is a solution to (\ref{prelimeq}) and by (\ref%
{eqn:defofHn}),%
\begin{equation*}
\mathbb{E}_{\boldsymbol{m}}\left[ \tilde{V}(\boldsymbol{\mu }(t\wedge T_{K}))%
\right] +\mathbb{E}_{\boldsymbol{m}}\left[ \int_{0}^{t\wedge {T_{K}}}\left(
\sum_{(x,y)\in \mathcal{Z}}\mu _{x}(s)F_{xy}(q_{xy}(s))+R(\boldsymbol{\mu }%
(s))\right) ds\right] \geq \tilde{V}(\boldsymbol{m}).
\end{equation*}%
By Lemma \ref{finitetime}, $\,T_{K}<\infty $ almost surely. Letting $%
t\rightarrow \infty $, Lemma \ref{lemma:fnproperties} and the monotone
convergence theorem imply 
\begin{equation*}
J_{K}^{n}(\boldsymbol{m},\boldsymbol{q})=\mathbb{E}_{\boldsymbol{m}}\left[
\int_{0}^{T_{K}}\sum_{(x,y)\in \mathcal{Z}}\mu _{x}(t)F_{xy}(q_{xy}(s))+R(%
\boldsymbol{\mu }(s))ds\right] \geq \tilde{V}(\boldsymbol{m}).
\end{equation*}%
Since $\boldsymbol{q}\in \mathcal{A}_{b}^{n,|\mathcal{Z}|}$ was arbitrary we
get $V_{K}^{n}(\boldsymbol{m})\geq \tilde{V}(\boldsymbol{m}).$ We will now prove the opposite inequality. Let $\epsilon>0.$ For $\boldsymbol{m}\in \mathcal{P}^{n}(\mathcal{X}),$ we can find $\bar{ \boldsymbol{q}}(%
\boldsymbol{m})$ that satisfies  
\begin{equation}
\sum_{(x,y)\in \mathcal{Z}}\!\!\left(\! \bar{q}_{xy}(\boldsymbol{m})n\left(\! \tilde{%
V}\left( \boldsymbol{m}{+}\frac{1}{n}v_{xy}\right) -\tilde{V}(\boldsymbol{m}%
))\right) +m_{x}F_{xy}(\bar{q}_{xy}(\boldsymbol{m}))\right) +R(\boldsymbol{m}%
)\leq \epsilon\hspace{-8pt} \sum_{(x,y)\in \mathcal{Z}} m_{x}F_{xy}(\bar{q}%
_{xy}(\boldsymbol{m})).  \label{eoptimal}
\end{equation}

To see that such a $\bar{\boldsymbol{q}}(\boldsymbol{m})$ exists and it is
actually bounded away from zero, we take a minimizing sequence $\bar{\boldsymbol{q}}_{n}(\boldsymbol{m})$ in the definition of $H\left( \boldsymbol{m},\Delta ^{n}\tilde{V}(\boldsymbol{m})\right)$ (see (\ref{eqn:defofHn})).  By using the continuity of the function $\sum_{(x,y)\in \mathcal{Z}}\!\!\left(\! \bar{q}_{xy}(\boldsymbol{m})n\left(\! \tilde{%
	V}\left( \boldsymbol{m}{+}\frac{1}{n}v_{xy}\right) -\tilde{V}(\boldsymbol{m}%
))\right) +m_{x}F_{xy}(\bar{q}_{xy}(\boldsymbol{m}))\right) $ with respect to $\bar{\boldsymbol{q}}_{n}(\boldsymbol{m}),$ we can assume that all $q_{xy,n}$ are strictly positive. Furthermore, with no loss of
generality we can assume that the sequence is converging. If all elements
converge to the original rates, by recalling \eqref{prelimeq}, we notice that we can just take those and the inequality is
satisfied trivially. If on the other hand it converges to different values
the right hand will be always bounded away from zero while the left hand
will converge to zero by \eqref{prelimeq}, therefore for sufficiently large value of $n,$ we
will recover the desired control. We can construct a solution to (\ref%
{eqdef:calTn}) with $\boldsymbol{u}$ replaced by the feedback control $\bar{%
\boldsymbol{q}}(\boldsymbol{\mu })$, and then obtain $\hat{\boldsymbol{q}}%
\in \mathcal{A}_{b}^{|\mathcal{Z}|}$ by setting $\hat{\boldsymbol{q}}(t)=%
\bar{\boldsymbol{q}}(\boldsymbol{\mu }(t)).$
Then%
\begin{equation*}
\mathbb{E}_{\boldsymbol{m}}\left[ \tilde{V}(\boldsymbol{\mu }(t\wedge T_{K}))%
\right] -\mathbb{E}_{\boldsymbol{m}}\left[ \int_{0}^{t\wedge {T_{K}}%
}\sum_{(x,y)\in \mathcal{Z}}\mu _{x}(s)\bar{q}_{xy}(\boldsymbol{\mu }%
(s))\Delta ^{n}\tilde{V}(\boldsymbol{\mu }(s))ds\right] =\tilde{V}(%
\boldsymbol{m}),
\end{equation*}%
and therefore by (\ref{eoptimal})%
\begin{equation*}
\mathbb{E}_{\boldsymbol{m}}\left[ \tilde{V}(\boldsymbol{\mu }(t\wedge T_{K}))%
\right] +\mathbb{E}_{\boldsymbol{m}}\left[ \int_{0}^{t\wedge
T_{K}}\left((1-\epsilon) \sum_{(x,y)\in \mathcal{Z}}\mu _{x}(t)F^{n}(\bar{q}%
_{xy}(\boldsymbol{\mu }(s)))+R(\boldsymbol{\mu }(s))\right) ds\right] \leq 
\tilde{V}(\boldsymbol{m}).
\end{equation*}%
Again using Lemma \ref{finitetime} and the monotone convergence theorem
gives 
\begin{equation*}
(1-\epsilon )\mathbb{E}_{\boldsymbol{m}}\left[ \int_{0}^{{T_{K}}}\left(
\sum_{(x,y)\in \mathcal{Z}}\mu _{x}(t)F^{n}(\bar{q}_{xy}(\boldsymbol{\mu }%
(s)))+R(\boldsymbol{\mu }(s))\right) ds\right] \leq \tilde{V}(\boldsymbol{m}%
),
\end{equation*}%
and therefore $V_{K}^{n}(\boldsymbol{m})\leq J_{K}^{n}(\boldsymbol{m},\hat{%
\boldsymbol{q}})\leq \frac{1}{1-\epsilon }\tilde{V}(\boldsymbol{m}).$ Since $%
\epsilon $ is arbitrary we get $V_{K}^{n}(\boldsymbol{m})=\tilde{V}(%
\boldsymbol{m})$, which implies the uniqueness of $\tilde{V}$. We now
proceed with the proof that $W_{K}^{n}$ is the unique solution to 
\begin{equation}
\sup_{\boldsymbol{\ u}\in (0,\infty )^{|\mathcal{Z}|}}\left\{ \sum_{(x,y)\in 
\mathcal{Z}}\mu _{x}\left( u_{xy}\left( \frac{W(\boldsymbol{\mu })-W\left( 
\boldsymbol{\mu }+\frac{\boldsymbol{v}_{xy}}{n}\right) }{W(\boldsymbol{\mu })%
}\right) -\gamma _{xy}C_{xy}\left( \frac{u_{xy}}{\gamma _{xy}}\right)
\right) \right\} =-R(\boldsymbol{\mu }).  \label{fornow}
\end{equation}%
Since $V_{K}^{n}$ is a solution to (\ref{prelimeq}), by Lemma \ref{iff} we
get that $\frac{1}{n}\log (V_{K}^{n})$ is a solution to (\ref{fornow}), and
thus uniqueness will imply $\frac{1}{n}\log (V_{K}^{n})=W_{K}^{n}.$ Let 
$\tilde{W}$ be any solution to (\ref{fornow}), $m\in \mathcal{P}^{n}(%
\mathcal{X})$, and $\boldsymbol{u}\in \mathcal{A}_{b}^{n,|\mathcal{Z}|}$,
and let $\boldsymbol{\mu }$ solve (\ref{eqdef:calTn}). Further assume that
there exists $\tau <\infty $ such that for $t>\tau $ 
\begin{equation}
\sum_{(x,y)\in \mathcal{Z}}\mu _{x}(t)\gamma _{xy}C_{xy}\left( \frac{%
u_{xy}(t))}{\gamma _{xy}}\right) -R(\boldsymbol{\mu }(t))\leq 0.
\label{eventnegative}
\end{equation}%
To show $J_{K}^{n}(\boldsymbol{m},\boldsymbol{u})\geq \tilde{W}(\boldsymbol{m%
})$ we can assume that $J_{K}^{n}(\boldsymbol{m},\boldsymbol{u})<\infty $,
since otherwise there is nothing to prove. By Lemma \ref{expmartingale}%
\begin{equation*}
\frac{\tilde{W}(\boldsymbol{\mu }(t\wedge T_{K}))}{\tilde{W}(\boldsymbol{m})}%
\exp \left\{ -\int_{0}^{t\wedge {T_{K}}}\sum_{(x,y)\in \mathcal{Z}}\mu
_{x}(s)u_{xy}(s)\frac{\Delta ^{n}\tilde{W}(\boldsymbol{\mu }(s))}{\tilde{W}(%
\boldsymbol{\mu }(s))}ds\right\}
\end{equation*}%
is a martingale. Taking expectations gives%
\begin{equation*}
\mathbb{E}_{\boldsymbol{m}}\left[ \tilde{W}(\boldsymbol{\mu }(t\wedge
T_{K}))\exp \left\{ -\int_{0}^{t\wedge {T_{K}}}\sum_{(x,y)\in \mathcal{Z}%
}\mu _{x}(s)u_{xy}(s)\frac{\Delta ^{n}\tilde{W}(\boldsymbol{\mu }(s))}{%
\tilde{W}(\boldsymbol{\mu }(s))}ds\right\} \right] =\tilde{W}(\boldsymbol{m}%
),
\end{equation*}%
and by (\ref{prelimeq}) and the definition of $\Delta ^{n}$ 
\begin{equation*}
\mathbb{E}_{\boldsymbol{m}}\left[ \tilde{W}(\boldsymbol{\mu }(t\wedge
T_{K}))\exp \left\{ n\int_{0}^{t\wedge T_{K}}\left( \sum_{(x,y)\in \mathcal{Z%
}}\mu _{x}(s)\gamma _{xy}C_{xy}\left( \frac{u_{xy}(s)}{\gamma _{xy}}\right)
-R(\boldsymbol{\mu }(s))\right) ds\right\} \right] \geq \tilde{W}(%
\boldsymbol{m}).
\end{equation*}

We claim that 
\begin{equation}
\mathbb{E}_{\boldsymbol{m}}\left[ \tilde{W}(\boldsymbol{\mu }(t\wedge
T_{K}))\exp \left\{ n\int_{0}^{\tau \wedge T_{K}}\left( \sum_{(x,y)\in 
\mathcal{Z}}\mu _{x}(s)\gamma _{xy}C_{xy}\left( \frac{u_{xy}(s)}{\gamma _{xy}%
}\right) -R(\boldsymbol{\mu }(s))\right) ds\right\} \right] <\infty .
\label{eqn:partofrep}
\end{equation}%
Since $\tilde{W}$ is uniformly bounded this term can be ignored. One can
then bound what remains in (\ref{eqn:partofrep}) by using 
\begin{equation*}
\infty >J_{K}^{n}(\boldsymbol{m},\boldsymbol{u})=\mathbb{E}_{\boldsymbol{m}}%
\left[ \exp \left\{ n\int_{0}^{T_{K}}\left( \sum_{(x,y)\in \mathcal{Z}}\mu
_{x}(s)\gamma _{xy}C_{xy}\left( \frac{u_{xy}(s)}{\gamma _{xy}}\right) -R(%
\boldsymbol{\mu }(s))\right) ds\right\} \right] ,
\end{equation*}%
breaking the integral over $[0,T_{K}]$ into contributions over $[0,\tau
\wedge T_{K}]$ and $[\tau \wedge T_{K},T_{K}]$, and then conditioning on $%
\mathcal{F}_{\tau }$ and using the lower bound on the term corresponding to $%
[\tau \wedge T_{K},T_{K}]$ provided by Lemma \ref{finitetime} (as in the
proof of Lemma \ref{eventuallydecreasingcontrols}). Since (by Lemma \ref%
{finitetime}) $T_{K}$ is finite almost surely, and (\ref{eventnegative})
holds for $t\geq \tau $, by dominated convergence theorem and (\ref%
{eqn:partofrep}) it follows that%
\begin{equation*}
J_{K}^{n}(\boldsymbol{m},\boldsymbol{u})=\mathbb{E}\left[ \exp \left\{
n\int_{0}^{{T_{K}}}\left( \sum_{(x,y)\in \mathcal{Z}}\mu _{x}(s)\gamma
_{xy}C_{xy}\left( \frac{u_{xy}(s)}{\gamma _{xy}}\right) -R(\boldsymbol{\mu }%
(s))\right) ds\right\} \right] \geq \tilde{W}(\boldsymbol{m}).
\end{equation*}%
By minimizing over all $\boldsymbol{u}$ that satisfy (\ref{eventnegative})
and applying Lemma \ref{eventuallydecreasingcontrols}, we get $W_{K}^{n}(%
\boldsymbol{m})\geq \tilde{W}(\boldsymbol{m}).$

Next let $\epsilon \in (0,1/2).$ For $\boldsymbol{m}\in \mathcal{P}^{n}(%
\mathcal{X}),t\geq 0$ we choose $\bar{\boldsymbol{u}}(\boldsymbol{m},t)$
such that 
\begin{equation}
\sum_{(x,y)\in \mathcal{Z}}m_{x}\left( \bar{u}_{xy}(\boldsymbol{m},t)\left( 
\frac{\tilde{W}(\boldsymbol{m})-\tilde{W}\left( \boldsymbol{m}+\frac{%
\boldsymbol{v}_{xy}}{n}\right) }{\tilde{W}(\boldsymbol{m})}\right) -\gamma
_{xy}C_{xy}\left( \frac{\bar{u}_{xy}(\boldsymbol{m},t)}{\gamma _{xy}}\right)
\right) \geq -R(\boldsymbol{m})-\frac{\epsilon }{t^{2}+1}.  \label{eoptimalW}
\end{equation}%
As before we can solve (\ref{eqdef:calTn}) and then generate a corresponding
element $\boldsymbol{u}$ of $\mathcal{A}_{b}^{n,|\mathcal{Z}|}$ by composing 
$\bar{u}_{xy}(\boldsymbol{m},t)$ with the solution. It is easy to see that $%
\boldsymbol{u}$ is an element of $\mathcal{A}_{b}^{n,|\mathcal{Z}|}$, since
very big or very small values of $\bar{u}_{xy}(\boldsymbol{m},t)$ will make
the left hand of (\ref{eoptimalW}) tend to $-\infty $. Arguing as before,
for fixed $t<\infty $%
\begin{equation*}
\mathbb{E}_{\boldsymbol{m}}\left[ \tilde{W}(\boldsymbol{\mu }(t\wedge
T_{K}))\exp \left\{ n\int_{0}^{{T_{K}\wedge t}}\left( \sum_{(x,y)\in 
\mathcal{Z}}\mu _{x}(s)\gamma _{xy}C_{xy}\left( \frac{\bar{u}_{xy}(%
\boldsymbol{\mu }(s),s)}{\gamma _{xy}}\right) -R(\boldsymbol{\mu }(s))-\frac{%
\epsilon }{s^{2}+1}\right) ds\right\} \right] \leq \tilde{W}(\boldsymbol{m}).
\end{equation*}%
By sending $t\rightarrow \infty $ and using the boundary condition, Fatou's
lemma gives 
\begin{equation*}
\mathbb{E}_{\boldsymbol{m}}\left[ \exp \left( \int_{0}^{\infty }-\frac{%
\epsilon }{s^{2}+1}ds\right) \exp \left\{ n\int_{0}^{{T_{K}}}\left(
\sum_{(x,y)\in \mathcal{Z}}\mu _{x}(s)\gamma _{xy}C_{xy}\left( \frac{\bar{u}%
_{xy}(\boldsymbol{\mu }(s),s)}{\gamma _{xy}}\right) -R(\boldsymbol{\mu }%
(s))\right) ds\right\} \right] \leq \tilde{W}(\boldsymbol{m}),
\end{equation*}%
from which we get $W_{K}^{n}(\boldsymbol{m})\leq \tilde{W}(\boldsymbol{m}%
)\exp [\epsilon \int_{0}^{\infty }1/(s^{2}+1)ds].$ Sending $\epsilon $ to
zero shows $W_{K}^{n}(\boldsymbol{m})\leq \tilde{W}(\boldsymbol{m})$.

The proof that $\mathcal{V}_{K}^{n}(\boldsymbol{x}^{n})=-\frac{1}{n}\log (%
\mathcal{W}_{K}^{n}(\boldsymbol{x}^{n}))$ is similar and thus omitted. It
remains only to prove $V_{K}^{n}(L(\boldsymbol{x}^{n}))=\mathcal{V}_{K}^{n}(%
\boldsymbol{x}^{n})$. We have established that $V_{K}^{n}$ is the only
function that satisfies 
\begin{equation*}
\inf_{\boldsymbol{q}\in (0,\infty )^{|\mathcal{Z}|}}\left\{ \sum_{(x,y)\in 
\mathcal{Z}}m_{x}\left( q_{xy}\Delta _{xy}^{n}V_{K}^{n}\left( \boldsymbol{m}%
\right) +F_{xy}(q_{xy})\right) \right\} =-R(\boldsymbol{m}), 
\end{equation*}
and that $\mathcal{V}_{K}^{n}$ is the only function that satisfies%
\begin{equation}
\inf_{\boldsymbol{q}\in (0,\infty )^{n|\mathcal{Z}|}}\left\{
\sum_{i=1}^{n}\sum_{y\in \mathcal{Z}_{x_{i}^{n}}}\left( q_{x_{i}^{n}y}\Delta
_{i,x_{i}^{n}y}^{n}\mathcal{V}_{K}^{n}\left( \boldsymbol{x}^{n}\right)
+F_{x_{i}^{n}y}(q_{x_{i}^{n}y})\right) \right\} =-nR(L(\boldsymbol{x}^{n})).
\label{xhjb}
\end{equation}%
Since $K\subset \mathcal{X}^{n}$ is invariant under permutations, and
therefore can be identified with a subset of $\mathcal{P}^{n}(\mathcal{X}),$
we have that there exists a function $\bar{V}:\mathcal{P}^{n}(\mathcal{X}%
)\rightarrow \lbrack 0,\infty )$ such that $\bar{V}(L(\boldsymbol{x}^{n}))=%
\mathcal{V}_{K}^{n}(\boldsymbol{x}^{n}),$ and therefore (\ref{xhjb}) becomes 
\begin{equation*}
\inf_{\boldsymbol{q}\in (0,\infty )^{n|\mathcal{Z}|}}\left\{
\sum_{i=1}^{n}\sum_{y\in \mathcal{Z}_{x_{i}^{n}}}\left( q_{x_{i}^{n}y}\Delta
_{i,x_{i}^{n}y}^{n}\bar{V}\left( L(\boldsymbol{x}^{n})\right)
+F_{x_{i}^{n}y}(q_{x_{i}^{n}y})\right) \right\} =-nR(L(\boldsymbol{x}^{n})).
\end{equation*}

For $\epsilon >0$, let $\bar{\boldsymbol{q}}\in (0,\infty )^{n|\mathcal{Z}|}$
satisfy 
\begin{equation*}
\sum_{i=1}^{n}\sum_{y\in \mathcal{Z}_{x_{i}^{n}}}\left[ \bar{q}%
_{x_{i}^{n}y}\Delta _{i,x_{i}^{n}y}^{n}\bar{V}\left( L(\boldsymbol{x}%
^{n})\right) +F_{x_{i}^{n}y}(\bar{q}_{x_{i}^{n}y})\right] \leq -nR(L(%
\boldsymbol{x}^{n}))+\epsilon .
\end{equation*}%
Now pick $\tilde{\boldsymbol{q}}\in (0,\infty )^{|\mathcal{Z}|}$ by
requiring $nL_{x}(\boldsymbol{x}^{n})\tilde{q}_{xy}=%
\sum_{i=1}^{n}I_{x_{i}^{n}=x}\bar{q}_{x_{i}^{n}y}$, so that 
\begin{equation*}
\sum_{(x,y)\in \mathcal{Z}}nL_{x}(\boldsymbol{x}^{n})\tilde{q}_{xy}\Delta
_{xy}^{n}\bar{V}\left( L(\boldsymbol{x}^{n})\right)
+\sum_{i=1}^{n}\sum_{y\in \mathcal{Z}_{x_{i}^{n}}}F_{x_{i}^{n}y}(\bar{q}%
_{x_{i}^{n}y})\leq -nR(L(\boldsymbol{x}^{n}))+\epsilon .
\end{equation*}%
By using convexity of $F_{xy}$ (see Lemma \ref{lemma:fnproperties}) we get%
\begin{equation*}
\sum_{(x,y)\in \mathcal{Z}}L_{x}(\boldsymbol{x}^{n})\left[ \tilde{q}%
_{xy}\Delta _{xy}^{n}\bar{V}\left( L(\boldsymbol{x}^{n})\right) +F_{xy}(%
\tilde{q}_{xy})\right] \leq -R(L(\boldsymbol{x}^{n}))+\epsilon /n,
\end{equation*}%
and sending $\epsilon \downarrow 0$ gives%
\begin{equation*}
\inf_{\boldsymbol{q}\in (0,\infty )^{|\mathcal{Z}|}}\left\{ \sum_{(x,y)\in 
\mathcal{Z}}L_{x}(\boldsymbol{x}^{n})\left[ q_{xy}\Delta _{xy}^{n}\bar{V}%
\left( L(\boldsymbol{x}^{n})\right) +F_{xy}(q_{xy})\right] \right\} \leq
-R(L(\boldsymbol{x}^{n})).
\end{equation*}

The other direction is trivial, and follows if in (\ref{xhjb}) one uses
rates that are the same for all agents in the same position.
\end{proof}

\section{Discussion regarding convergence}

Before we introduce the deterministic control problem, we define the set of
admissible controls and controlled trajectories.

\begin{definition}
We define the space of paths and controls by%
\begin{equation*}
\mathcal{C}\doteq \left\{ (\boldsymbol{\mu },\boldsymbol{q})\in \mathcal{D}%
([0,\infty );\mathcal{P}(\mathcal{X}))\times \mathcal{F}\left( [0,\infty
);[0,\infty )^{\otimes \mathcal{Z}}\right) :\,\mu _{x}q_{xy}\,\text{is
locally integrable }\forall (x,y)\in \mathcal{Z}\right\} ,
\end{equation*}
where $\mathcal{F}\left( [0,\infty );[0,\infty )^{\otimes \mathcal{Z}%
}\right) $ was defined in (\ref{measurable}). We define $\Lambda :\mathcal{C}%
\times \mathcal{P}(\mathcal{X})\rightarrow \mathcal{D}([0,\infty );\mathcal{H%
})$ by 
\begin{equation}
\Lambda (\boldsymbol{\mu },\boldsymbol{q},\boldsymbol{m})(t)\doteq 
\boldsymbol{m}+\sum_{(x,y)\in \mathcal{Z}}\boldsymbol{v}_{xy}\int_{[0,t)}\mu
_{x}(s)q_{xy}(s)ds.  \label{eqdef:calT}
\end{equation}%
Also we define the set of all deterministic pairs that correspond to a
solution of the equation $\boldsymbol{\mu }=\Lambda (\boldsymbol{\mu },%
\boldsymbol{q},\boldsymbol{m}),$ i.e., 
\begin{equation*}
\mathcal{T}_{\boldsymbol{m}}\doteq \{(\boldsymbol{\mu },\boldsymbol{q})\in 
\mathcal{C}:\boldsymbol{\mu }=\Lambda (\boldsymbol{\mu },\boldsymbol{q},%
\boldsymbol{m}),\boldsymbol{\mu }(0)=\boldsymbol{m}\}
\end{equation*}%
Finally we introduce the set of controls that generate controlled
trajectories%
\begin{equation}
\mathcal{U}_{\boldsymbol{m}}\doteq \left\{ \boldsymbol{q}\in \mathcal{F}%
([0;\infty );[0,\infty )^{\otimes \mathcal{Z}}):\exists \boldsymbol{\mu }\in 
\mathcal{D}([0,\infty );\mathcal{P}(\mathcal{X}))\hspace{4pt}\text{such that}%
\hspace{4pt}(\boldsymbol{\mu },\boldsymbol{q})\in \mathcal{T}_{\boldsymbol{m}%
}\right\} .  \label{eqdef:calUm}
\end{equation}
\end{definition}

Then the deterministic control problems are given by

\begin{equation}
V_{K}(\boldsymbol{m})\doteq \inf_{(\boldsymbol{\mu },\boldsymbol{q})\in 
\mathcal{T}_{\boldsymbol{m}}}J_{K}(\boldsymbol{m},\boldsymbol{\mu },%
\boldsymbol{q}),  \label{eqdef:V}
\end{equation}
with 
\begin{equation*}
J_{K}(\boldsymbol{m},\boldsymbol{\mu },\boldsymbol{q})\doteq \left\{
\int_{0}^{T_{K}}\left( \sum_{(x,y)\in \mathcal{Z}}\mu
_{x}(t)F_{xy}(q_{xy}(t))+R(\boldsymbol{\mu }(t))\right) dt\right\} , \hspace{%
8pt} T_{K}\doteq \inf_{t\in \lbrack 0,\infty ]}\left\{ \boldsymbol{\mu }%
(t)\in K\right\}.  
\end{equation*}

In this section we consider sets $K\subset {\mathcal{P}(\mathcal{X})}$ that
satisfy the following assumption.

\begin{assumption}
\label{assumption:o} $K=\overline{K^{\circ }}\neq \emptyset $.
\end{assumption}

For such sets we show that the sequence of values functions $V_{K}^{n}$
converges uniformly to the function $V_{K}$. To simplify the notation we
will drop the index that corresponds to the set from the stopping time. We
split the study of the convergence in two parts. In the first part, without
making any extra assumptions on the cost functions and in great generality,
we prove that for any sequence $\{\boldsymbol{m}^{n}\},$ with $\boldsymbol{m}%
^{n}\in \mathcal{P}^{n}(\mathcal{X})$ converging in $\boldsymbol{m}\in 
\mathcal{P}(\mathcal{X})$,%
\begin{equation*}
\liminf_{n\rightarrow \infty }V_{K}^{n}(\boldsymbol{m}^{n})\geq V_{K}(%
\boldsymbol{m}).
\end{equation*}%
The other direction of the inequality, i.e., $\limsup_{n\rightarrow \infty
}V_{K}^{n}(\boldsymbol{m}^{n})\leq V_{K}(\boldsymbol{m}),$ is not as
straightforward and its analysis can be quite involved. In order to avoid
technical issues relating to controllability we will add some assumptions.

Before we present the extra assumptions on $\boldsymbol{C}$ we discuss an
almost trivial choice for the cost function that will motivate these extra
assumptions. As stated in Lemma \ref{lemos}, for every $(x,y)\in \mathcal{Z}$
we have $C_{xy}(u)\geq -\log u+u-u.$ Actually the function $C_{xy}(u)=-\log
u+u-1$ satisfies Assumption \ref{assumption} and therefore is an eligible
cost function. Setting $C_{xy}(u)\equiv C(u)=-\log u+u-1,$ we get 
\begin{eqnarray}
G_{xy}(u,q)&=&u\ell \left( \frac{q}{u}\right) -\gamma _{xy}C_{xy}\left( 
\frac{u}{\gamma _{xy}}\right) = q\log \frac{q}{u}-q+u+\gamma _{xy}\log \frac{%
u}{\gamma _{xy}}-u+\gamma _{xy}  \label{eqn:specialcost} \\
&=&q\log q+(\gamma _{xy}-q)\log u-q+\gamma _{xy}.  \notag
\end{eqnarray}

Examining (\ref{eqn:specialcost}) and referring to the definition of $F_{xy}$
in (\ref{eqdef:Fxy,Gxy}), we observe that if $q_{xy}>\gamma _{xy} $ then the
\textquotedblleft maximizing player\textquotedblright\ (the one that picks $%
\boldsymbol{u}$), can produce an arbitrarily large cost by making $u_{xy}$
as \textbf{small} as needed. If $q_{xy}<\gamma _{xy}$, this player can
produce an arbitrarily large cost by making $u_{xy}$ as \textbf{big} as
needed. Hence the minimizing player must keep $q_{xy}=\gamma _{xy}$, and the
value function $V(\boldsymbol{m})$ is infinite unless the solution of the
equation $\dot{\boldsymbol{\nu }}(t)=\boldsymbol{\nu }(t)\boldsymbol{\gamma }
$ passes through $K$ for the specific choice of initial data $\boldsymbol{m}%
. $ To resolve this difficulty we could start by imposing the following
assumption on the cost.

\begin{equation*}
\lim_{u\rightarrow 0}u(C_{xy})^{\prime }(u)=-\infty, 
\hspace{16pt} \liminf_{u\rightarrow \infty}\{u
(C_{xy})^{\prime }(u)-u\}\geq 0.
\end{equation*}

This assumption makes $F$ finite on $(0,\infty )$ and allows for some
controllability. Specifically, if the first point is true and if $%
\boldsymbol{m},\tilde{\boldsymbol{m}}\in \mathcal{P}_{a}(\mathcal{X})$ for
some $a>0$, then one can observe (see the proof of Lemma \ref%
{lemma:fnproperties}) that the total cost $V_{\{\bar{\boldsymbol{m}}\}}(%
\boldsymbol{m})$ for moving from point $\boldsymbol{m}$ to $\tilde{%
\boldsymbol{m}}$ is uniformly bounded by $c_{a}\Vert \boldsymbol{m}-\tilde{%
\boldsymbol{m}}\Vert ,$ where $c_{a}>0$ is an appropriate constant, where
the minimizing player picks $\tilde{q}_{xy}(t) $ to be uniformly bounded
from above, but big enough to reach the desired point. In particular, the
maximizing player cannot impose an arbitrarily large cost by taking $u_{xy}$
small. In an analogous fashion, the second point implies that the minimizer
can choose controls so that the total cost $V_{\{\tilde{\boldsymbol{m}}\}}(%
\boldsymbol{m})$ for moving from point $\boldsymbol{m}$ to $\tilde{%
\boldsymbol{m}}$ is uniformly bounded by $c_{a}^{\prime }\Vert \boldsymbol{m}%
-\tilde{\boldsymbol{m}}\Vert $ by picking $\tilde{q}_{xy}(t)$ bounded from
below but small enough.

However, if $\tilde{\boldsymbol{m}}$ is in the natural boundary of the
simplex $\mathcal{P}(\mathcal{X})$ an additional complication arises,
because to reach the natural boundary it must be true that for at least one $%
(x,y)\in \mathcal{Z}$ the quantity $\tilde{q}_{xy}(t)$ will scale like $1/%
\tilde{\mu}_{x}(t)$. In that case, the first point is not enough for a
finite cost, since sending $\tilde{q}_{xy}(t)$ to infinity in order to reach
the natural boundary may result in an infinite total cost. Taking all these
issues into account we end up with the following assumption.
\begin{assumption}
\label{assumption:uh} Let $\boldsymbol{C},R$ be as in Assumption \ref%
{assumption}. Assume that for all $(x,y)\,\in \mathcal{Z}$, the following
are valid.

\begin{enumerate}
\item \label{enumerate:uhgrowthofh} There exists $p>0$ such that 
\begin{equation*}
\lim_{u\rightarrow 0}u^{p+1}C_{xy}^{\prime }(u)=-\infty .
\end{equation*}

\item \begin{equation*}
\liminf_{u\rightarrow \infty}\{u C_{xy}^{\prime }(u)-u\}\geq 0.
\end{equation*}
\end{enumerate}
\end{assumption}

It is straightforward to see that Assumption \ref{assumption:uh} is satisfied by all functions in Example \ref{example}, with $p,q>1.$ Now we state the second main theorem of the paper.
\begin{theorem}
\label{main2} Let $\boldsymbol{C}, R,$ satisfy Assumption \ref{assumption:uh}. Let also $K$ be a closed subset of $\mathcal{P}(%
\mathcal{X}) $ that satisfies Assumption \ref{assumption:o}. Finally assume
that in every compact subset of $K^{c},$ $R$ is bounded from below by a
positive constant. Then the sequence of functions $V_{K}^{n}$ defined in (%
\ref{eqdef:Vn}) converges uniformly to $V_{K}$ defined in (\ref{eqdef:V}).
\end{theorem}

Before proceeding with the proof, we state some properties of $F_{xy}$.

\begin{lemma}
\label{lemma:fnproperties2} For every $(x,y)\in \mathcal{Z},$ let $F_{xy}$
be as in (\ref{eqdef:Fxy,Gxy}), where $C_{xy}$ satisfy Assumption \ref%
{assumption:uh}. Then the following hold.

\begin{enumerate}
\item There exists a constant $M\in (0,\infty )$ and a decreasing function $%
\bar{M}:(0,\infty )\rightarrow (0,\infty ),$ such that for every $\epsilon
>0 $ and every $q\geq \epsilon $, 
\begin{equation*}
F_{xy}(q)\leq q\log \frac{q}{\min \left\{ \gamma _{xy}\left( \gamma
_{xy}/q\right) ^{1/p},M\right\} }+\bar{M}(\epsilon ).
\end{equation*}

\item $F_{xy}$ is continuous on the interval $(0,\infty ).$
\end{enumerate}
\end{lemma}

The proof of the Lemma \ref{lemma:fnproperties2} can be found in Appendix \ref{ap:b}.
It is worth mentioning that it is possible that $F_{xy}(0)=\infty.$ In
the sequel we will make use of the following remark, which states a property
proved in \cite[Proposition 4.14]{Dupuis2016}

\begin{remark}
\label{away} There exists $D\geq 1$ and $b_{1}>0,b_{2}<\infty $ such that
for every $\boldsymbol{m}\in \mathcal{P}(X),$ if $\boldsymbol{\nu }(%
\boldsymbol{m},t)$ is the solution of $\dot{\boldsymbol{\nu }}(t)=%
\boldsymbol{\nu }(t)\boldsymbol{\gamma }$ with initial point $\boldsymbol{%
\nu }(0)=\boldsymbol{m}$, then

\begin{equation*}
1.\,\forall x\in\mathcal{X}, \nu_{x}(\boldsymbol{m},t)\geq b_{1}t^{D} 
\hspace{32pt}\text{and}\hspace{32pt} 2.\,\Vert \nu (\boldsymbol{m},t)-%
\boldsymbol{m}\Vert \leq b_{2}t.
\end{equation*}
\end{remark}

Before proceeding with the proof of Theorem \ref{main2}, we prove that the
function $V(\boldsymbol{m})$ is continuous. We will actually prove something
stronger. Recall that $\boldsymbol{\gamma }$ denotes the original
unperturbed jump rates and the definitions of $\mathcal{P}_{\ast }(\mathcal{X%
})$ and $\mathcal{P}_{a}(\mathcal{X})$ in (\ref{eqn:pa}).

\begin{theorem}
\label{lemma:uniFUont}

There is a constant $\bar{c}\in\mathbb{R}$ that depends only the dimension $d$ and the
unperturbed rates $\boldsymbol{\gamma },$ such that for every $\boldsymbol{m}%
\in \mathcal{P}_{\ast}(\mathcal{X}),\, \boldsymbol{\tilde{m}}\in \mathcal{P}(%
\mathcal{X})$ there exists a control $\boldsymbol{q}\in \mathcal{U}_{%
\boldsymbol{m}},$ that generates a unique $\boldsymbol{\mu }$ with $(%
\boldsymbol{\mu },\boldsymbol{q})\in \mathcal{T}_{\boldsymbol{m}},$
satisfying

\begin{enumerate}
\item $\boldsymbol{\mu}$ is a constant speed parametrization of the straight
line that connects $\boldsymbol{m}$ and $\boldsymbol{\tilde{m}},$

\item the exit time $T_{\{\tilde{\boldsymbol{m}}\}}$ is equal to $\|%
\boldsymbol{m}-\tilde{\boldsymbol{m}}\|$,

\item $\gamma _{xy}\leq q_{xy}(t)$ \text{and} $\mu_{x}(t) q_{xy}(t)\leq \bar{%
c}$.
\end{enumerate}

Furthermore, if $\boldsymbol{m},\tilde{\boldsymbol{m}}\in \mathcal{P}_{a}(%
\mathcal{X})$ then 
\begin{equation*}
\gamma _{xy}\leq q_{xy}(t)\leq \frac{\bar{c}}{a},
\end{equation*}
and we can find a constant $c_{a}<\infty $ such that the total cost for
applying the control is bounded above by $c_{a}\Vert \boldsymbol{m}-\tilde{%
\boldsymbol{m}}\Vert .$ Finally, for every $\epsilon >0$ there exists $%
\delta >0,$ such that $\Vert \bar{\boldsymbol{m}}-\tilde{\boldsymbol{m}}%
\Vert \leq \delta $ implies $V_{\{\tilde{\boldsymbol{m}}\}}(\bar{\boldsymbol{%
m}}),V_{\{\bar{\boldsymbol{m}}\}}(\tilde{\boldsymbol{m}})\leq \epsilon ,$
and therefore as a function of two variables $V$ is continuous on $\mathcal{P%
}(\mathcal{X})\times \mathcal{P}(\mathcal{X})$.
\end{theorem}

\begin{proof}
Recall the definitions above (\ref{Hdef}), and let $\boldsymbol{m}\in 
\mathcal{P}_{\ast }(\mathcal{X}),\,\boldsymbol{\tilde{m}}\in \mathcal{P}(%
\mathcal{X})$. We can find a positive constant $\bar{c}$ that depend only
the dimension $d$ and on the rates $\boldsymbol{\gamma },$ and also rates $%
\boldsymbol{q}$ such that

\begin{equation*}
1.\, q_{xy}\geq \gamma _{xy},\hspace{16pt} 2.\,\sum_{(x,y)\in \mathcal{Z}%
}m_{x}q_{xy}\boldsymbol{v}_{xy}=\frac{\tilde{\boldsymbol{m}}-\boldsymbol{m}}{%
\Vert \tilde{\boldsymbol{m}}-\boldsymbol{m}\Vert },\hspace{16pt}3.\, \max
\{m_{x}q_{xy},(x,y)\in \mathcal{Z}\}\leq \bar{c}.
\end{equation*}

Indeed, since \eqref{Hdef} holds, we can find a constant $c<\infty $ such that for every
point $\boldsymbol{m}\in \mathcal{P}_{\ast }(\mathcal{X}),$ there exist
vectors $q_{xy}m_{x}\boldsymbol{v}_{xy}$ with $q_{xy}m_{x}\leq c,$ and $%
\sum_{(x,y)\in \mathcal{Z}}m_{x}q_{xy}\boldsymbol{v}_{xy}=\frac{\tilde{%
\boldsymbol{m}}-\boldsymbol{m}}{\Vert \tilde{\boldsymbol{m}}-\boldsymbol{m}%
\Vert }.$ Now if for some $(x_{1},y_{1})\in \mathcal{Z}$ we do not have $%
q_{x_{1}y_{1}}\geq \gamma _{x_{1}y_{1}},$ then by ergodicity we can pick $%
x_{1},x_{2}=y_{1},x_{3},\ldots ,x_{j},$ with $j\leq d,$ such that $%
\sum_{i=1}^{j-1}\boldsymbol{v}_{x_{i}x_{i+1}}=0.$ If we pick the new $%
q_{x_{i}x_{i+1}}$ equal to $\max_{xy}\{\gamma _{xy}\}/m_{x_{i}}$ plus the
original $q_{x_{i}x_{i+1}},$ then property 2 is still satisfied, but we now
also have $q_{x_{1}y_{1}}\geq \gamma _{x_{1}y_{1}}.$ We have to repeat the
procedure at most $|\mathcal{Z}|$ times to enforce property 1, and can then
set $\bar{c}\doteq \max \{m_{x}q_{xy},(x,y)\in \mathcal{Z}\}$.

Let 
\begin{equation}  \label{vare8hka}
\tilde{\boldsymbol{\mu}}(t)=[(\tilde{\boldsymbol{m}}-\boldsymbol{m})t/\Vert 
\tilde{\boldsymbol{m}}-\boldsymbol{m}\Vert +\boldsymbol{m}],
\end{equation}
and define $\tilde{\boldsymbol{q}}\in \mathcal{U}_{\boldsymbol{m}}$ by 
\begin{equation}
\tilde{\mu}_{x}(t)\tilde{q}_{xy}(t)=m_{x}q_{xy}\leq \bar{c}.
\label{linearcontrols}
\end{equation}%
Then automatically%
\begin{equation*}
\sum_{(x,y)\in \mathcal{Z}}\boldsymbol{v}_{xy}\int_{[0,t)}\tilde{\mu}_{x}(s)%
\tilde{q}_{xy}(s)ds=t\frac{\tilde{\boldsymbol{m}}-\boldsymbol{m}}{\Vert 
\tilde{\boldsymbol{m}}-\boldsymbol{m}\Vert }=\tilde{\boldsymbol{\mu }}(t)-%
\boldsymbol{m},
\end{equation*}%
and thus $(\tilde{\boldsymbol{\mu }},\tilde{\boldsymbol{q}})\in \mathcal{T}_{%
\boldsymbol{m}}$. This will lead to hitting $\{\tilde{\boldsymbol{m}}\}$ in
time $T_{\{\tilde{\boldsymbol{m}}\}}=\Vert \boldsymbol{m}-\tilde{\boldsymbol{%
m}}\Vert .$ Using properties stated in Lemma \ref{lemma:fnproperties2} we get%
\begin{equation*}
\begin{split}
& \inf_{(\boldsymbol{\mu },\boldsymbol{q})\in \mathcal{T}_{\boldsymbol{m}%
}}J_{\{\tilde{\boldsymbol{m}}\}}(\boldsymbol{m},\boldsymbol{\mu },%
\boldsymbol{q})\leq J_{\{\tilde{\boldsymbol{m}}\}}(\boldsymbol{m},\tilde{%
\boldsymbol{\mu }},\tilde{\boldsymbol{q}})\leq \sum_{(x,y)\in \mathcal{Z}%
}\int_{0}^{T_{\{\tilde{\boldsymbol{m}}\}}}\tilde{\mu}_{x}(t)F_{xy}(\tilde{q}%
_{xy}(t))+R_{\max }T_{\{\tilde{\boldsymbol{m}}\}} \\
& \overset{(\ref{linearcontrols})}{\leq }\!\!\!\!\!\sum_{(x,y)\in \mathcal{Z}%
}\!\int_{0}^{T_{\{\tilde{\boldsymbol{m}}\}}}\!\!\!\left(\! \tilde{\mu}_{x}(t)%
\tilde{q}_{xy}(t)\log \frac{\tilde{q}_{xy}(t)}{\min \left\{ \gamma
_{xy}\left( \gamma _{xy}/\tilde{q}_{xy}(t)\right) ^{1/p},M\right\} }%
+\max_{(x,y)\in \mathcal{Z}}\bar{M}(\gamma _{xy})\!\!\right) \!\!dt+R_{\max
}T_{\{\tilde{\boldsymbol{m}}\}} \\
& \leq \sum_{(x,y)\in \mathcal{Z}}\int_{0}^{T_{\{\tilde{\boldsymbol{m}}%
\}}}\left\vert \tilde{\mu}_{x}(t)\tilde{q}_{xy}(t)\log \tilde{q}%
_{xy}(t)\right\vert dt+\sum_{(x,y)\in \mathcal{Z}}\int_{0}^{T_{\{\tilde{%
\boldsymbol{m}}\}}}\left\vert \tilde{\mu}_{x}(t)\tilde{q}_{xy}(t)\log \left(
\gamma _{xy}/\tilde{q}_{xy}(t)\right) ^{1/p}\right\vert dt+ \\
& +\sum_{(x,y)\in \mathcal{Z}}\int_{0}^{T_{\{\tilde{\boldsymbol{m}}%
\}}}\left\vert \tilde{\mu}_{x}(t)\tilde{q}_{xy}(t)\log \gamma
_{xy}\right\vert dt+\sum_{(x,y)\in \mathcal{Z}}\int_{0}^{T_{\{\tilde{%
\boldsymbol{m}}\}}}\left\vert \tilde{\mu}_{x}(t)\tilde{q}_{xy}(t)\log
M\right\vert dt+c^{\prime }T_{\{\tilde{\boldsymbol{m}}\}} \\
& \overset{(\ref{linearcontrols})}{\leq }\bar{c}\sum_{(x,y)\in \mathcal{Z}%
}\int_{0}^{T_{\{\tilde{\boldsymbol{m}}\}}}|\log \tilde{q}_{xy}(t)|dt+\bar{c}%
\sum_{(x,y)\in \mathcal{Z}}\int_{0}^{T_{\{\tilde{\boldsymbol{m}}%
\}}}\left\vert \log \left( \gamma _{xy}/\tilde{q}_{xy}(t)\right)
^{1/p}\right\vert dt+c^{\prime \prime }T_{\{\tilde{\boldsymbol{m}}\}} \\
& \overset{(\ref{linearcontrols})}{\leq }\bar{c}\sum_{(x,y)\in \mathcal{Z}%
}\int_{0}^{T_{\{\tilde{\boldsymbol{m}}\}}}\left\vert \log \frac{m_{x}q_{xy}}{%
\tilde{\mu}_{x}(t)}\right\vert dt+\bar{c}\sum_{(x,y)\in \mathcal{Z}%
}\int_{0}^{T_{\{\tilde{\boldsymbol{m}}\}}}\left\vert \log \left( \tilde{\mu}%
_{x}(t)\gamma _{xy}/m_{x}q_{xy}\right) ^{1/p}\right\vert dt+c^{\prime \prime
}T_{\{\tilde{\boldsymbol{m}}\}} \\
& \overset{(\ref{linearcontrols})}{\leq }\bar{c}\sum_{(x,y)\in \mathcal{Z}%
}\int_{0}^{T_{\{\tilde{\boldsymbol{m}}\}}}\left\vert \log \tilde{\mu}%
_{x}(t)\right\vert dt+\bar{c}\sum_{(x,y)\in \mathcal{Z}}\int_{0}^{T_{\{%
\tilde{\boldsymbol{m}}\}}}\frac{1}{p}\left\vert \log \tilde{\mu}%
_{x}(t)\right\vert dt+c^{\prime \prime \prime }T_{\{\tilde{\boldsymbol{m}}%
\}},
\end{split}%
\end{equation*}%
where the constants $c^{\prime },c^{\prime \prime },c^{\prime \prime \prime
} $ depend only on $\boldsymbol{\gamma },\bar{c}$ and $R_{\max }.
$

Now if $\boldsymbol{m},\tilde{\boldsymbol{m}}\in \mathcal{P}_{a}(\mathcal{X}%
) $, then all elements are bounded by a constant $c_{a}$ (that depends on $%
\gamma ,\bar{c},R_{\max },$ and $a$) times $T_{\{\tilde{\boldsymbol{m}}%
\}}=\Vert \tilde{\boldsymbol{m}}-\boldsymbol{m}\Vert $, and therefore the
first part of the theorem follows.

Let $1>\delta >0,$ and $\bar{\boldsymbol{m}},\tilde{\boldsymbol{m}}\in 
\mathcal{P}(\mathcal{X}),$ with $\Vert \bar{\boldsymbol{m}}-\tilde{%
\boldsymbol{m}}\Vert <\delta .$ We take $\boldsymbol{m}=\boldsymbol{\nu} (\bar{%
\boldsymbol{m}},\delta ),$ where $\boldsymbol{\nu }(\bar{\boldsymbol{m}},t)$
is the solution of $\dot{\boldsymbol{\nu }}(t)=\boldsymbol{\nu }(t)%
\boldsymbol{\gamma },$ with initial data $\boldsymbol{\nu }(0)=\bar{%
\boldsymbol{m}}$. Now by appropriate use of the inequality $\tilde{\mu}%
_{x}(t)\geq \min \{m_{x},m_{x}(T_{\{\tilde{\boldsymbol{m}}\}}-t)\},$ that we
get from (\ref{vare8hka}), and using the last display, we get%
\begin{equation*}
V_{\{\tilde{\boldsymbol{m}}\}}(\boldsymbol{m})\leq c^{\prime \prime \prime
\prime }\left( \sum_{(x,y)\in \mathcal{Z}}\int_{0}^{T_{\{\tilde{\boldsymbol{m%
}}\}}}\left( |\log m_{x}|+|\log (T_{\{\tilde{\boldsymbol{m}}\}}-t)|\right)
dt+T_{\{\tilde{\boldsymbol{m}}\}}\right) .
\end{equation*}%
By a simple change of variable and Remark \ref{away}, we have%
\begin{equation*}
V_{\{\tilde{\boldsymbol{m}}\}}(\boldsymbol{m})\leq c^{\prime \prime \prime
\prime }\left( \sum_{(x,y)\in \mathcal{Z}}\int_{0}^{b_{2}\delta }\left(
|\log b_{1}\delta ^{D}|+|\log t|\right) dt+b_{2}\delta \right) .
\end{equation*}
Therefore 
\begin{equation*}
V_{\{\tilde{\boldsymbol{m}}\}}(\bar{\boldsymbol{m}})\leq V_{\{\boldsymbol{m}%
\}}(\bar{\boldsymbol{m}})+V_{\{\tilde{\boldsymbol{m}}\}}(\boldsymbol{m})\leq
\delta R_{\max }+c^{\prime \prime \prime \prime }\left( \sum_{(x,y)\in 
\mathcal{Z}}\int_{0}^{b_{2}\delta }\left( |\log b_{1}\delta ^{D}|+|\log
t|\right) dt+b_{2}\delta \right) ,
\end{equation*}%
and the right hand side can be made as small as desired by making $\delta $
small enough. {The estimate for $V_{\{\bar{\boldsymbol{m}}\}}(\tilde{%
\boldsymbol{m}})$ is proved in a symmetric way.} This proves the last
statement of the theorem.
\end{proof}

\section{Lower bound}

For the proof of Theorem \ref{main2}, we first prove the lower bound: for
every sequence $\boldsymbol{m}^{n}\in \mathcal{P}^{n}(\mathcal{X})$ and $%
\boldsymbol{m}\in \mathcal{P}(\mathcal{X}),$ with $\boldsymbol{m}%
^{n}\rightarrow \boldsymbol{m},$ we have

\begin{equation*}
\liminf_{n\rightarrow \infty }V_{K}^{n}(\boldsymbol{m}^{n})\geq V_{K}(%
\boldsymbol{m}).
\end{equation*}%
Without loss of generality we can assume that the liminf is actually a
limit, otherwise we can just work with a subsequence. If the limit is $%
\infty $ then the conclusion is trivial, therefore we can assume that there
is $c\in \mathbb{R}$ such that 
\begin{equation}
\sup_{n\in \mathbb{N}}V_{K}^{n}(\boldsymbol{m}^{n})\leq c.
\label{eq:valuefuntionbound}
\end{equation}%
Let $\epsilon \in (0,1)$. Recalling (\ref{eqdef:Vn}), let $\boldsymbol{q}%
^{n}\in \mathcal{A}_{b}^{n,|\mathcal{Z}|}$ be such that

\begin{equation}
\mathbb{E}_{\boldsymbol{m}^{n}}\left[ \int_{0}^{T^{n}}\left( \sum_{(x,y)\in 
\mathcal{Z}}\mu _{x}^{n}(t)F_{xy}(q_{xy}^{n}(t))+R(\boldsymbol{\mu }%
^{n}(t))\right) dt\right] <V_{K}^{n}(\boldsymbol{m}^{n})+\epsilon ,
\label{eq:epsilonbound}
\end{equation}%
where $\boldsymbol{\mu }^{n}=h^{n}\left( \boldsymbol{\mu }^{n},\boldsymbol{q}%
^{n},\boldsymbol{m}^{n},\boldsymbol{N}^{n}/n\right) $ and $T^{n}\doteq \inf
\left\{ t\in \lbrack 0,\infty ]:\boldsymbol{\mu }^{n}(t)\in K\right\} .$ For 
$\delta >0$ such that 
\begin{equation*}
\Vert \bar{\boldsymbol{m}}-\tilde{\boldsymbol{m}}\Vert \leq \delta
\Rightarrow V_{\bar{\boldsymbol{m}}}(\tilde{\boldsymbol{m}})\leq \epsilon ,
\end{equation*}%
\begin{equation*}
\hspace{-16pt}\text{we define }\hspace{38pt} K_{\delta }\doteq \{m:d(\boldsymbol{m},K)\leq \delta \}\hspace{16pt}\text{and%
}\hspace{16pt}T^{n,\delta }\doteq \inf \{t\in \lbrack 0,\infty ]:\boldsymbol{%
\mu }_{n}(t)\in K_{\delta }\}.
\end{equation*}%
The existence of such a $\delta $ is given by Theorem \ref{lemma:uniFUont}.
Now for $\boldsymbol{\mu }^{n},\boldsymbol{q}^{n}$ as in %
\eqref{eq:epsilonbound} and $T^{n,\delta }$ as above, we define the
sequences $\boldsymbol{\mu }^{n,\delta }(t)=\boldsymbol{\mu }^{n}(t\wedge
T^{n,\delta }),$ 
\begin{equation*}
\boldsymbol{q}^{n,\delta }(t)=%
\begin{cases}
\boldsymbol{q}^{n}(t) & t\leq T^{n,\delta } \\ 
\boldsymbol{\gamma } & T>T^{n,\delta }%
\end{cases}%
.
\end{equation*}%
We note that for $t>T^{\delta },$ $\boldsymbol{q}^{n,\delta }(t)$ does not
actually generate $\boldsymbol{\mu }^{n},$ but we define it this way to
simplify some arguments later on. We will show that 
\begin{equation}
\begin{split}
& \liminf_{n\rightarrow \infty }\mathbb{E}_{\boldsymbol{m}^{n}}\left[
\int_{0}^{T^{n}}\left( \sum_{(x,y)\in \mathcal{Z}}\mu
_{x}^{n}(t)F_{xy}(q_{xy}^{n}(t))+R(\boldsymbol{\mu }^{n}(t))\right) dt\right]
\geq \\
& \liminf_{n\rightarrow \infty }\mathbb{E}_{\boldsymbol{m}^{n}}\left[
\int_{0}^{T^{n,\delta }}\left( \sum_{(x,y)\in \mathcal{Z}}\mu _{x}^{n,\delta
}(t)F_{xy}(q_{xy}^{n,\delta }(t))+R(\boldsymbol{\mu }^{n,\delta }(t))\right)
dt\right] \geq V_{K_{\delta }}(\boldsymbol{m}),
\end{split}
\label{twopartsproof}
\end{equation}%
and then by an application of Theorem \ref{lemma:uniFUont} and %
\eqref{eq:epsilonbound} deduce $\lim_{n\rightarrow \infty }V_{K}^{n}(%
\boldsymbol{m}^{n})+2\epsilon \geq V_{K}(\boldsymbol{m})$. Since $\epsilon $
is arbitrary the lower bound will follow. The first inequality in (\ref%
{twopartsproof}) is true since $F_{xy}\geq 0,R\geq 0$ and $T^{n,\delta }\leq
T^{n}$. Therefore only the second inequality needs to be proved.

Before proceeding we introduce some auxiliary random measures. For $(x,y)\in 
\mathcal{Z},\,q_{xy}\in \mathcal{F}([0,\infty );[0,\infty )),$ and $t\in
\lbrack 0,\infty ),$ define 
\begin{equation*}
\eta _{xy}(dr;t)\doteq \delta _{q_{xy}(t)}(dr)\mu _{x}(t).
\end{equation*}%
For each $t\in \lbrack 0,\infty ),$ $(x,y)\in \mathcal{Z}$ we have that $%
\eta _{xy}(\cdot ;t)$ is a subprobability measure on $[0,\infty )$. Also we
consider the measures $\theta _{xy}(drdt)=\eta _{xy}(dr;t)dt$ on $[0,\infty
)\times \lbrack 0,\infty )$ as equipped with the topology that generalizes
the weak convergence of probability measures to general measures that have
at most mass $T$ on $[0,\infty )\times \lbrack 0,T]$. This can be defined in
terms of a distance (a generalization of the Prohorov metric) $d_{T}$,%
and the metric on measures on $[0,\infty
)\times \lbrack 0,\infty )$ is 
\begin{equation}
\sum_{T\in \mathbb{N}}2^{-T}\left[ d_{T}(\boldsymbol{\mu }|_{T},\boldsymbol{%
\nu }|_{T})\vee 1\right] ,  \label{eqdef:topology}
\end{equation}%
where $\boldsymbol{\mu }|_{T}$ denotes the restriction to $[0,T]$ in the
last variable.

Let $\boldsymbol{\theta }^{n,\delta }=\{\theta ^{n,\delta }\}_{(x,y)\in 
\mathcal{Z}}$ be the random measures that correspond to $\boldsymbol{\mu }%
^{n,\delta },\boldsymbol{q}^{n,\delta },$ according to the construction
above. We observe that

\begin{equation*}
\boldsymbol{\mu }^{n,\delta }(t)=\boldsymbol{m}^{n}+\sum_{(x,y)\in \mathcal{Z%
}}\boldsymbol{v}_{xy}\int_{0}^{t\wedge T^{n,\delta }}\int_{0}^{\infty
}r\theta _{xy}^{n,\delta }(drds)+\text{ a martingale},
\end{equation*}%
where the martingale will converge to zero as $n\rightarrow \infty $, and
that for every $(x,y)\in \mathcal{Z},$

\begin{equation}  \label{eqdef:nucost}
\mathbb{E}_{\boldsymbol{m}^{n}}\left[\int_{0}^{T^{n,\delta}}F_{xy}(q^{n,%
\delta}_{xy}(t))\mu_{x}^{n,\delta}(t)dt\right]=\mathbb{E}_{\boldsymbol{m}%
^{n}}\left[\int_{0}^{T^{n,\delta}}\int_{0}F_{xy}(r)\theta_{xy}^{n,%
\delta}(drdt)\right].
\end{equation}

We will split the proof of \eqref{twopartsproof} in three parts. First we
prove that $(\boldsymbol{\mu }^{n,\delta },\boldsymbol{\theta }^{n,\delta
},T^{n,\delta })$ is tight. Then we show that for every limit point $(%
\boldsymbol{\mu }^{\delta },\boldsymbol{\theta }^{\delta },T^{\delta }),$ $%
\theta _{xy}^{\delta }$ has the decomposition $\theta _{xy}^{\delta
}(drdt)=\eta _{xy}^{\delta }(dr;t)dt,$ with $\sum_{y\in \mathcal{X}}\eta
_{xy}^{\delta }([0,\infty );t)=\mu _{x}^{\delta }(t),$ and for $\boldsymbol{q%
}^{\delta }$ defined by $\mu _{x}^{\delta }(t)q_{xy}^{\delta
}(t)=\int_{0}^{\infty }r\eta _{xy}^{\delta }(dr;t),$ that

\begin{equation*}
\boldsymbol{\mu }^{\delta }(t)=\boldsymbol{m}+\sum_{(x,y)\in \mathcal{Z}}%
\boldsymbol{v}_{xy}\int_{0}^{t\wedge T^{\delta }}\int_{0}^{\infty }r\theta
_{xy}^{\delta }(drds)=\boldsymbol{m}+\sum_{(x,y)\in \mathcal{Z}}\boldsymbol{v%
}_{xy}\int_{0}^{t\wedge T^{\delta }}\mu _{x}^{\delta }(s)q_{xy}^{\delta
}(s)ds.
\end{equation*}%
Finally, by an application of Fatou's Lemma, for such a $\boldsymbol{q}%
^{\delta },$ we get 
\begin{equation*}
\begin{split}
& \liminf_{n\rightarrow \infty }\mathbb{E}_{\boldsymbol{m}^{n}}\!\!\left[\!
\int_{0}^{T^{n,\delta }}\!\!\!\!\!\int_{0}^{\infty }\!\!\!F_{xy}(r)\theta _{xy}^{n,\delta
}(drdt)\right] \!\!\geq\! \mathbb{E}_{\boldsymbol{m}}\!\!\left[ \int_{0}^{T^{\delta
}}\!\!\!\int_{0}^{\infty }\!\!\!\!F_{xy}(r)\theta _{xy}^{\delta }(drdt)\right]\!\!  \geq\!\! \mathbb{E}_{\boldsymbol{m}}\left[ \int_{0}^{T^{\delta
}}\!\!\!\!\int_{0}^{\infty }F_{xy}(r)\eta _{xy}^{\delta }(dr;t)dt\right] \\&\geq 
\mathbb{E}_{\boldsymbol{m}}\left[ \int_{0}^{T^{\delta }}F_{xy}\left(
\int_{0}^{\infty }r\frac{\eta _{xy}^{\delta }(dr;t)}{\eta _{xy}^{\delta
}([0,\infty );t)}\right) \eta _{xy}^{\delta }([0,\infty );t)dt\right]=\mathbb{E}_{\boldsymbol{m}}\left[ \int_{0}^{T^{\delta }}F
_{xy}(q_{xy}^{\delta }(t))\mu _{x}^{\delta }(t)dt\right] ,
\end{split}%
\end{equation*}%
where for the third estimate, we applied Jensen's inequality. Together with $%
\boldsymbol{\mu }^{n,\delta }\rightarrow \boldsymbol{\mu }^{\delta },F_{xy},
R\geq 0$ and another application of Fatou's Lemma, this gives %
\eqref{twopartsproof}.

\subsection{Tightness of $(\boldsymbol{\protect\mu}^{n,\protect\delta},%
\boldsymbol{\protect\theta}^{n,\protect\delta},T^{n,\protect\delta})$}

\label{section:nu2}

First, we prove that $(\boldsymbol{\mu }^{n,\delta }(\cdot ),T^{n,\delta })$%
, which takes values in $D([0,\infty );\mathcal{P}(\mathcal{X}))\times
\lbrack 0,\infty )\subset D([0,\infty );\mathbb{R}^{d})\times \lbrack
0,\infty ),$ is tight. For that, we introduce some auxiliary random
variables $\tilde{\boldsymbol{\mu }}^{n,\delta }$ in $D([0,\infty );\mathbb{R%
}^{d}),$ to compare with $\boldsymbol{\mu }^{n,\delta }$, given by 
\begin{equation}  \label{eqdef:upmu2}
\tilde{\boldsymbol{\mu}}^{n,\delta}(t)=\boldsymbol{m}^{n}+\sum_{(x,y)\in%
\mathcal{Z}}\boldsymbol{v}_{xy}\int_{0}^{t\wedge
T^{n,\delta}}\mu^{n}_{x}(s)q^{n}_{xy}(s)ds.
\end{equation}

Since $\gamma _{xy}\ell \left( \cdot /\gamma _{xy}\right) \leq F_{xy}(\cdot
),$ recalling (\ref{eq:valuefuntionbound}), (\ref{eq:epsilonbound}) and that
$R$ is bounded away from zero in $K_{\delta }=\{%
\boldsymbol{m}:d(\boldsymbol{m},K)\geq \delta \}$ by a constant $R_{\min
}^{\delta },$ we get

\begin{equation}
\mathbb{E}_{\boldsymbol{m}^{n}}\left[ \int_{0}^{T^{n,\delta }}\left(
\sum_{(x,y)\in \mathcal{Z}}\mu _{x}^{n}(t)\gamma _{xy}\ell \left( \frac{%
q_{xy}^{n}(t)}{\gamma _{xy}}\right) \right) dt+R_{\min }^{\delta
}T^{n,\delta }\right] \leq c+1,  \label{eq:npathellbound}
\end{equation}%
which shows tightness of $\{T^{n,\delta }\}$. By setting $\gamma _{\max
}=\max \{\gamma _{xy}:(x,y)\in \mathcal{Z}\},$ we get

\begin{equation*}
\mathbb{E}_{\boldsymbol{m}^{n}}\left[ \int_{0}^{T^{n,\delta }}\left(
\sum_{(x,y)\in \mathcal{Z}}\gamma _{\max }\frac{\mu _{x}^{n}(t)\gamma _{xy}}{%
\gamma _{\max }}\ell \left( \frac{q_{xy}^{n}(t)}{\gamma _{xy}}\right)
\right) dt+R_{\min }^{\delta }T^{n,\delta }\right] \leq c+1.
\end{equation*}%
Using the fact that $\ell $ is convex and $\ell (1)=0,$ by Jensen's
inequality $a\ell (b)\geq \ell (ab+1-a)$ for $a\in \lbrack 0,1]$ and $b\geq
0.$ By setting $a=\frac{\mu _{x}^{n}(t)\gamma _{xy}}{\gamma _{\max }},$ the
inequality above gives

\begin{equation*}
\mathbb{E}_{\boldsymbol{m}^{n}}\left[ \int_{0}^{T^{n,\delta }}\left(
\sum_{(x,y)\in \mathcal{Z}}\gamma _{\max }\ell \left( \frac{\mu _{x}^{n}(t)}{%
\gamma _{\max }}q_{xy}^{n}(t)+1-\frac{(\mu _{x}^{n}(t)\gamma _{xy})}{\gamma
_{\max }}\right) \right) dt+R_{\min }^{\delta }T^{n,\delta }\right] \leq c+1.
\end{equation*}%
By applying Jensen's inequality once more 
\begin{equation*}
\mathbb{E}_{\boldsymbol{m}^{n}}\left[ \int_{0}^{T^{n,\delta }}|\mathcal{Z}%
|\gamma _{\max }\ell \left( \frac{1}{|\mathcal{Z}|\gamma _{\max }}%
\sum_{(x,y)\in \mathcal{Z}}\mu _{x}^{n}(t)q_{xy}^{n}(t)+\sum_{(x,y)\in 
\mathcal{Z}}\left[ 1-\frac{(\mu _{x}^{n}(t)\gamma _{xy})}{|\mathcal{Z}%
|\gamma _{\max }}\right] \right) dt+R_{\min }^{\delta }T^{n,\delta }\right]
\leq c+1.
\end{equation*}

Now by multiplying with $\frac{1}{\mathcal{|\mathcal{Z}|}\gamma _{\max }},$
using (\ref{eqdef:upmu2}) and the fact that $q\leq q^{\prime }$ implies $%
\ell (q)\leq \ell (q^{\prime })+1$, we get

\begin{equation*}
\mathbb{E}_{\boldsymbol{m}^{n}}\left[ \int_{0}^{T^{n,\delta }}\ell \left( 
\frac{|\dot{\tilde{\boldsymbol{\mu }}}^{n,\delta }(t)|}{|\mathcal{Z}|\gamma
_{\max }}\right) dt+\left( \frac{1}{|\mathcal{Z}|\gamma _{\max }}R_{\min
}^{\delta }-1\right) T^{n,\delta }\right] \leq \frac{c+1}{|\mathcal{Z}%
|\gamma _{\max }}.
\end{equation*}%
Finally, by using that for every $\bar{c}>0$ there exists $%
c_{1}>0,c_{2}<\infty $ such that $\ell (\bar{c}q)\geq c_{1}\ell (q)-c_{2},$
we get%
\begin{equation*}
\mathbb{E}_{\boldsymbol{m}^{n}}\left[ \int_{0}^{T^{n,\delta }}c_{1}\ell
\left( |\dot{\tilde{\boldsymbol{\mu }}}^{n,\delta }(t)|\right) dt+\left( 
\frac{1}{|\mathcal{Z}|\gamma _{\max }}R_{\min }^{\delta }-1-c_{2}\right)
T^{n,\delta }\right] \leq \frac{c+1}{|\mathcal{Z}|\gamma _{\max }},
\end{equation*}%
which implies

\begin{equation*}
\mathbb{E}_{\boldsymbol{m}^{n}}\left[ \int_{0}^{T^{n,\delta }}\ell \left( |%
\dot{\tilde{\boldsymbol{\mu }}}^{n,\delta }(t)|\right) dt+\frac{1}{|\mathcal{%
Z}|\gamma _{\max }c_{1}}R_{\min }^{\delta }T^{n,\delta }\right] \leq \frac{%
c+1}{|\mathcal{Z}|\gamma _{\max }c_{1}}+\frac{(c_{2}+1)}{c_{1}}\mathbb{E}_{%
\boldsymbol{m}^{n}}[T^{n,\delta }]\leq c^{\prime },  \label{eqn:tildebound}
\end{equation*}%
where 
\begin{equation}\label{eqdef:c'}c^{\prime }=\frac{c+1}{|\mathcal{Z}|\gamma _{\max }c_{1}}+\frac{%
(c+1)(c_{2}+1)}{c_{1}}. \end{equation} It will follow from the following lemma that $\tilde{\boldsymbol{\mu }}%
^{n,\delta }$ is a tight sequence in $D([0,\infty );\mathbb{R}^{d})$. Let $%
\mathcal{S}$ be the elements $(\boldsymbol{\mu },T)$ of $C([0,\infty );%
\mathcal{P}(\mathcal{X}))\times \lbrack 0,\infty )$ that satisfy $%
\boldsymbol{\mu }(t)=\boldsymbol{\mu }(T)$ for $t\geq T$.

\begin{lemma}
\label{lemma:tightnessupmu} For every positive number $a,$ the function 
\begin{equation*}
H(\boldsymbol{\mu },T)=%
\begin{cases}
\int_{0}^{T}\ell \left( |\dot{\boldsymbol{\mu }}(t)|\right) dt+aT, & 
\boldsymbol{\mu }\in AC([0,\infty );\mathbb{R}^{d}),T\in \lbrack 0,\infty )
\\ 
\infty , & otherwise,%
\end{cases}%
\end{equation*}%
is a tightness function on $\mathcal{S},$ where $AC([0,\infty );\mathbb{R}%
^{d})$ is the set of all absolutely continuous functions from $[0,\infty )$
to $\mathbb{R}^{d}$.
\end{lemma}

The proof of this lemma is in Appendix \ref{appendix:tightness}. {Using the bound (\ref{eqn:tildebound}%
), it follows from Lemma \ref{lemma:tightnessupmu}} that $\{\tilde{%
\boldsymbol{\mu }}^{n,\delta }\}$ is tight in $D([0,\infty );\mathbb{R}^{d})$%
. Now we have that 
\begin{equation*}
|\boldsymbol{\mu }^{n,\delta }(t)-\tilde{\boldsymbol{\mu }}^{n,\delta
}(t)|\leq \sum_{(x,y)\in \mathcal{Z}}\left\vert \int_{0}^{t\wedge
T^{n,\delta }}\mu _{x}^{n}(s)q_{xy}^{n}(s)ds-\int_{0}^{t\wedge T^{n,\delta
}}\int_{0}^{\infty }1_{[0,\mu _{x}^{n}(s)q_{xy}^{n}(s)]}(r)\frac{1}{n}%
N_{xy}^{n}(dsdr)\right\vert ,
\end{equation*}%
where the summands on the right side, denoted from now on by $%
Q_{xy,t}^{n,\delta },$ are all martingales with quadratic variation $\mathbb{%
Q}_{xy,t}^{n,\delta }$ that is bounded above by

\begin{equation*}
\begin{split}
& \frac{1}{n^{2}}\mathbb{E}_{\boldsymbol{m}^{n}}\left[ \int_{0}^{t\wedge
T^{n,\delta }}\int_{0}^{\infty }1_{[0,\mu
_{x}^{n}(s)q_{xy}^{n}(s)]}(r)N_{xy}^{n}(dsdr)\right] =\frac{1}{n}\mathbb{E}_{%
\boldsymbol{m}^{n}}\left[ \int_{0}^{t\wedge T^{n,\delta }}\mu
_{x}^{n}(s)q_{xy}^{n}(s)ds\right] \\
& \leq \frac{1}{n}\mathbb{E}_{\boldsymbol{m}^{n}}\left[ \int_{0}^{t\wedge
T^{n,\delta }}(\ell (\mu _{x}^{n}(s)q_{xy}^{n}(s))+e)ds\right] \overset{(\ref%
{eqdef:c'})}{\leq }\frac{c^{\prime }+e\mathbb{E}_{\boldsymbol{m}%
^{n}}[T^{n,\delta }\wedge t]}{n}\leq \frac{c^{\prime }+e\mathbb{E}_{%
\boldsymbol{m}^{n}}[T^{n,\delta }]}{n}\overset{(\ref{eq:npathellbound})}{%
\leq }\frac{\left( \frac{(c+1)e}{R_{\min }^{\delta }}+c^{\prime }\right) }{n}%
,
\end{split}%
\end{equation*}%
where in the first inequality of the last line, the estimate $ab\leq
e^{a}+\ell (b),$ with $a=1,b=\mu _{x}^{n}(s)q_{xy}^{n}(s)$ was used. By
using the Burkholder-Gundy-Davis inequality, for every $T\in (0,\infty )$

\begin{equation}
\mathbb{E}_{\boldsymbol{m}^{n}}\left[ \sup_{t\in \lbrack
0,T]}|Q_{xy,t}^{n,\delta }|\right] \leq c_{BGD}\mathbb{E}_{\boldsymbol{m}%
^{n}}[\mathbb{Q}_{xy}^{n,\delta }]_{T}^{1/2}\leq c_{BGD}\sqrt{\frac{\left( 
\frac{(c+1)e}{R_{\min }^{\delta }}+c^{\prime }\right) }{n}},  \label{BGD}
\end{equation}%
from which we get that $\mathbb{E}_{\boldsymbol{m}^{n}}[\sup_{t\in \lbrack
0,T]}|Q_{xy,t}^{n,\delta }|]$ converges to zero as $n\rightarrow \infty $.
Recalling that {we} already proved $\{\tilde{\boldsymbol{\mu }}^{n,\delta
}\} $ is tight in $D([0,\infty );\mathbb{R}^{d})$, it follows from $\mathbb{E%
}_{\boldsymbol{m}^{n}}\left[ d(\boldsymbol{\mu }^{n,\delta },\tilde{%
\boldsymbol{\mu }}^{n,\delta })\right] \rightarrow 0$ that $\{(\boldsymbol{%
\mu }^{n,\delta },T^{n,\delta })\}$ is tight as well.

To show that the variable $\boldsymbol{\theta }^{n,\delta }$ is tight, we
combine (\ref{eqdef:nucost}) and (\ref{eq:valuefuntionbound}), (\ref%
{eq:epsilonbound}) and use the monotonicity with respect to $\delta$ to get
\begin{equation*}
\mathbb{E}_{\boldsymbol{m}^{n}}\left[ \sum_{(x,y)\in \mathcal{Z}%
}\int_{0}^{T^{n,\delta }}\int_{0}^{\infty }F_{xy}(r)\theta _{xy}^{n,\delta
}(drdt)+\int_{0}^{T^{n,\delta }}R(\boldsymbol{\mu }^{n,\delta }(t))\right]
<c+1.
\end{equation*}%
Since, by part 1 of Lemma \ref{lemma:fnproperties}, we have $\gamma
_{xy}\ell \left( \cdot /\gamma _{xy}\right) \leq F_{xy}(\cdot )$, and $%
q^{n,\delta }=\boldsymbol{\gamma }$ for $t>T^{n,\delta }$, we get

\begin{equation*}
\mathbb{E}_{\boldsymbol{m}^{n}}\left[ \sum_{(x,y)\in \mathcal{Z}%
}\int_{0}^{\infty }\int_{0}^{\infty }\gamma _{xy}\ell \left( \frac{r}{\gamma
_{xy}}\right) \theta _{xy}^{n,\delta }(drdt)\right] =\mathbb{E}_{\boldsymbol{%
m}^{n}}\left[ \sum_{(x,y)\in \mathcal{Z}}\int_{0}^{T^{n,\delta
}}\int_{0}^{\infty }\gamma _{xy}\ell \left( \frac{r}{\gamma _{xy}}\right)
\theta _{xy}^{n,\delta }(drdt)\right] <c+1.
\end{equation*}%
Now by using the fact that 
\begin{equation*}
\tilde{H}(\theta )=\int_{0}^{\infty }\int_{0}^{T}\gamma _{xy}\ell \left( 
\frac{r}{\gamma _{xy}}\right) \theta (drdt),
\end{equation*}%
is a tightness function on the space of measures on $[0,\infty )\times
\lbrack 0,T]$ with mass no greater than $T$, we conclude that for every $%
(x,y)\in \mathcal{Z},$ $\theta _{xy}^{n,\delta }$ is tight with the topology
introduced in \eqref{eqdef:topology}.

\subsection{Distributional limits and the lower bound}

From the previous two subsections we have that $(\boldsymbol{\mu }^{n,\delta
},\tilde{\boldsymbol{\mu }}^{n,\delta },\boldsymbol{\theta }^{n,\delta
},T^{n,\delta }),$ is tight. For proving the lower bound, we can assume
without loss that the sequence has a distributional limit $(\boldsymbol{\mu }%
^{\delta },\tilde{\boldsymbol{\mu }}^{\delta },\boldsymbol{\theta }^{\delta
},T^{\delta }).$ By using the Skorohod representation theorem we can also
assume the sequence of variables is on the same probability space $(\Omega ,%
\mathcal{F},\mathbb{P})$, and that $(\boldsymbol{\mu }^{\delta },\tilde{%
\boldsymbol{\mu }}^{\delta },\boldsymbol{\theta }^{\delta },T^{\delta })$ is
an a.s. pointwise limit.

Consider any $\omega \in \Omega $ for which there is convergence. Since by
the definition of $\theta ^{n,\delta }$

\begin{equation*}
\theta _{xy}^{n,\delta }([0,\infty )\times A)=\int_{A\cap \lbrack
0,T^{n,\delta }]}\!\!\!\!\!\!\!\!\!\!\!\!\!\!\!\!\!\!\!\mu _{x}^{n,\delta
}(t)dt,\,\forall A\in \mathcal{B}(\mathbb{R}),
\end{equation*}%
for every continuity set $A$ of $\theta _{xy}^{\delta }([0,\infty )\times
\cdot )$ we have 
\begin{align*}
& \left\vert \theta _{xy}^{\delta }([0,\infty )\times A)-\int_{A\cap \lbrack
0,T^{\delta }\rbrack}\!\!\!\!\!\!\!\!\!\!\!\!\!\!\!\!\!\mu _{x}^{\delta
}(t)dt\right\vert \leq \left\vert \theta _{xy}^{\delta }([0,\infty )\times
A)-\theta _{xy}^{n,\delta }([0,\infty )\times A)\right\vert +\left\vert
\int_{A\cap \lbrack 0,T^{n,\delta }]}\!\!\!\!\!\!\!\!\!\!\!\!\!\!\!\!\!\mu
_{x}^{n,\delta }(t)dt-\int_{A\cap \lbrack 0,T^{\delta
}]}\!\!\!\!\!\!\!\!\!\!\!\!\!\!\!\!\!\mu _{x}^{\delta }(t)dt\right\vert \\
& \leq \left\vert \theta _{xy}^{\delta }([0,\infty )\times A)-\theta
_{xy}^{n,\delta }([0,\infty )\times A)\right\vert +\left\vert \int_{A\cap
\lbrack 0,T^{\delta }]}\!\!\!\!\!\!\!\!\!\!\!\!\!\!\!\!\!\mu _{x}^{n,\delta
}(t)dt-\int_{A\cap \lbrack 0,T^{\delta
}]}\!\!\!\!\!\!\!\!\!\!\!\!\!\!\!\!\!\mu _{x}^{\delta }(t)dt\right\vert
+\int_{A\cap \lbrack \min \{T^{n,\delta },T^{\delta }\},\max \{T^{\delta
},T^{n,\delta
}\}]}\!\!\!\!\!\!\!\!\!\!\!\!\!\!\!\!\!\!\!\!\!\!\!\!\!\!\!\!\!\!\!\!\!\!\!%
\!\!\!\!\!\!\!\!\!\!\!\!\!\!\!\!\!\!\!\!\!\!\!\!\!\!\!\!\!\!\!\!\!\mu
_{x}^{n,\delta }(t)dt \, \\
& \leq \left\vert \theta _{xy}^{\delta }([0,\infty )\times A)-\theta
_{xy}^{n,\delta }([0,\infty )\times A)\right\vert +d(\mu _{x}^{n,\delta
},\mu _{x}^{\delta })+|T^{\delta }-T^{n,\delta }|\rightarrow 0.
\end{align*}%
Therefore for every continuity set $A$ of $\theta _{xy}^{\delta }([0,\infty
)\times \cdot )$ 
\begin{equation*}
\theta _{xy}^{\delta }([0,\infty )\times A)=\int_{A\cap \lbrack 0,T^{\delta
}]}\mu _{x}^{\delta }(t)dt,
\end{equation*}%
from which we conclude that for all $(x,y)\in \mathcal{Z},$ $\theta
_{xy}^{\delta }$ has the decomposition $\theta _{xy}^{\delta }(drdt)=\eta
_{xy}^{\delta }(dr;t)dt,$ with $\eta _{xy}^{\delta }([0,\infty );t)=\mu
_{x}^{\delta }(t).$ Also, since $\int_{0}^{\infty }\int_{0}^{\infty }\ell
(r)\theta _{xy}^{n,\delta }(drdt)$ is uniformly bounded and $\ell $ is
superlinear, we have convergence of the first moments of the first marginal,
i.e., 
\begin{equation*}
\int_{\mathbb{R}}f(t)r\theta _{xy}^{n,\delta }(drdt)\rightarrow \int_{%
\mathbb{R}}f(t)r\theta _{xy}^{\delta }(dt),\hspace{10pt}\forall f\in C_{b}(%
\mathbb{R}).
\end{equation*}%
Hence for $\boldsymbol{q}^{\delta }$ defined by $\mu _{x}^{\delta
}(t)q_{xy}^{\delta }(t)=\int_{0}^{\infty }r\eta _{xy}^{\delta }(dr;t),$ we
get that for all $(x,y)\in \mathcal{Z}$ 
\begin{equation}
\int_{0}^{\infty }f(t)\mu _{x}^{n,\delta }(t)q_{xy}^{n,\delta
}(t)dt\rightarrow \int_{0}^{\infty }f(t)\mu _{x}^{\delta }(t)q_{xy}^{\delta
}(t)dt,\hspace{10pt}\forall f\in C_{b}(\mathbb{R}).  \label{12}
\end{equation}

Using the fact that $d(\boldsymbol{\mu }^{n,\delta },\tilde{\boldsymbol{\mu }%
}^{n,\delta })\rightarrow 0$ and (\ref{eqdef:upmu2}), we get 
\begin{equation}
\left\vert \boldsymbol{\mu }^{n,\delta }(t)-\boldsymbol{m}%
^{n}-\sum_{(x,y)\in \mathcal{Z}}\boldsymbol{v}_{xy}\int_{0}^{T^{n,\delta
}\wedge t}\mu _{x}^{n,\delta }(s)q_{xy}^{n,\delta }(s)ds\right\vert
=\left\vert \boldsymbol{\mu }^{n,\delta }(t)-\tilde{\boldsymbol{\mu }}%
^{n,\delta }(t)\right\vert \rightarrow 0,  \label{123}
\end{equation}%
for a.e. $t.$ Applying \eqref{12} for suitable choices of $f$ and using \eqref{123}, 
\begin{equation*}
\boldsymbol{\mu }^{\delta }(t)=\boldsymbol{m}+\sum_{(x,y)\in \mathcal{Z}}%
\boldsymbol{v}_{xy}\int_{0}^{T^{\delta }\wedge t}\mu _{x}^{\delta
}(s)q_{xy}^{\delta }(s)ds
\end{equation*}%
for a.e. $t$, and since the left side is cadlag and the right side is
continuous in the last display, equality holds for $t\geq 0$. We conclude
that $\boldsymbol{q}^{\delta }$ is the control that generates $\boldsymbol{%
\mu }^{\delta },$ and we also already noticed that $\mu _{x}^{\delta
}(t)q_{xy}^{\delta }(t)=\int_{0}^{\infty }r\eta _{xy}^{\delta }(dr;t).$
Finally, since $\boldsymbol{\mu }^{n,\delta }(T^{n,\delta })\in K_{\delta }$
and $d(\boldsymbol{\mu }^{n,\delta },\boldsymbol{\mu }^{\delta })\rightarrow
0,$ by continuity of $\boldsymbol{\mu }^{\delta }$ we get $\boldsymbol{\mu }%
^{\delta }(T^{\delta })\in K_{\delta }$. As discussed below (\ref%
{eqdef:nucost}), this concludes the lower bound proof.

\section{Upper bound}

Before we proceed with the proof of the upper bound%
\begin{equation*}
\limsup_{n\rightarrow \infty }V_{K}^{n}(\boldsymbol{m}^{n})\leq V_{K}(%
\boldsymbol{m}),
\end{equation*}%
we establish some preliminary lemmas. In the following lemmas, we make use
of $\mathcal{T}_{\boldsymbol{m}},$ $\mathcal{U}_{\boldsymbol{m}}$ and $F
_{xy},$ defined in (\ref{eqdef:calT}), (\ref{eqdef:calUm}), and (\ref{def:F}%
) respectively. For the properties of $F_{xy},$ see Lemma \ref%
{lemma:fnproperties}.

\begin{lemma}
\label{lemma:awayfromboundary} Let $\boldsymbol{m}\in \mathcal{P}_{\ast }(%
\mathcal{X}),$ and $\boldsymbol{q}\in \mathcal{U}_{\boldsymbol{m}}$ be such
that $(\boldsymbol{\mu },\boldsymbol{q})\in \mathcal{T}_{\boldsymbol{m}}.$
Given $T<\infty $ and $\epsilon >0,$ we can find $a_{1},a_{2},a_{3}\in
(0,\infty )$ and $\tilde{\boldsymbol{q}}\in \mathcal{U}_{m},$ with $(\tilde{%
\boldsymbol{\mu }},\tilde{\boldsymbol{q}})\in \mathcal{T}_{\boldsymbol{m}},$
such that 
\begin{equation}\label{bounded deterministic controls}
\begin{split}
&a_{1}\leq \!\!\inf_{(x,y)\in \mathcal{Z},t\in \lbrack 0,T]}\tilde{q}_{xy}(t)\leq\!\!
\sup_{(x,y)\in \mathcal{Z},t\in \lbrack 0,T]}\tilde{q}_{xy}(t)\leq a_{2},\hspace{8pt}
\inf_{x\in\mathcal{X}, t\in[0,T]}\tilde{\mu}_{x}(t)>a_{3},\hspace{4pt}\sup_{t%
\in[0,T]}\|\boldsymbol{\mu}(t)-\tilde{\boldsymbol{\mu}}(t)\|<\epsilon, \\
&\hspace{32pt} \text{and}\hspace{32pt} \sum_{(x,y)\in\mathcal{Z}}\int_{0}^{T}\tilde{\mu}_{x}(t)F_{xy}(\tilde{q}%
_{xy}(t))dt\leq \sum_{(x,y)\in\mathcal{Z}}\int_{0}^{T}%
\mu_{x}(t)F_{xy}(q_{xy}(t))dt.
\end{split}%
\end{equation}
\end{lemma}

\begin{proof}
Recall that $\boldsymbol{m}\in \mathcal{P}_{\ast }(\mathcal{X})$ implies $%
\boldsymbol{m}_{x}>0$ for all $x\in \mathcal{X}$. Let $\boldsymbol{\nu }(%
\boldsymbol{m},t)$ be the solution to the equation $\boldsymbol{\dot{\nu}}%
(t)=\boldsymbol{\gamma }\boldsymbol{\nu }(t),$ with initial data $\boldsymbol{m}.$ By
Remark \ref{away}, we know that there exists $1\geq a>0$ such that $%
\boldsymbol{\nu }(\boldsymbol{m},t)\in \mathcal{P}_{a}(\mathcal{X}),$ for
every $t\in \lbrack 0,T]$. We can assume without loss that the right hand
side on the second line of (\ref{bounded deterministic controls}) is greater than zero, since if not true
then the controlled rates are $\boldsymbol{\gamma }$ and the conclusion of
the lemma is automatic. For $\frac{\epsilon }{2}\geq \delta >0,$ let 
\begin{equation}
\boldsymbol{\mu }^{\delta }(\cdot )\doteq \displaystyle\delta \boldsymbol{\nu} (%
\boldsymbol{m},\cdot )+(1-\delta )\boldsymbol{\mu }(\cdot ),  \label{mudelta}
\end{equation}%
and note that $\boldsymbol{\mu }_{x}^{\delta }(t)>0$ for every $t\in \lbrack
0,T]$ and $x\in \mathcal{X}$. Therefore, for $\delta $ as above and $%
(x,y)\in {\mathcal{Z}}$, we can define 
\begin{equation}
q_{xy}^{\delta }(\cdot )=\displaystyle\gamma _{xy}\frac{\delta \nu _{x}(%
\boldsymbol{m},\cdot )}{\mu _{x}^{\delta }(\cdot )}+q_{xy}(\cdot )\frac{%
(1-\delta )\mu _{x}(\cdot )}{\mu _{x}^{\delta }(\cdot )}.  \label{qdelta}
\end{equation}%
Then it is straightforward to check that $(\boldsymbol{\mu }^{\delta },%
\boldsymbol{q}^{\delta })\in \mathcal{T}_{\boldsymbol{m}}.$ Moreover, since $%
\frac{\delta \nu _{x}(\boldsymbol{m},t)}{\mu _{x}^{\delta }(t)}+\frac{%
(1-\delta )\mu _{x}(t)}{\mu _{x}^{\delta }(t)}=1$ for all $t\in \lbrack 0,T]$%
, by the convexity of $F$ we obtain 
\begin{equation*}
\begin{split}
& \sum_{(x,y)\in \mathcal{Z}}\int_{0}^{T}\mu _{x}^{\delta }(t)F_{xy}\left(
q_{xy}^{\delta }(t)\right) dt=\sum_{(x,y)\in \mathcal{Z}}\int_{0}^{T}\mu
_{x}^{\delta }(t)F_{xy}\left( \gamma _{xy}\frac{\delta \nu _{x}(\boldsymbol{m%
},t)}{\mu _{x}^{\delta }(t)}+q_{xy}\frac{(1-\delta )\mu _{x}(t)}{\mu
_{x}^{\delta }(t)}\right) dt \leq\\
&   \!\!\!\!\sum_{(x,y)\in \mathcal{Z}}\!\int_{0}^{T}\!\!\!\mu _{x}^{\delta }(t)%
\frac{\delta \nu _{x}(\boldsymbol{m},t)}{\mu _{x}^{\delta }(t)}\!F_{xy}\left(
\gamma _{xy}\right)\! dt{+}\!\!\!\sum_{(x,y)\in \mathcal{Z}}\int_{0}^{T}\!\!\!\!\mu
_{x}^{\delta }(t)\frac{(1{-}\delta )\mu _{x}(t)}{\mu _{x}^{\delta }(t)}%
F_{xy}\left(\! q_{xy}(t)\!\right) \!dt \leq \!(1{-}\delta )\!\!\!\!\!\sum_{(x,y)\in \mathcal{Z}}\int_{0}^{T}\!\!\!\!\mu
_{x}(t)F_{xy}\left( q_{xy}(t)\right) \!dt,
\end{split}%
\end{equation*}%
where in the second inequality, we used the fact that $F^{\infty }(\gamma
_{xy})=0$ [see Lemma \ref{lemma:fnproperties}]. Therefore, we get a couple $%
(\boldsymbol{\mu }^{\delta },\boldsymbol{q}^{\delta })\in \mathcal{T}_{%
\boldsymbol{m}}$ with cost strictly less than the initial one, and with $%
\boldsymbol{\mu }^{\delta }$ that satisfies 
\begin{equation}
\mu _{x}^{\delta }(t)\geq \delta a\hspace{16pt}\text{and}\hspace{16pt}\frac{%
(1-\delta )\mu _{x}(t)}{\mu _{x}^{\delta }(t)}\leq \frac{(1-\delta )}{\delta
a+(1-\delta )}\equiv c<1,  \label{estimateofmudelta}
\end{equation}%
for all $t\in \lbrack 0,T].$ However, since this couple does not
necessarily satisfy condition (\ref{bounded deterministic controls}), we
modify it even further. Specifically, we pick $M\in (2\gamma _{\max },\infty
)$ big enough such that%
\begin{equation}
\sum_{(x,y)\in \mathcal{Z}}\int_{0}^{T}\mu _{x}^{\delta }(t)\left\vert \min
\left\{ q_{xy}^{\delta }(t),M\right\} -q_{xy}^{\delta }(t)\right\vert dt\leq 
\frac{a\delta (1-\sqrt{c})}{\sqrt{2}},
\label{estimateformdelta}
\end{equation}%
and define 
\begin{equation}
\boldsymbol{\mu }^{\delta ,M}(t)=\int_{0}^{t}\sum_{(x,y)\in \mathcal{Z}}\mu
_{x}^{\delta }(t)\min \left\{ q_{xy}^{\delta }(t),M\right\} \boldsymbol{v}%
_{xy}dt.  \label{eqdef:mudeltam}
\end{equation}%
Then 
\begin{equation}
\begin{split}
\left\vert \mu _{x}^{\delta ,M}(t)-\mu _{x}^{\delta }(t)\right\vert & \leq
\left\Vert \boldsymbol{\mu }^{\delta ,M}(t)-\boldsymbol{\mu }^{\delta
}(t)\right\Vert \overset{(\ref{eqdef:mudeltam})}{=}\left\Vert \sum_{(x,y)\in 
\mathcal{Z}}\int_{0}^{T}\left( \mu _{x}^{\delta }(t)\left( q_{xy}^{\delta
}(t)-\min \left\{ q_{xy}^{\delta }(t),M\right\} \right) \right) \boldsymbol{v%
}_{xy}dt\right\Vert \\
& \leq \sum_{(x,y)\in \mathcal{Z}}\int_{0}^{T}\left\vert \mu _{x}^{\delta
}(t)\left( q_{xy}^{\delta }(t)-\min \left\{ q_{xy}^{\delta }(t),M\right\}
\right) \right\vert \Vert \boldsymbol{v}_{xy}\Vert dt \\
& \leq \sqrt{2}\sum_{(x,y)\in \mathcal{Z}}\int_{0}^{T}\left( \mu
_{x}^{\delta ,M}(t)\left( q_{xy}^{\delta }(t)-\min \left\{ q_{xy}^{\delta
}(t),M\right\} \right) \right) dt\overset{(\ref{estimateformdelta})}{\leq }%
a\delta (1-\sqrt{c}),
\end{split}
\label{estimatesmore}
\end{equation}%
and therefore for $t\in \lbrack 0,T],$ 
\begin{equation}
\mu _{x}^{\delta ,M}(t)\geq \mu _{x}^{\delta }(t)-\left\vert \mu
_{x}^{\delta ,M}(t)-\mu _{x}^{\delta }(t)\right\vert \overset{(\ref%
{estimateofmudelta})}{\geq }a\delta -\left\vert \mu _{x}^{\delta ,M}(t)-\mu
_{x}^{\delta }(t)\right\vert \overset{(\ref{estimatesmore})}{\geq }a\delta 
\sqrt{c}.  \label{estimateofmudeltam}
\end{equation}%
We also get 
\begin{equation*}
\left\vert 1-\frac{\mu _{x}^{\delta ,M}(t)}{\mu _{x}^{\delta }(t)}%
\right\vert \overset{(\ref{estimatesmore})}{\leq }\frac{a\delta (1-\sqrt{c})%
}{\min_{x}\mu _{x}^{\delta }(t)}\overset{(\ref{estimateofmudelta})}{\leq }(1-%
\sqrt{c})
\end{equation*}%
or 
\begin{equation}
\frac{\mu _{x}^{\delta }(t)}{\mu _{x}^{\delta ,M}(t)}\geq \frac{1}{2-\sqrt{c}%
}\hspace{24pt}\text{and}\hspace{24pt}\frac{\mu _{x}^{\delta }(t)}{\mu
_{x}^{\delta ,M}(t)}\leq \frac{1}{\sqrt{c}}=\frac{\sqrt{c}}{c}.
\label{estimateforratioofmu}
\end{equation}%
We deduce that $\boldsymbol{\mu }^{\delta ,M}(t)\in \mathcal{P}_{\ast }(%
\mathcal{X}),$ for all $t\in \lbrack 0,T],$ and therefore can define 
\begin{equation}
q_{xy}^{\delta ,M}(t)=\frac{\min \left\{ q_{xy}^{\delta }(t),M\right\} \mu
_{x}^{\delta }(t)}{\mu _{x}^{\delta ,M}(t)},  \label{qdeltaM}
\end{equation}%
which will give $(\boldsymbol{\mu }^{\delta ,M},\boldsymbol{q}^{\delta
,M})\in \mathcal{T}_{\boldsymbol{m}}.$ {We can see that 
\eqref{bounded
deterministic controls} is satisfied, since by \eqref{qdelta} and the first inequality in %
\eqref{estimateforratioofmu} for the bound from below and the second inequality in %
\eqref{estimateforratioofmu} for the bound from above we have 
\begin{equation*}
\frac{\gamma _{xy}\delta \nu _{x}(\boldsymbol{m},\cdot )}{2}\leq
q_{xy}^{\delta ,M}(t)\leq M\frac{\sqrt{c}}{c}
\end{equation*}%
} It is worth mentioning at this point that trying to get an estimate for
the cost of $(\boldsymbol{\mu }^{\delta ,M},\boldsymbol{q}^{\delta ,M}),$
with respect to the cost of $(\boldsymbol{\mu }^{\delta },\boldsymbol{q}%
^{\delta }),$ would require some extra properties of $F$. However, we can
obtain an estimate of the cost $(\boldsymbol{\mu }^{\delta ,M},\boldsymbol{q}%
^{\delta ,M})$ with respect to the cost of the initial triplet $(\boldsymbol{%
\mu },\boldsymbol{q})$, by utilizing only the convexity of $F_{xy},$ and
choosing the right parameters. Using the fact that $F_{xy}$ is increasing on 
$[\gamma _{xy},\infty )$ in the first inequality, and that $M\mu
_{x}^{\delta }(t)/\mu _{x}^{\delta ,M}(t)\geq \gamma _{xy}$ by (\ref%
{estimateforratioofmu}) and $M\geq 2\gamma _{xy}$, 
\begin{equation}
\begin{split}
&F_{xy}\left( q_{xy}^{\delta ,M}(t)\right) \overset{(\ref{qdeltaM})}{=}
F_{xy}\left( \frac{\min \left\{ q_{xy}^{\delta }(t),M\right\} \mu
_{x}^{\delta }(t)}{\mu _{x}^{\delta ,M}(t)}\right) \leq F_{xy}\left( \frac{%
q_{xy}^{\delta }(t)\mu _{x}^{\delta }(t)}{\mu _{x}^{\delta ,M}(t)}\right)
\overset{(\ref{qdelta})}{=} \\& F_{xy}\left( \frac{\mu _{x}^{\delta }(t)}{\mu
_{x}^{\delta ,M}(t)}\left( \gamma _{xy}\frac{\delta \nu _{x}(\boldsymbol{m}%
,t)}{\mu _{x}^{\delta }(t)}+q_{xy}(t)\frac{(1-\delta )\mu _{x}(t)}{\mu
_{x}^{\delta }(t)}\right)\!\! \right)= F_{xy}\left( \gamma _{xy}\frac{\delta \nu _{x}(\boldsymbol{m}%
,t)}{\mu _{x}^{\delta ,M}(t)}+q_{xy}(t)\frac{(1-\delta )\mu _{x}(t)}{\mu
_{x}^{\delta ,M}(t)}\!\right) .
\end{split}
\label{convex sum}
\end{equation}%
However, from (\ref{estimateofmudelta}) and (\ref{estimateforratioofmu}), we
have 
\begin{equation*}
\frac{(1-\delta )\mu _{x}(t)}{\mu _{x}^{\delta ,M}(t)}=\frac{(1-\delta )\mu
_{x}(t)}{\mu _{x}^{\delta }(t)}\frac{\mu _{x}^{\delta }(t)}{\mu _{x}^{\delta
,M}(t)}\leq c\frac{\sqrt{c}}{c}=\sqrt{c}<1.
\end{equation*}%
Therefore using the convexity of $F$ we have 
\begin{equation}
\begin{split}
& \!F_{xy}\!\left(\!\! \gamma _{xy}\frac{\delta \nu _{x}(\boldsymbol{m},t)}{\mu
_{x}^{\delta ,M}(t)}+q_{xy}(t)\frac{(1-\delta )\mu _{x}(t)}{\mu _{x}^{\delta
,M}(t)}\!\!\right)\! =\!\!F_{xy}\!\left( \!\!\frac{\left( 1-\frac{(1-\delta )\mu
_{x}(t)}{\mu _{x}^{\delta ,M}(t)}\right) }{\left( 1-\frac{(1-\delta )\mu
_{x}(t)}{\mu _{x}^{\delta ,M}(t)}\right) }\gamma _{xy}\frac{\delta \nu _{x}(%
\boldsymbol{m},t)}{\mu _{x}^{\delta ,M}(t)}+q_{xy}(t)\frac{(1-\delta )\mu
_{x}(t)}{\mu _{x}^{\delta ,M}(t)}\!\!\right) \\
& \quad \leq \left( 1-\frac{(1-\delta )\mu _{x}(t)}{\mu _{x}^{\delta ,M}(t)}%
\right) F_{xy}\left( \gamma _{xy}\frac{\delta \nu _{x}(\boldsymbol{m},t)}{%
\mu _{x}^{\delta ,M}(t)-(1-\delta )\mu _{x}(t)}\right) +\frac{(1-\delta )\mu
_{x}(t)}{\mu _{x}^{\delta ,M}(t)}F_{xy}\left( q_{xy}(t)\right) .
\end{split}
\label{combined estimate}
\end{equation}

Combining (\ref{convex sum}) and (\ref{combined estimate}) and then using %
\eqref{mudelta}, we obtain 
\begin{equation}
\begin{split}
& \mu _{x}^{\delta ,M}(t)F_{xy}\left( q_{xy}^{\delta ,M}(t)\right) \\
&  \leq \left( \mu _{x}^{\delta ,M}(t)-(1-\delta )\mu _{x}(t)\right)
F_{xy}\left( \gamma _{xy}\frac{\delta \nu _{x}(\boldsymbol{m},t)}{\mu
_{x}^{\delta ,M}(t)-(1-\delta )\mu _{x}(t)}\right) +(1-\delta )\mu
_{x}(t)F_{xy}\left( q_{xy}(t)\right) \\
&  =\left(\! \mu _{x}^{\delta ,M}(t)-\mu _{x}^{\delta }(t)+\delta \nu
_{x}(M,t)\!\right)\! F_{xy}\left(\!\! \gamma _{xy}\frac{\delta \nu _{x}(\boldsymbol{m%
},t)}{\mu _{x}^{\delta ,M}(t)-\mu _{x}^{\delta }(t)+\delta \nu _{x}(%
\boldsymbol{m},t)}\right) +(1-\delta )\mu _{x}(t)F_{xy}\left(
q_{xy}(t)\right) .
\end{split}
\label{I dont know}
\end{equation}%
We can make $|\mu _{x}^{\delta ,M}(t)-\mu _{x}^{\delta }(t)|$ uniformly as
close to zero as desired and therefore we can make the quantity $\gamma _{xy}%
\frac{\delta \nu _{x}(\boldsymbol{m},t)}{\mu _{x}^{\delta ,M}(t)-\mu
_{x}^{\delta }(t)+\delta \nu _{x}(\boldsymbol{m},t)}$ as close to $\gamma
_{xy}$ as desired by picking $M$ sufficiently large. Since $F_{xy}\left(
\gamma _{xy}\right) =0$ and $F_{xy}\left( \cdot \right) $ is continuous on $%
(0,\infty )$ by Lemma \ref{lemma:fnproperties2}, we can pick $M<\infty $
such that for every $t\in \lbrack 0,T],$

\begin{equation}
F_{xy}\left( \gamma _{xy}\frac{\delta \nu _{x}(\boldsymbol{m},t)}{\mu
_{x}^{\delta ,M}(t)-\mu _{x}^{\delta }(t)+\delta \nu _{x}(\boldsymbol{m},t)}%
\right) \leq \frac{1}{2T}\int_{0}^{T}\mu _{x}(s)F_{xy}(q_{xy}(s))ds.
\label{estimate around gamma}
\end{equation}

Then from (\ref{I dont know}) and (\ref{estimate around gamma}) and the fact
that $\nu _{x}(m,t)\leq 1$ and (\ref{estimatesmore}), for $t\in \lbrack 0,T]$
\begin{equation*}
\begin{split}
& \sum_{(x,y)\in \mathcal{Z}}\int_{0}^{T}\mu _{x}^{\delta ,M}(t)F_{xy}\left(
q_{xy}^{\delta ,M}(t)\right) dt \leq \sum_{(x,y)\in \mathcal{Z}}
\int_{0}^{T}\left( 2\delta \right) \left( \frac{1}{2T}\int_{0}^{T}\mu
_{x}(s)F_{xy}(q_{xy}(s))ds\right) dt  \\
& +\int_{0}^{T}(1-\delta )\mu _{x}(t)F_{xy}\left( q_{xy}(t)\right) dt
=\sum_{(x,y)\in \mathcal{Z}}\int_{0}^{T}\mu_{x}(t)F_{xy}\left(
q_{xy}(t)\right) dt.
\end{split}%
\end{equation*}
\end{proof}

Next, we are going to prove the following result.

\begin{lemma}[Law of large numbers]
\label{lolna}Let $T\in (0,\infty )$ be given. There exists a constant $%
c<\infty $ such that if $(\boldsymbol{\mu }^{n},\boldsymbol{\gamma })\in 
\mathcal{T}_{\boldsymbol{m}}^{n}$ (see (\ref{eqdef:calTn})), and $(%
\boldsymbol{\nu },\boldsymbol{\gamma })\in \mathcal{T}_{\boldsymbol{m}},$
then 
\begin{equation}
\mathbb{P}\left( \sup_{t\in \lbrack 0,T]}\Vert \boldsymbol{\mu }^{n}(t)-%
\boldsymbol{\nu }(\boldsymbol{m},t)\Vert \geq \epsilon \right) \leq \frac{c}{%
\epsilon \sqrt{n}}.  \label{loln}
\end{equation}
\end{lemma}

\begin{proof}
We have 
\begin{align*}
\left\Vert \boldsymbol{\mu }^{n}(t)-\boldsymbol{\nu }(\boldsymbol{m}%
,t)\right\Vert & \leq \sum_{(x,y)}\left\vert \int_{0}^{t}\int_{0}^{\infty
}1_{[0,\mu _{x}^{n}(s)\gamma _{xy}]}(r)\frac{1}{n}N_{xy}^{n}(dsdr)-%
\int_{0}^{t}\int_{0}^{\infty }1_{[0,\nu _{x}(m,s)\gamma
_{xy}]}(r)dsdr\right\vert \\
& \leq \sum_{(x,y)}\left\vert \int_{0}^{t}\int_{0}^{\infty }1_{[0,\mu
_{x}^{n}(s)\gamma _{xy}]}(r)\frac{1}{n}N_{xy}^{n}(dsdr)-\int_{0}^{t}%
\int_{0}^{\infty }1_{[0,\mu _{x}^{n}(s)\gamma _{xy}]}(r)dsdr\right\vert \\
& \quad +\sum_{(x,y)}\left\vert \int_{0}^{t}\int_{0}^{\infty }1_{[0,\mu
_{x}^{n}(s)\gamma _{xy}]}(r)dsdr-\int_{0}^{t}\int_{0}^{\infty }1_{[0,\nu
_{x}(m,s)\gamma _{xy}]}(r)ds\right\vert .
\end{align*}%
For a constant $K$ that depends on $d$ and the maximum of $\gamma _{xy},$%
\begin{equation*}
\sum_{(x,y)}\left\vert \int_{0}^{t}\int_{0}^{\infty }1_{[0,\mu
_{x}^{n}(s)\gamma _{xy}]}(r)dsdr-\int_{0}^{t}\int_{0}^{\infty }1_{[0,\nu
_{x}(m,s)\gamma _{xy}]}(r)ds\right\vert \leq K\sup_{0\leq s\leq t}\left\Vert 
\boldsymbol{\mu }^{n}(s)-\boldsymbol{\nu }(\boldsymbol{m},s)\right\Vert .
\end{equation*}%
Hence by Gronwall's inequality, for $r\in \lbrack 0,T]$%
\begin{equation*}
\left\Vert \boldsymbol{\mu }^{n}(r)-\boldsymbol{\nu }(\boldsymbol{m}%
,r)\right\Vert \leq e^{KT}\sup_{0\leq t\leq r}\sum_{(x,y)}\left\vert
\int_{0}^{t}\int_{0}^{\infty }1_{[0,\mu _{x}^{n}(s)\gamma _{xy}]}(r)\frac{1}{%
n}N_{xy}^{n}(dsdr)-\int_{0}^{t}\int_{0}^{\infty }1_{[0,\mu _{x}^{n}(s)\gamma
_{xy}]}(r)dsdr\right\vert .
\end{equation*}%
Using the Burkholder-Gundy-Davis inequality as was done to obtain (\ref{BGD}%
), 
\begin{equation*}
\mathbb{P}\left( \sup_{t\in \lbrack 0,T]}\left\vert \int_{0}^{t}\int_{0}^{\infty
}1_{[0,\mu _{x}^{n}(s)\gamma _{xy}]}(r)\frac{1}{n}N_{xy}^{n}(dsdr)-%
\int_{0}^{t}\int_{0}^{\infty }1_{[0,\mu _{x}^{n}(s)\gamma
_{xy}]}(r)dsdr\right\vert \geq \epsilon \right) \leq \frac{\bar{c}}{\epsilon 
\sqrt{n}},
\end{equation*}%
and hence 
\begin{equation*}
\mathbb{P}\left( \sup_{t\in \lbrack 0,T]}\left\Vert \boldsymbol{\mu }^{n}(t)-%
\boldsymbol{\nu }(\boldsymbol{m},t)\right\Vert \geq \epsilon \right) \leq
d^{2}\frac{e^{KT}\bar{c}}{\epsilon \sqrt{n}},
\end{equation*}%
which is (\ref{loln}).
\end{proof}

We now obtain the following result.

\begin{lemma}
\label{Vuniformlybounded}The sequence $V^{n}(\boldsymbol{m})$ is bounded,
uniformly in $n$ and $\boldsymbol{m}\in \mathcal{P}(X).$
\end{lemma}

\begin{proof}
Let $\tau =diameter(\mathcal{P}(\mathcal{X})).$ By Remark \ref{away}, there
exists $a>0$ such that $\boldsymbol{\nu }(\boldsymbol{m},\tau )\in 
\mathcal{P}_{2a}(\mathcal{X})$ regardless of the initial data $%
\boldsymbol{m}$. We can further assume that $\mathcal{P}_{a}(\mathcal{X%
})\cap K^{\circ }\neq \emptyset ,$ and in particular that there exists an
element $\tilde{\boldsymbol{m}}$ such that $B(\tilde{\boldsymbol{m}},a%
/2)\subset \mathcal{P}_{a}(\mathcal{X})\cap K^{\circ }$. 

Since $\tilde{\boldsymbol{m}}\in \mathcal{P}_{a}(\mathcal{X})$, the first part of Theorem %
\ref{lemma:uniFUont} implies that for every point $\boldsymbol{m}$ in $%
\mathcal{P}_{a}(\mathcal{X})$ we can find a control $\boldsymbol{q}_{\boldsymbol{m}}$
with the following properties: there is a unique $\boldsymbol{\mu }$ such
that $(\boldsymbol{\mu },\boldsymbol{q}_{\boldsymbol{m}})\in \mathcal{T}_{%
\boldsymbol{\boldsymbol{m}}}$; $\boldsymbol{\mu }$ is a constant speed
parametrization of the straight line that connects $\boldsymbol{m}$ to $%
\tilde{\boldsymbol{m}}$ in time $T_{\{\tilde{\boldsymbol{m}}\}}=\Vert 
\boldsymbol{m}-\tilde{\boldsymbol{m}}\Vert $; and the control $\boldsymbol{q}%
_{\boldsymbol{m}}$ satisfies 
\begin{equation*}
\gamma _{xy}\leq q_{\boldsymbol{m},xy}(t)\leq \frac{c_{1}}{a},
\end{equation*}%
for $t\in \lbrack 0,T_{\{\tilde{\boldsymbol{m}}\}}],(x,y)\in \mathcal{Z},$
where $c_{1}>0$ is a constant that does not depend on $a.$ For every $%
\boldsymbol{m},$ we let 
\begin{equation*}
q_{xy}(\boldsymbol{m},t)=%
\begin{cases}
q_{\boldsymbol{m},xy}(t) & t\leq \Vert \boldsymbol{m}-\tilde{\boldsymbol{m}}%
\Vert , \\ 
\gamma _{xy} & t>\Vert \boldsymbol{m}-\tilde{\boldsymbol{m}}\Vert ,%
\end{cases}
\end{equation*}%
denote the control that takes $\boldsymbol{m}$ to $\tilde{\boldsymbol{m}}$
in time $\Vert \boldsymbol{m}-\tilde{\boldsymbol{m}}\Vert ,$ in the sense
that it was described above, and after that time is equal to the original
rates.

For $i\in \mathbb{N}$ we define a control for the interval $i\tau \leq
t<(i+1)\tau $ as follows. Let $f(t-)$ denote the limit of $f(s)$ from the
left at time $t$, and recall that $\boldsymbol{\mu }(\boldsymbol{m},\cdot )$
is the straight line that connects $\boldsymbol{m}$ to $\tilde{\boldsymbol{m}%
}$ in time $T_{\{\tilde{\boldsymbol{m}}\}}$, where $\tilde{\boldsymbol{m}}$
is fixed and we explicitly indicate the dependence on $\boldsymbol{m}$. Then
set%
\begin{equation*}
q_{xy}^{n}(t)=%
\begin{cases}
q_{xy}(\boldsymbol{m},t-i\tau )\frac{\mu _{x}^{n}(t-)}{\mu _{x}(\boldsymbol{m%
},t-i\tau )}, & \text{if}\,\left( \sup_{s\in \left[ i\tau ,t\right] }\Vert 
\boldsymbol{\mu }(\boldsymbol{m},t)-\boldsymbol{\mu }^{n}(t)\Vert \leq \frac{%
a}{2}\right) \ \text{and }(\boldsymbol{\mu }^{n}(i\tau )=\boldsymbol{m}\in 
\mathcal{P}_{a}(\mathcal{X})) \\ 
\gamma _{xy}, & \text{otherwise}.
\end{cases}%
\end{equation*}%
The idea with these controls is that, within each time interval with length $%
\tau $, the control considers the starting point $\boldsymbol{m}$, and then
if $\boldsymbol{m}\in \mathcal{P}_{a}(\mathcal{X}) $, it attempts to force
the process to follow the straight line to $\tilde{\boldsymbol{m}}.$ If $%
m\notin\mathcal{P}_{a}(\mathcal{X})$ or the process goes close to the
boundary of the simplex $\mathcal{P}(\mathcal{X})\setminus \mathcal{P}_{\ast
}(\mathcal{X}),$ then we just use original rates to push the process inside $%
\mathcal{P}_{a}(\mathcal{X}).$ Since all controls used are bounded from
above and below, the total cost is a multiple of $\mathbb{E}[T^{n}]$. Thus
we need only show this expected exit time is uniformly bounded.

By using (\ref{BGD}), we can find constant $c<\infty $ such that%
\begin{equation*}
\mathbb{P}\left( \sup_{t\in \lbrack i\tau ,(i+1)\tau ]}\Vert \boldsymbol{\mu 
}^{n}(t)-\boldsymbol{\mu }(\boldsymbol{m},t)\Vert \geq \frac{a}{2}\,\Bigg|\,%
\boldsymbol{\mu }^{n}(i\tau )=\boldsymbol{m}\in \mathcal{P}_{a}(\mathcal{X}%
)\right) \leq c\frac{2}{\sqrt{n}a},  
\end{equation*}%
from which we get 
\begin{equation*}
\begin{split}
\!\!\!\mathbb{P}(T^{n}>(i+1)\tau |\boldsymbol{\mu }^{n}(i\tau )\in \mathcal{P}_{a}(X))\!
\leq\!\!\! \inf_{\boldsymbol{m}\in \mathcal{P}_{a}(\mathcal{X})}\!\!\!\mathbb{P}\!\left(\!
\sup_{t\in \lbrack i\tau ,(i+1)\tau ]}\!\!\Vert \boldsymbol{\mu }^{n}(t)-%
\boldsymbol{\mu }(\boldsymbol{m},t)\Vert \geq \frac{a}{2}\,\Bigg|\,%
\boldsymbol{\mu }^{n}(i\tau )=\boldsymbol{m}\!\right)\! \! \leq\! c\frac{2}{\sqrt{n}a}.
\end{split}
\end{equation*}%
By Lemma \ref{lolna}, we have that for some $c^{\prime }<\infty $%
\begin{equation*}
\mathbb{P}\left( \sup_{t\in \lbrack i\tau ,(i+1)\tau ]}\Vert \boldsymbol{\mu 
}^{n}(t)-\boldsymbol{\nu }(\boldsymbol{m},t)\Vert \geq \frac{a}{2}\,\Bigg|\,%
\boldsymbol{\mu }^{n}(i\tau )=\boldsymbol{m}\notin \mathcal{P}_{a}(\mathcal{X%
})\right) \leq c^{\prime }\frac{2}{a\sqrt{n}},  
\end{equation*}%
which implies that 
\begin{equation*}
\begin{split}
& \mathbb{P}\!\left(\!\! \boldsymbol{\mu }^{n}((i+1)\tau \!)\notin \!\mathcal{P}_{a}(%
\mathcal{X})\bigg|\boldsymbol{\mu }^{n}(i\tau )\notin \!\mathcal{P}_{a}(%
\mathcal{X})\!\!\right) {\leq} \!\!\!\inf_{\boldsymbol{m}\notin \mathcal{P}_{a}(\mathcal{X})}\!\!\mathbb{P}
\!\left(\! \sup_{t\in \lbrack i\tau ,(i+1)\tau ]}\!\!\!\Vert \boldsymbol{\mu }^{n}(t)-%
\boldsymbol{\nu }(\boldsymbol{m},t)\Vert {\geq} \frac{a}{2}\,\Bigg|\,%
\!\boldsymbol{\mu }^{n}(i\tau ){=}\boldsymbol{m}\right) {\leq} \frac{2c^{\prime }}{%
a\sqrt{n}}.
\end{split}
\end{equation*}%
Thus the probability to escape in the next $2\tau $ units of time has a
positive lower bound that is independent of $n$ and the starting position.
This implies the uniform upper bound on the mean escape time.
\end{proof}

Now we proceed with the proof of the upper bound.

\begin{proof}[Proof of upper bound]
We will initially assume that $\boldsymbol{m}$ is in $\mathcal{P}_{a}(%
\mathcal{X}),$ for some $a>0.$ Recall that $V_{K}(\boldsymbol{m})<\infty $.
Let $\epsilon >0$. By the definition of $V_{K}(\boldsymbol{m}),$ we can find
a pair $(\boldsymbol{\mu },\boldsymbol{q})\in \mathcal{T}_{\boldsymbol{m}%
} $ and a $T\in \lbrack 0,\infty ],$ such that 
\begin{equation*}
\int_{0}^{T}\left( \sum_{(x,y)\in \mathcal{Z}}\mu _{x}(t)F_{xy}\left(
q_{xy}(t)\right) +R(\boldsymbol{\mu }(t))\right) dt\leq V_{K}(\boldsymbol{m}%
)+\epsilon .
\end{equation*}%
Since we assumed that $R$ is bounded from below by a positive constant for
every compact subset of $K^{c}$, we can furthermore find a $\delta $ such
that for finite time $T^{\delta }\in \lbrack 0,\infty )$ we have%
\begin{equation*}
\int_{0}^{T^{\delta }}\left( \sum_{(x,y)\in \mathcal{Z}}\mu
_{x}(t)F_{xy}\left( q_{xy}(t)\right) +R(\boldsymbol{\mu }(t))\right) dt\leq
V_{K}(\boldsymbol{m})+\epsilon ,
\end{equation*}%
and $d(\boldsymbol{\mu }(T^{\delta }),K)\leq \delta $. By the second part of
Theorem \ref{lemma:uniFUont}, we can extend the path so it can reach a point 
$\tilde{\boldsymbol{m}}$ of $K,$ with extra cost less than $\epsilon .$
Since $K=\overline{(K^{\circ })},$ by a second application of Theorem \ref%
{lemma:uniFUont}, we can assume that $\tilde{\boldsymbol{m}}$ is an internal
point of $K,$ by again adding an extra cost less than $\epsilon .$

Let $r>0$ be such that $B(\tilde{\boldsymbol{m}},r)\subset K^{\circ }.$ From
Lemma \ref{lemma:awayfromboundary}, without any loss of generality, we can
assume that there exist $a_{1},a_{2},a_{3}\in (0,\infty )$ such that

\begin{equation}\label{the qs}
a_{1}\leq \inf_{(x,y)\in \mathcal{Z},t\in \lbrack 0,S]}q_{xy}(t)\leq
\sup_{(x,y)\in \mathcal{Z},t\in \lbrack 0,S]}q_{xy}(t)\leq a_{2}, 
\inf_{x\in\mathcal{X}, t\in[0,S]}\mu_{x}(t)>a_{3},\hspace{4pt}\|\boldsymbol{%
\mu}(T)-\tilde{\boldsymbol{m}}\|<\frac{r}{2},
\end{equation}
where the $S$ used above is the one obtained by starting with $T^{\delta }$
and adding segments as just described. Finally, we can assume the existence of a $r_{1}>0$
such that for every point $\bar{\boldsymbol{m}}$ in $B(%
\boldsymbol{m},r_{1}),$ we can find a path like the one described above, by
connecting $\bar{\boldsymbol{m}}$ with a straight line to $\boldsymbol{m}.$
Of course this could generate $a_{1},a_{2},a_{3},S$ different from the initial ones, though universal for all  $\bar{\boldsymbol{m}}$ in $B(%
\boldsymbol{m},r_{1}),$  (see Theorem \ref{lemma:uniFUont}
for details).

Now let $\boldsymbol{m}^{n}$ be a sequence that converges to $\boldsymbol{m}%
. $ For big enough $n,$ we can assume that $\boldsymbol{m}^{n}\in B(%
\boldsymbol{m},r_{1}).$ By the continuity of $F$ on compact subsets of $%
(0,\infty ),$ we can find $r_{2}>0$ such that if $\boldsymbol{m}_{1},%
\boldsymbol{m}_{2}\in \mathcal{P}_{{\frac{a_{3}}{2}}}(\mathcal{X})$ and $%
\Vert \boldsymbol{m}_{1}-\boldsymbol{m}_{2}\Vert \leq r_{2},$ then for every 
$\boldsymbol{q}$ that satisfies (\ref{the qs}), we have

\begin{equation}
\sum_{(x,y)\in \mathcal{Z}}\left\vert
m_{1,x}F_{xy}(q_{xy})-m_{2,x}F_{xy}\left( q_{xy}\frac{m_{1,x}}{m_{2,x}}%
\right) \right\vert \leq \frac{\epsilon }{S}.
\label{close estimate}
\end{equation}

Now for every $n\in \mathbb{N},$ we define the following control for the
time interval $[0,S],$

\begin{equation}
q_{xy}^{n}(t)=%
\begin{cases}
q_{xy}(t)\frac{\mu _{x}^{n}(t-)}{\mu _{x}(t)}, & \text{if}\,\sup_{s\in
\lbrack 0,t]}\Vert \boldsymbol{\mu }(t)-\boldsymbol{\mu }^{n}(t)\Vert \leq
r_{2} \\ 
\gamma _{xy}, & \text{otherwise}.%
\end{cases}
\label{qnfor}
\end{equation}%
Note that either $\boldsymbol{\mu }^{n}$ enters $K$ by time $S$, or the
control has switch to $\gamma _{xy}$ before $S$. For every $n,$ we define an
auxiliary stopping time $S^{n}=\inf \{t\in \lbrack 0,S]:\Vert 
\boldsymbol{\mu }^{n}(t)-\boldsymbol{\mu }(t)\Vert >r_{2}\},$ and also we
define $R_{max}=\sup_{\boldsymbol{m}\in \mathcal{P}(\mathcal{X})}R(%
\boldsymbol{m})$. We can get an estimate of the cost accumulated up to time $%
S,$ for the pair $(\boldsymbol{\mu }^{n},\boldsymbol{q}^{n})\in \mathcal{T}%
_{\boldsymbol{m}^{n}}^{n}$. Specifically,%
\begin{equation*}
\begin{split}
& \mathbb{E}\left[ \int_{0}^{S}\left( \sum_{(x,y)\in \mathcal{Z}}\mu
_{x}^{n}(t)F_{xy}\left( q_{xy}^{n}(t)\right) +R(\boldsymbol{\mu }%
^{n}(t))\right) dt\right] \\
& \leq \mathbb{E}\left[ \int_{0}^{S}\left( \sum_{(x,y)\in \mathcal{Z}}\mu
_{x}^{n}(t)F_{xy}\left( q_{xy}^{n}(t)\right) +R(\boldsymbol{\mu }%
^{n}(t))\right) dt\cdot 1_{\left\{ \sup_{t\in \lbrack 0,S]}\Vert 
\boldsymbol{\mu }(t)-\boldsymbol{\mu }^{n}(t)\Vert \leq r_{2}\right\} }%
\right] \\
& +\mathbb{P}\left( \sup_{t\in \lbrack 0,S]}\Vert \boldsymbol{\mu }^{n}(t)-%
\boldsymbol{\mu }(t)\Vert >r_{2}\right) \times \\
& \left( \mathbb{E}\left[ \int_{0}^{S^{n}}\left( \sum_{(x,y)\in 
\mathcal{Z}}\mu _{x}^{n}(t)F_{xy}\left( q_{xy}^{n}(t)\right) +R(\boldsymbol{%
\mu }^{n}(t))\right) dt\Bigg |\sup_{t\in \lbrack 0,S]}\Vert \boldsymbol{\mu }%
^{n}(t)-\boldsymbol{\mu }(t)\Vert >r_{2}\right] +SR_{max}\right)
\end{split}%
\end{equation*}
Now by (\ref{qnfor}) the last quantity is equal to 
\begin{equation*}
\begin{split}
&\mathbb{E}\left[ \int_{0}^{S}\left( \sum_{(x,y)\in \mathcal{Z}}\mu
_{x}^{n}(t)F_{xy}\left( q_{xy}(t)\frac{\mu _{x}^{n}(t-)}{\mu _{x}(t)}\right)
+R(\boldsymbol{\mu }^{n}(t))\right) dt\cdot 1_{\left\{ \sup_{t\in \lbrack
0,S]}\Vert \boldsymbol{\mu }(t)-\boldsymbol{\mu }^{n}(t)\Vert \leq
r_{2}\right\} }\right] \\
& +\mathbb{P}\left( \sup_{t\in \lbrack 0,S]}\Vert \boldsymbol{\mu }^{n}(t)-%
\boldsymbol{\mu }(t)\Vert >r_{2}\right) \times \\
& \left(\!\! \mathbb{E}\!\left[\! \int_{0}^{S^{n}}\left(
\sum_{(x,y)\in \mathcal{Z}}\mu _{x}^{n}(t)F_{xy}\left( q_{xy}(t)\frac{\mu
_{x}^{n}(t-)}{\mu _{x}(t)}\right) +R(\boldsymbol{\mu }^{n}(t))\right) dt%
\Bigg |\sup_{t\in \lbrack 0,S]}\Vert \boldsymbol{\mu }(t)-\boldsymbol{\mu }%
^{n}(t)\Vert >r_{2}\right] +SR_{max}\right).
\end{split}%
\end{equation*}

Then using (\ref{close estimate}) with $m_{1,x}=\mu _{x}(t),m_{2,x}=\mu
_{x}^{n}(t-),$ for big enough $n$ we can bound 
\begin{equation*}
\mathbb{E}\left[ \int_{0}^{T}\left( \sum_{(x,y)\in \mathcal{Z}}\mu
_{x}^{n}(t)F_{xy}\left( q_{xy}^{n}(t)\right) +R(\boldsymbol{\mu }%
^{n}(t))\right) dt\right]
\end{equation*}%
above by

\begin{equation*}
V_{K}(\boldsymbol{m})+2\epsilon +\mathbb{P}\left( \sup_{t\in \lbrack
0,S]}\Vert \boldsymbol{\mu }^{n}(t)-\boldsymbol{\mu }(t)\Vert >r_{2}\right)
(V_{K}(\boldsymbol{m})+SR_{\max }+2\epsilon ).
\end{equation*}%
By using (\ref{BGD}), the probability that there was no exit in the time
interval $[0,S]$ is

\begin{equation*}
\mathbb{P}(S^{n}\geq{S})\leq \mathbb{P}\left(\sup_{t\in[0,S] }\|\boldsymbol{%
\mu}^{n}(t)-\boldsymbol{\mu}(t)\|> r_{2}\right)\leq c\frac{1}{\sqrt{n}r_{2}}.
\end{equation*}

Letting $V_{\max }$ be the upper bound identified in Lemma \ref%
{Vuniformlybounded} for the given $a>0$, the total cost satisfies 
\begin{equation*}
\begin{split}
&V_{K}^{n}(\boldsymbol{m}^{n}) \leq \mathbb{E}\left[ \int_{0}^{S}\left(
\sum_{(x,y)\in \mathcal{Z}}\mu _{x}^{n}(t)F_{xy}\left( q_{xy}^{n}(t)\right)
+R(\boldsymbol{\mu }^{n}(t))\right) dt+V(\boldsymbol{\mu }^{n}(S\wedge
S^{n}))\right] \\
& \leq V_{K}(\boldsymbol{m})+2\epsilon +\mathbb{P}\left( \sup_{t\in \lbrack
0,S]}\Vert \boldsymbol{\mu }^{n}(t)-\boldsymbol{\mu }(t)\Vert >r_{2}\right)
(V_{K}(\boldsymbol{m})+SR_{\max }+2\epsilon )+\mathbb{P}(S^{n}\geq {S}%
)V_{max} \\
& \leq V_{K}(\boldsymbol{m})+2\epsilon +2(SR_{\max }+V_{\max }+2\epsilon )\frac{c%
}{\sqrt{n}r_{2}}.
\end{split}%
\end{equation*}
By sending $n$ to infinity we get the upper bound if $\boldsymbol{m}\in 
\mathcal{P}_{a}(\mathcal{X})$ for some $a>0$. Next let $\boldsymbol{m}\in 
\mathcal{P}(\mathcal{X})\setminus \mathcal{P}_{\ast }(\mathcal{X})$. Let $%
t_{0}\leq \epsilon $ be such that $V_{K}(\boldsymbol{\nu }(\boldsymbol{m}%
,t_{0}))\leq V_{K}(\boldsymbol{m})+\epsilon , $ where $\boldsymbol{\nu }(%
\boldsymbol{m},t)$ is the solution to the original equation after time $t.$
We can find a $r>0$ such that for every $\tilde{\boldsymbol{m}}\in B(%
\boldsymbol{\nu }(\boldsymbol{m},t_{0}),r),$ $V_{K}(\tilde{\boldsymbol{m}}%
)\leq V_{K}(\boldsymbol{m})+2\epsilon .$ If $q^{n}(\bar{\boldsymbol{m}},t)$
is an $\epsilon $-optimal control that corresponds to each initial condition 
$\bar{\boldsymbol{m}}$, we define the control 
\begin{equation*}
q_{xy}^{n}(t)=%
\begin{cases}
\gamma _{xy}, & t\leq t_{0}, \\ 
q_{xy}^{n}(\boldsymbol{\mu }^{n}(t_{0}),t-t_{0}), & t>t_{0},%
\end{cases}%
\end{equation*}%
which gives 
\begin{equation*}
\begin{split}
& V_{K}^{n}(\boldsymbol{m}^{n})\leq \mathbb{E}\left[ \int_{0}^{T^{n}}\left(
\sum_{(x,y)\in \mathcal{Z}}\mu _{x}^{n}(s)F_{xy}\left( q_{xy}^{n}(s)\right)
+R(\boldsymbol{\mu }^{n}(s))\right) dt\right] \\
& \leq \mathbb{E}\left[ \int_{0}^{t_{0}}\left( \sum_{(x,y)\in \mathcal{Z}%
}\mu _{x}^{n}(s)F_{xy}\left( \gamma _{xy}^{n}(s)\right) +R(\boldsymbol{\mu }%
^{n}(s))\right) dt\right] \\
& +\mathbb{E}\left[ \int_{t_{0}}^{T^{n}}\left( \sum_{(x,y)\in \mathcal{Z}%
}\mu _{x}^{n}(s)F_{xy}\left( q_{xy}^{n}(\boldsymbol{\mu }%
^{n}(t_{0}),s-t_{0})\right) +R(\boldsymbol{\mu }^{n}(s))\right) dt\right]
\leq t_{0}R_{\max }+\mathbb{E}\left[ V(\boldsymbol{\mu }^{n}(t_{0}))\right]
\\
& \overset{Lemma\,\ref{Vuniformlybounded}}{\leq }\epsilon R_{\max }+P\left( 
\boldsymbol{\mu }^{n}(t_{0})\in B(\boldsymbol{\nu }(\boldsymbol{m}%
,t_{0}),r)\right) (V_{K}(\boldsymbol{m})+2\epsilon )+P\left( \boldsymbol{\mu 
}^{n}(t_{0})\notin B(\boldsymbol{\nu }(\boldsymbol{m},t_{0}),r)\right)
V_{\max } \\
& \leq V_{K}(\boldsymbol{m})+(2+R_{\max })\epsilon +P\left( \boldsymbol{\mu }%
^{n}(t_{0})\notin B(\boldsymbol{\nu }(\boldsymbol{m},t_{0}),r)\right)
V_{\max }.
\end{split}%
\end{equation*}%
Now by an application of Lemma \ref{lolna}, we get that the last term goes
to zero as $n$ goes to $\infty ,$ and since $\epsilon $ is arbitrary, we get
that $\limsup V_{K}^{n}(\boldsymbol{m}^{n})\leq V_{K}(\boldsymbol{m}).$

\end{proof}

\appendix
\section{Properties of Hamiltonians}\label{ap:a}

In this section we establish Lemma \ref{lemma:isaac} and Theorem \ref{minmax}. We start with the proof of Lemma \ref{lemma:isaac}.

\begin{proof}[Proof of Lemma \ref{lemma:isaac}] 
	To prove the exchange between supremum and infimum, we will apply a modification of  Sion's Theorem
	(Corollary 3.3 in \cite{Sion1958}), which states that if a continuous
	$G(u,q)$ is quasi-concave for every $u$ is some convex set $\mathcal U$ and
	quasi-convex for every $q$ in some convex set $\mathcal Q,$ and if one of the two
	sets is compact, then we can exchange the supremum with the infimum.
	We start by investigating the validity of these properties   when $ G= L_{xy}.$ Since $\ell $ is convex, for each $u \geq 0$, 
	\begin{equation*}
	L_{xy}(u,q)=q\xi +u\ell \left( \frac{q}{u}\right) -\gamma _{xy}C_{xy}\left( 
	\frac{u}{\gamma _{xy}}\right)
	\end{equation*}%
	is convex with respect to $q$.  It is easy to see that $u \mapsto L_{xy}(u,q)$ is not
	concave for each $q \geq 0$.  However we now show
	that under Assumption \ref{assumption},  for each $q \geq 0$,
	$u \mapsto L_{xy}(u,q)$ is quasi-concave, or equivalently, that  $\{u \geq 0:L_{xy}(u,q)\geq c\}$
	is convex for every  $c\in \mathbb{R}.$ 
	By differentiating with respect to $u$ we get 
	\begin{equation*}
	\partial _{u}L_{xy}(u,q)=-\frac{q}{u}+1-(C_{xy})^{\prime }\left( \frac{u}{%
		\gamma _{xy}}\right) .
	\end{equation*}%
	If we prove that for each $q$ the set of roots for $\partial _{u}L_{xy}(u,q)$
	is an interval or a point we are done, because a real function that changes
	monotonicity from increasing to decreasing at most once is quasi-concave.
	However $\partial _{u}L_{xy}(u,q)$ has the same roots as $Q(u) = u(C_{xy})^{\prime
	}\left( {\frac{u}{\gamma _{xy}}}\right) -u+q$. By part 1 of Assumption \ref{assumption},
	$Q(u)$ is increasing, which gives what is needed.
	
	Thus, we are almost in a situation where we can apply Sion's theorem, except
	that our  sets are $[0,\infty)$ and hence, non-compact.
	However, as we explain below, we can still apply this result
	by  using the fact that $\lim_{q\rightarrow
		\infty }L_{xy}(q,1)=\infty $. 
	If we prove that 
	\begin{equation*}
	\inf_{q\in \lbrack 0,\infty )}\sup_{u\in (0,\infty
		)}L_{xy}(u,q)=\lim_{r\rightarrow \infty }\inf_{q\in \lbrack 0,\infty
		)}\sup_{u\in \left[ r,\frac{1}{r}\right] }L_{xy}(u,q),
	\end{equation*}%
	then we are done, since by Corollary 3.3 in \cite{Sion1958}%
	\begin{eqnarray*}
		\inf_{q\in \lbrack 0,\infty )}\sup_{u\in (0,\infty )}L_{xy}(u,q)
		&=&\lim_{r\rightarrow \infty }\inf_{q\in \lbrack 0,\infty )}\sup_{u\in \left[
			r,\frac{1}{r}\right] }L_{xy}(u,q)= \\
		\lim_{r\rightarrow \infty }\sup_{u\in \left[ r,\frac{1}{r}\right]
		}\inf_{q\in \lbrack 0,\infty )}L_{xy}(u,q) &=&\sup_{u\in (0,\infty
			)}\inf_{q\in \lbrack 0,\infty )}L_{xy}(u,q).
	\end{eqnarray*}
	
	Let $M:= \inf_{q\in \lbrack 0,\infty )}\sup_{u\in (0,\infty
		)}L_{xy}(u,q) $. We will assume that $M<\infty $, and note that the case $%
	M=\infty $ is treated similarly. Since $\lim_{q\rightarrow \infty
	}L_{xy}(q,1)=\infty ,$ we can find $\tilde{q}$ such that $L_{xy}(q,1)>2M$
	for every $q\geq \tilde{q}.$ Now we have 
	\begin{equation*}
	\inf_{q\in \lbrack 0,\infty )}\sup_{u\in (0,\infty )}L_{xy}(u,q)=\inf_{q\in
		\lbrack 0,\tilde{q}]}\sup_{u\in \left( 0,\infty \right) }L_{xy}(u,q),
	\end{equation*}%
	and 
	\begin{equation*}
	\inf_{q\in \lbrack 0,\tilde{q}]}\sup_{u\in \left[ r,\frac{1}{r}\right]
	}L_{xy}(u,q)=\inf_{q\in \lbrack 0,\infty )}\sup_{u\in \left[ r,\frac{1}{r}%
		\right] }L_{xy}(u,q),
	\end{equation*}%
	which gives 
	\begin{eqnarray*}
		&&\inf_{q\in \lbrack 0,\infty )}\sup_{u\in (0,\infty
			)}L_{xy}(u,q)=\inf_{q\in \lbrack 0,\tilde{q}]}\sup_{u\in \left( 0,\infty
			\right) }L_{xy}(u,q)=\sup_{u\in \left( 0,\infty \right) }\inf_{q\in \lbrack
			0,\tilde{q}]}L_{xy}(u,q)= \\
		&&\lim_{r\rightarrow \infty }\sup_{u\in \left[ r,\frac{1}{r}\right]
		}\inf_{q\in \lbrack 0,\tilde{q}]}L_{xy}(u,q)=\lim_{r\rightarrow \infty
		}\inf_{q\in \lbrack 0,\tilde{q}]}\sup_{u\in \left[ r,\frac{1}{r}\right]
		}L_{xy}(u,q)=\lim_{r\rightarrow \infty }\inf_{q\in \lbrack 0,\infty
			)}\sup_{u\in \left[ r,\frac{1}{r}\right] }L_{xy}(u,q).
	\end{eqnarray*}
\end{proof} \\

\begin{proof}[Proof of Theorem \ref{minmax}]
	  Let $H^-$ (respectively, $H^+$) denote the left-hand side (respectively, right-hand side), of
		\eqref{hamiltonian}.
	Since each term in the sum that generates $H^{+}$ is bigger than the
	corresponding one in the sum of $H^{-},$ we get equality for all of them. By
	the theory of the Legendre transform we know that $\inf_{q\in \lbrack
		0,\infty )}\sup_{ u\in (0,\infty )}\left\{q\xi_{xy}+G_{xy}(u,q)\right\}$ is
	actually a concave function. Since we can exchange the order between
	the supremum and infimum, then $\sup_{u\in (0,\infty )}\inf_{q\in \lbrack 0,\infty
		)}\left\{q\xi_{xy}+G_{xy}(u,q)\right\}$ must be a concave function as well.
	By using the formula 
	\begin{equation*}
	\sup_{u\in (0,\infty )}\inf_{q\in \lbrack 0,\infty
		)}\left\{q\xi+G_{xy}(u,q)\right\} =\sum_{\left( x,y\right) \in \mathcal{Z}%
	}m_{x}\gamma _{xy}\left( C_{xy}\right) ^{\ast }\left( -\ell ^{\ast }\left(
	-\xi _{xy}\right) \right)
	\end{equation*}%
	we have that $\left( C_{xy}\right) ^{\ast }\left( -\ell ^{\ast }\left( \xi
	\right) \right) =\left( C_{xy}\right) ^{\ast }\left( 1-e^{\xi }\right)$ must
	also be concave. By differentiating with respect to $\xi $ we get, $e^{2\xi
	}\left( \left( C_{xy}\right) ^{\ast }\right) ^{\prime \prime }\left(
	1-e^{\xi }\right) -e^{\xi }\left( \left( C_{xy}\right) ^{\ast }\right)
	^{\prime }\left( 1-e^{\xi }\right) \leq 0,$ from which, by using the
	identity $(f^{\ast })^{\prime }=(f^{\prime })^{-1}$, we get 
	\begin{equation*}
	e^{2\xi }\left( \left( \left( C_{xy}\right) ^{\prime }\right) ^{-1}\right)
	^{\prime }\left( 1-e^{\xi }\right) -e^{\xi }\left( \left( C_{xy}\right)
	^{\prime }\right) ^{-1}\left( 1-e^{\xi }\right) \leq 0.
	\end{equation*}%
	By substituting $\tilde{u}=1-e^{\xi }$ we get 
	\begin{eqnarray*}
		&\left( 1-\tilde{u}\right) \left( \left( \left( C_{xy}\right) ^{\prime
		}\right) ^{-1}\right) ^{\prime }\left( \tilde{u}\right) -\left( \left(
		C_{xy}\right) ^{\prime }\right) ^{-1}\left( \tilde{u}\right) \leq 0,&\hspace{%
			4pt}\text{with}\hspace{4pt}\tilde{u}\leq 1 \\
		&\left( 1-\tilde{u}\right) \frac{1}{\left( C_{xy}\right) ^{\prime \prime
			}\left( \left( \left( C_{xy}\right) ^{\prime }\right) ^{-1}\left( \tilde{u}%
			\right) \right) }-\left( \left( C_{xy}\right) ^{\prime }\right) ^{-1}\left( 
		\tilde{u}\right) \leq 0,&\hspace{4pt}\text{with}\hspace{4pt}\tilde{u}\leq 1
		\\
		&\left( 1-\left( C_{xy}\right) ^{\prime }\left( r\right) \right) \frac{1}{%
			\left( C_{xy}\right) ^{\prime \prime }\left( r\right) }-r\leq 0,&\hspace{4pt}%
		\text{with}\hspace{4pt}\left( C_{xy}\right) ^{\prime }\left( r\right) \leq 1
		\\
		&r\left( C_{xy}\right) ^{\prime \prime }\left( r\right) +\left(
		C_{xy}\right) ^{\prime }\left( r\right) -1\geq 0,&\hspace{4pt}\text{with}%
		\hspace{4pt}\left( C_{xy}\right) ^{\prime }\left( r\right) \leq 1.
	\end{eqnarray*}
	
	Now the last inequality implies that either $\left( C_{xy}\right) ^{\prime
	}\left( u\right) \geq 1$ or that $u(C_{xy})^{\prime }\left( u\right) -u$ is
	locally increasing and even more that if $\left( C_{xy}\right) ^{\prime
	}\left( u_{0}\right) \geq 1$ for some $u_{0},$ then it must remain like that
	for every $u\geq u_{0}.$ If that was not the case then we can find $%
	u_{1}>u_{0}$ such that $u_{1}(C_{xy})^{\prime }\left( u_{1}\right) -u_{1}<%
	\hat{q}$ for some negative $\hat{q},$ while $u_{0}(C_{xy})^{\prime }\left(
	u_{0}\right) -u_{0}\geq0.$ By a suitable application of the mean value
	theorem we will get the existence of an $r$ that the last inequality fails.
	If we set $\tilde{u}_{xy}=\inf\{u: \left(C_{xy}\right)^{\prime}\left(
	u\right) \geq 1\},$ then the Assumption \ref{assumption} is recovered.
\end{proof}

\section{Properties of $F_{xy}$}\label{ap:b}

\begin{proof}[Proof of Lemma \protect\ref{lemma:fnproperties}]
(1) We have 
\begin{equation*}
F_{xy}(q)=\sup_{u\in (0,\infty )}\left\{ u\ell \left( \frac{q}{u}\right)
-\gamma _{xy}C_{xy}\left( \frac{u}{\gamma _{xy}}\right) \right\} \geq \gamma
_{xy}\ell \left( \frac{q}{\gamma _{xy}}\right) -\gamma _{xy}C_{xy}\left( 
\frac{\gamma _{xy}}{\gamma _{xy}}\right) \geq \gamma _{xy}\ell \left( \frac{q%
}{\gamma _{xy}}\right)\geq 0 .
\end{equation*}

(2) We have 
\begin{equation*}
\begin{split}
F_{xy}(\gamma _{xy})& =\sup_{u\in (0,\infty )}G_{xy}(u,\gamma
_{xy})=\sup_{u\in (0,\infty )}\left\{ u\ell \left( \frac{\gamma _{xy}}{u}%
\right) -\gamma _{xy}C_{xy}\left( \frac{u}{\gamma _{xy}}\right) \right\} \\
& =\sup_{u\in (0,\infty )}\left\{ \gamma _{xy}\log \gamma _{xy}-\gamma
_{xy}\log u-\gamma _{xy}+u-\gamma _{xy}C_{xy}\left( \frac{u}{\gamma _{xy}}%
\right) \right\} ,
\end{split}%
\end{equation*}%
and by applying part 2 of Lemma \ref{lemos} 
\begin{equation*}
\gamma _{xy}C_{xy}\left( \frac{u}{\gamma _{xy}}\right) \geq \gamma _{xy}\log
\gamma _{xy}-\gamma _{xy}\log u-\gamma _{xy}+u.
\end{equation*}%
Therefore $F_{xy}(\gamma _{xy})\leq 0.$ However, by part (1) of this lemma $%
F_{xy}(\gamma _{xy})\geq 0,$ and therefore the equality follows.

(3) By definition $F_{xy}(q)=\sup_{u\in (0,\infty )}G_{xy}(u,q).$ Let $a\in
(0,1)$ and $0\leq q_{1}<q_{2}<\infty $, and let $q=$ $aq_{1}+(1-a)q_{2}$.
Using the convexity of $G_{xy}(u,q)$ for fixed $u$ as a function of $q,$ we
have 
\begin{equation*}
\begin{split}
F_{xy}(aq_{1}+(1-a)q_{2})& =\sup_{u\in (0,\infty
)}G_{xy}(u,aq_{1}+(1-a)q_{2}) \\
& \leq \sup_{u\in (0,\infty )}\left\{
aG_{xy}(u,q_{1})+(1-a)G_{xy}(u,q_{2})\right\} \\
& \leq a\sup_{u\in (0,\infty )}G_{xy}(u,q_{1})+(1-a)\sup_{u\in (0,\infty
)}G_{xy}(u,q_{2}) \\
& \leq aF_{xy}(q_{1})+(1-a)F_{xy}(q_{2}).
\end{split}%
\end{equation*}
\end{proof}

For the proof of Lemma \ref{lemma:fnproperties2} that is given below, we
will use the following auxiliary lemma. Recall the definition of $G_{xy}$ in
(\ref{eqdef:Fxy,Gxy}).

\begin{lemma}
\label{lemma:updown} If $\{\boldsymbol{C}^{n}\}$ satisfies Assumption \ref%
{assumption:uh}, then the following hold for every $(x,y)\in \mathcal{Z}$.

\begin{enumerate}
\item There exists a positive real number $M,$ that does not depend on $%
(x,y) $, such that for the decreasing function $M_{xy}^{1}:(0,\infty
)\rightarrow \lbrack 0,\infty ),$ given by 
\begin{equation*}
M_{xy}^{1}(q)\doteq \min \left\{ \gamma _{xy}\left( \frac{\gamma _{xy}}{q}%
\right) ^{1/p},M\right\} ,
\end{equation*}%
we have that $G_{xy}(u,q)$ is increasing as a function of $u$ on the
interval $(0,M_{xy}^{1}(q)].$

\item There exists a decreasing function $M_{xy}^{2}:(0,\infty )\rightarrow
\lbrack 0,\infty ),$ with $M_{xy}^{2}(q)\geq M_{xy}^{1}(q),$ such that $%
G_{xy}(u,q)$ is decreasing as a function of $u$ on the interval $\left[
M_{xy}^{2}(q),\infty \right) .$
\end{enumerate}
\end{lemma}

\begin{proof}
By taking the derivative with respect to $u$ in the definition (\ref%
{eqdef:Fxy,Gxy}) we get 
\begin{equation*}
-\frac{q}{u}-(C_{xy})^{\prime }\left( \frac{u}{\gamma _{xy}}\right) +1.
\label{asta}
\end{equation*}%

(1) By part 2 of Assumption \ref{assumption:uh} there exists $M\in (0,\infty
)$ such that if $u<M$, then 
\begin{equation*}
-\frac{q}{u}-(C_{xy})^{\prime }\left( \frac{u}{\gamma _{xy}}\right) +1\geq -%
\frac{q}{u}+\left( \frac{\gamma _{xy}}{u}\right) ^{p+1}+1,
\end{equation*}%
and by taking $u\leq \gamma _{xy}\left( \gamma _{xy}/q\right) ^{1/p}$ we get 
\begin{equation*}
-\frac{q}{u}+\left( \frac{\gamma _{xy}}{u}\right) ^{p+1}+1\geq -\frac{q}{u}+%
\frac{q}{u}+1>0.
\end{equation*}%
Therefore for 
\begin{equation*}
M_{xy}^{1}(q)=\min \left\{ \gamma _{xy}\left( \frac{\gamma _{xy}}{q}\right)
^{1/p},M\right\} ,
\end{equation*}%
we have $-\frac{q}{u}-(C_{xy})^{\prime }\left( \frac{u}{\gamma _{xy}}\right)
+1\geq 0$ on the interval $(0,M_{xy}^{1}(q)].$

(2) By applying part 3 of Assumption \ref{assumption:uh}, we get that there
exists decreasing $\tilde{M}_{xy}^{2}(q)<\infty ,$ such that if $u>\tilde{M}%
_{xy}^{2}(q)$ then 
\begin{equation}
\frac{u}{\gamma _{xy}}(C_{xy})^{\prime }\left( \frac{u}{\gamma _{xy}}\right)
-\frac{u}{\gamma _{xy}}\geq -\frac{q}{\gamma _{xy}}.  \label{ineq:1}
\end{equation}%
Then $M_{xy}^{2}(q)\doteq \max \{M_{xy}^{1}(q),\tilde{M}_{xy}^{2}(q)\},$ is
decreasing and bigger than $M_{xy}^{1}$, and using (\ref{ineq:1}) we get%
\begin{equation*}
-\frac{q}{u}-(C_{xy})^{\prime }\left( \frac{u}{\gamma _{xy}}\right) +1=-%
\frac{q}{u}-\frac{\gamma _{xy}}{u}\left( \frac{u}{\gamma _{xy}}%
(C_{xy})^{\prime }\left( \frac{u}{\gamma _{xy}}\right) -\frac{u}{\gamma _{xy}%
}\right) \leq 0
\end{equation*}%
on the interval $[M_{xy}^{2}(q),\infty ).$
\end{proof}

\begin{proof}[Proof of Lemma \protect\ref{lemma:fnproperties2}]
(1) Let $\epsilon >0,$ and $q\geq \epsilon $. By Lemma \ref{lemma:updown},
we have that $G_{xy}\left( u,q\right) ,$ as a function of $u,$ is increasing
on the interval $(0,M_{xy}^{1}(q)]$. Therefore for all $u\in
(0,M_{xy}^{1}(q)]$ we have 
\begin{equation*}
\begin{split}
u\ell \left( \frac{q}{u}\right) -\gamma _{xy}C_{xy}\left( \frac{u}{\gamma
_{xy}}\right) \leq & M_{xy}^{1}(q)\ell \left( \frac{q}{M_{xy}^{1}(q)}\right)
-\gamma _{xy}C_{xy}\left( \frac{M_{xy}^{1}(q)}{\gamma _{xy}}\right) \leq
M_{xy}^{1}(q)\ell \left( \frac{q}{M_{xy}^{1}(q)}\right) \\
& \leq q\log \left( \frac{q}{M_{xy}^{1}(q)}\right) +M_{xy}^{1}(q)\leq q\log
\left( \frac{q}{M_{xy}^{1}(q)}\right) +M_{xy}^{1}(\epsilon ) \\
& \leq q\log \left( q\right) -q\log \left( M_{xy}^{1}(q)\right)
+M_{xy}^{1}(\epsilon ) \\
& \overset{M_{xy}^{1}(\epsilon )\leq M_{xy}^{2}(\epsilon )}{\leq }q\log
\left( q\right) -q\log \left( M_{xy}^{1}(q)\right) +M_{xy}^{2}(\epsilon ).
\end{split}%
\end{equation*}%
By the second part of Lemma \ref{lemma:updown}, we have that $G_{xy}(u,q) $
is decreasing on the interval $(M_{xy}^{2}(\epsilon ),\infty )$. Therefore
for all $u\in (M_{xy}^{2}(\epsilon ),\infty )$%
\begin{equation*}
\begin{split}
u\ell \left( \frac{q}{u}\right) -\gamma _{xy}C_{xy}\left( \frac{u}{\gamma
_{xy}}\right) & \leq M_{xy}^{2}(\epsilon )\ell \left( \frac{q}{%
M_{xy}^{2}(\epsilon )}\right) -\gamma _{xy}C_{xy}\left( \frac{%
M_{xy}^{2}(\epsilon )}{\gamma _{xy}}\right) \leq M_{xy}^{2}(\epsilon )\ell
\left( \frac{q}{M_{xy}^{2}(\epsilon )}\right) \\
& \leq q\log \left( \frac{q}{M_{xy}^{2}(\epsilon )}\right)
+M_{xy}^{2}(\epsilon ) \\
& \overset{M_{xy}^{2}(q)\leq M_{xy}^{2}(\epsilon )}{\leq }q\log \left(
q\right) -q\log \left( M_{xy}^{2}(q)\right) +M_{xy}^{2}(\epsilon ) \\
& \overset{M_{xy}^{1}(q)\leq M_{xy}^{2}(q)}{\leq }q\log \left( q\right)
-q\log \left( M_{xy}^{1}(q)\right) +M_{xy}^{2}(\epsilon ).
\end{split}%
.
\end{equation*}%
Finally for the interval $[M_{xy}^{1}(q),M_{xy}^{2}(\epsilon )]$ we have%
\begin{equation*}
\begin{split}
u\ell \left( \frac{q}{u}\right) -\gamma _{xy}C_{xy}\left( \frac{u}{\gamma
_{xy}}\right) & \leq u\ell \left( \frac{q}{u}\right) =q\log q-q\log u-q+u \\
& \leq q\log q-q\log (M_{xy}^{1}(q))+M_{xy}^{2}(\epsilon ).
\end{split}%
\end{equation*}

Now if we recall the definition of $M_{xy}^{1}$ given in Lemma \ref%
{lemma:updown} and set $\bar{M}(q)\doteq \max \{M_{xy}^{2}(q):(x,y)\in 
\mathcal{Z}\},$ then 
\begin{equation*}
G_{xy}(u,q)\leq q\log \frac{q}{\min \left\{ \gamma _{xy}\left( \frac{\gamma
_{xy}}{q}\right) ^{1/p},M\right\} }+\bar{M}(\epsilon ),
\end{equation*}%
and by taking supremum over $u$ we end up with $F_{xy}(q)$ satisfying the
same bound.

(2) This is straightforward since $F_{xy}$ is finite on the interval $%
(0,\infty ),$ and convex.
\end{proof}
\section{ Tightness functionals}

\label{appendix:tightness}

\begin{proof}[Proof of Lemma \protect\ref{lemma:tightnessupmu}]
Let $c_{2}>0$ and $\{(\boldsymbol{\mu }^{n},T^{n})\}$ be a deterministic
sequence in $S$ with $\boldsymbol{\mu }^{n}$ absolutely continuous such that 
\begin{equation*}
\int_{0}^{T^{\boldsymbol{n}}}\ell \left( |\dot{\boldsymbol{\mu }}%
^{n}(t)|\right) dt+c_{1}T^{\boldsymbol{n}}\leq c_{2}
\end{equation*}%
and $|\dot{\boldsymbol{\mu }}^{n}(t)|=0$ for $t>T^{\boldsymbol{n}}$. We need
to show that $H$ has level sets with compact closure. Since all elements are
positive, we have that $T^{\boldsymbol{n}}\leq c_{2}/c_{1}$. Let $%
\boldsymbol{\bar{\mu}}^{n}$ denote the restriction of $\boldsymbol{\mu }^{n}$
to $[0,c_{2}/c_{1}]$. If we prove that $\boldsymbol{\bar{\mu}}^{n}$
converges along some subsequence then we are done. Using the inequality $%
ab\leq e^{ca}+\ell (b)/c,$ which is valid for $a,b\geq 0,$ and $c\geq 1,$ we
have that%
\begin{equation*}
|\boldsymbol{\mu }^{n}(t)-\boldsymbol{\mu }^{n}(s)|\leq \int_{t}^{s}|\dot{%
\boldsymbol{\mu }}^{n}(r)|dr\leq (t-s)e^{c}+\frac{c_{2}}{c}.
\end{equation*}%
This shows that $\left\{ \boldsymbol{\bar{\mu}}^{n}\right\} $ are
equicontinuous. Since $\boldsymbol{\bar{\mu}}^{n}(t)$ takes values in the
compact set $\mathcal{P}(\mathcal{X})$, by the Arzela-Ascoli theorem there
is a convergent subsequence.
\end{proof}

\bibliographystyle{plain}

\end{document}